\documentclass[11pt,oneside,reqno]{amsart}
\usepackage{geometry}
\usepackage[utf8]{inputenc}
\usepackage{amstext,latexsym,amsbsy,amsmath,amssymb,amsthm,mathtools,relsize,geometry,enumerate}
\usepackage{hyperref}
\hypersetup{
  colorlinks   = true, 
  urlcolor     = blue, 
  linkcolor    = red, 
  citecolor   = red 
}

\usepackage{mathtools,enumerate,enumitem}

\usepackage{graphicx}
\usepackage{epstopdf}
\setkeys{Gin}{width=\linewidth,totalheight=\textheight,keepaspectratio}
\graphicspath{{./images/}}

\usepackage{multicol}
\usepackage{booktabs}
\usepackage{mathrsfs}
\usepackage{color}

\newcommand{\rbb}{\mathbb{R}}

\renewcommand{\L}{\mathcal{L}}

\newcommand{\W}{\mathcal{W}}

\newcommand{\X}{\mathbf{X}}
\newcommand{\Q}{\mathbf{Q}}
\newcommand{\B}{\mathcal{B}}
\newcommand{\la}{\langle}
\newcommand{\ra}{\rangle}

\newcommand{\ctilde}{\tilde{c}}
\newcommand{\Ctilde}{\tilde{C}}

\newcommand{\nt}{\notag}

\newcommand{\xbar}{\bar{x}}

\newcommand{\Ht}{\tilde{H}}
\newcommand{\Lt}{\tilde{\mathcal{L}}}

\newcommand{\xb}{\mathrm{x}}
\newcommand{\vb}{\mathrm{v}}
\newcommand{\zb}{\mathrm{z}}
\newcommand{\fb}{\mathrm{f}}
\newcommand{\qb}{\mathrm{q}}

\newcommand{\mi}{\wedge}
\newcommand{\D}{\mathcal{D}}

\renewcommand{\d}{\text{d}}

\newcommand{\f}{\varphi}

\newcommand{\grad}{\nabla}

\newcommand{\E}{\mathbb{E}}

\renewcommand{\P}{\mathbb{P}}

\newcommand{\xhat}{\hat{x}}
\newcommand{\vhat}{\hat{v}}
\newcommand{\zhat}{\hat{z}}

\newcommand{\A}{\mathcal{A}}
\newcommand{\M}{\mathcal{M}}

\newcommand{\U}{U}
\newcommand{\G}{G}

\newcommand{\close}{\!\!\!}

\theoremstyle{plain}
\newtheorem{theorem}{Theorem}[section]

\newtheorem{lemma}[theorem]{Lemma}
\newtheorem{assumption}[theorem]{Assumption}
\newtheorem{proposition}[theorem]{Proposition}

\theoremstyle{definition}
\newtheorem{definition}[theorem]{Definition}

\newtheorem{remark}[theorem]{Remark}

\numberwithin{equation}{section}

\title{Asymptotic analysis for the generalized Langevin equation with singular potentials}

\author{Manh Hong Duong$^1$ and Hung Dang~Nguyen$^2$}

\address{$^1$ School of Mathematics, University of Birmingham, Birmingham, UK}

\address{$^2$ Department of Mathematics, University of Tennessee, Knoxville, Tennessee, USA}

\begin{document}

\maketitle

\begin{abstract}
We consider a system of interacting particles governed by the generalized Langevin equation (GLE) in the presence of external confining potentials, singular repulsive forces, as well as memory kernels. Using a Mori-Zwanzig approach, we represent the system by a class of Markovian dynamics. Under a general set of conditions on the nonlinearities, we study the large-time asymptotics of the multi-particle Markovian GLEs. We show that the system is always exponentially attractive toward the unique invariant Gibbs probability measure. The proof relies on a novel construction of Lyapunov functions. We then establish the validity of the small mass approximation for the solutions by an appropriate equation on any finite-time window. Important examples of singular potentials in our results include the Lennard-Jones and Coulomb functions.
\end{abstract}

\section{Introduction} \label{sec:intro}
Given $N\ge 2$, we are interested in the following system of generalized Langevin equations in $\rbb^d$, $d\ge 1$:
\begin{align} \label{eqn:GLE:N-particle:integral}
\d\, x_i(t) &= v_i(t)\d t,\qquad i=1,\dots,N, \nt  \\
m\,\d\, v_i(t) & = -\gamma v_i(t) \d t -\grad \U(x_i(t))\d t- \sum_{j\neq i}\grad \G\big(x_i(t)-x_j(t)\big) \d t +\sqrt{2\gamma} \,\d W_{i,0}(t)\nt \\
&\qquad   +\int_0^t K_i(t-s)v_i(s)\d s\d t+F_i(t)\d t.
\end{align}
System \eqref{eqn:GLE:N-particle:integral} was introduced to describe the evolution of $N$ interacting micro-particles in a thermally fluctuating viscoelastic medium, see \cite{baczewski2013numerical,
bockius2021model,
duong2015long,duong2019mean,
gomes2020mean,gottwald2015parametrizing,
jung2018generalized,
ness2015applications,
wei2016chain} and references therein. The bivariate process $(x_i(t),v_i(t))$, $i=1,\dots,N$ represents the position and velocity of the $i^{\text{th}}$ particle. On the $v_i$-equation in \eqref{eqn:GLE:N-particle:integral}, $m>0$ is the particle's mass, $\gamma>0$ governs the viscous friction, 
$\{W_{i,0}\}_{i=1,\ldots, N}$ are independent standard $d$-dimensional Wiener processes, $U:\rbb^d \to[0,\infty)$ represents a confining potential satisfying polynomial growth and certain dissipative conditions and $G:\rbb^d\setminus\{0\}\to\rbb$ is a singular repulsive potential. Furthermore, the $i^{\text{th}}$ particle is subjected to a convolution term involving the convolution kernel $K_i:[0,\infty)\to[0,\infty)$ that characterizes the delayed response of the fluid to the particle's past movement \cite{kneller2011generalized,mason1995optical,
mori1965continued}. In accordance with the \emph{fluctuation-dissipation} relationship \cite{kubo1966fluctuation,
zwanzig2001nonequilibrium}, the random force $F_i(t)$ is a mean-zero, stationary Gaussian process linked to $K_i$ via the relation 
\begin{align} \label{eqn:fluctuation-dissipation}
\E[F_i(t)F_i(s)]=K_i(|t-s|).
\end{align}

In the absence of memory effects, that is setting $K_i\equiv 0$ and $F_i\equiv 0$, \eqref{eqn:GLE:N-particle:integral} is reduced to the classical underdamped Langevin system modeling Brownian particles driven by repulsive external forces 
\begin{align} \label{eqn:Langevin:N-particle}
\d\, x_i(t) &= v_i(t)\d t,\quad i=1,\dots,N, \nt  \\
m\,\d\, v_i(t) & = -\gamma v_i(t) \d t -\grad \U(x_i(t))\d t -\sum_{j\neq i} \grad \G(x_i(t)-x_j(t)) \d t+\sqrt{2\gamma} \,\d W_{i,0}(t).
\end{align}
In particular, the large-time asymptotic of \eqref{eqn:Langevin:N-particle} is well-understood. That is under a wide class of polynomial potential $U$ and singular potential $G$, including the instances of Lennard-Jones and Coulomb functions, system \eqref{eqn:Langevin:N-particle} admits a unique invariant probability measure which is exponentially attractive and whose formula is given by \cite{conrad2010construction, cooke2017geometric, 
grothaus2015hypocoercivity,
herzog2017ergodicity, lu2019geometric}
\begin{align*}
\pi(\d\xb,\d\vb)=\frac{1}{Z}\exp\Big\{-\Big(\frac{1}{2}\sum_{i=1}^N m|v_i|^2+\sum_{i=1}^N U(x_i)+\close\sum_{1\le i<j\le N}\close G(x_i-x_j)\Big) \Big\}\d \xb\d \vb.
\end{align*}
In the above, $Z$ is the normalization constant, $\xb=(x_1,\dots,x_N)$ and $\vb=(v_1,\dots,v_N)$. However, as pointed out elsewhere in \cite{kubo1966fluctuation,zwanzig2001nonequilibrium}, the presence of elasticity in a viscoelastic medium induces a memory effect between the motion of the particles and the surrounding molecular bombardment. It is thus more physically relevant to consider \eqref{eqn:GLE:N-particle:integral}. On the other hand, in the absence of singularities in \eqref{eqn:GLE:N-particle:integral} ($G\equiv 0$), there is a vast literature in the context of, e.g., large-time behaviors \cite{glatt2020generalized,
herzog2023gibbsian,ottobre2011asymptotic,
pavliotis2014stochastic,pavliotis2021scaling} as well as small mass limits \cite{herzog2016small,hottovy2015smoluchowski,
lim2019homogenization,lim2020homogenization,
nguyen2018small,shi2021small}. In contrast, much less is known about the system \eqref{eqn:GLE:N-particle:integral} in the presence of both memory kernels and singular potentials for any of those limiting regimes.

The main goal of the present article is thus two-fold. Firstly, under a general set of conditions on the nonlinearities and memory kernels, we asymptotically characterize the equilibrium of \eqref{eqn:GLE:N-particle:integral} when $t\to\infty$. More specifically, we aim to prove that under these practical assumptions, \eqref{eqn:GLE:N-particle:integral} is exponentially attractive toward a unique ergodic probability measure. Secondly, we explore the behaviors of \eqref{eqn:GLE:N-particle:integral} in the small mass regime, i.e., by taking $m$ to zero on the right-hand side of the $v_i$-equation in \eqref{eqn:GLE:N-particle:integral}. Due to the singular limit when $m$ is small, the velocity $\vb(t)$ is oscillating fast whereas the position $\xb(t)$ is still moving slow. Hence, we seek to identify a limiting process $\qb(t)$ such that $\xb(t)$ can be related to $\qb(t)$ on any finite-time window. We now provide a more detailed description of the main results.

\subsection{Geometric ergodicity} \label{sec:intro:ergodicity}
In general, there is no Markovian dynamics associated with \eqref{eqn:GLE:N-particle:integral}, owing to the presence of the memory kernels. Nevertheless, it is well-known that for kernels which are a sum of exponential functions, we may adopt the Mori-Zwanzig approach to produce a Markovian approximation to \eqref{eqn:GLE:N-particle:integral} \cite{glatt2020generalized,kubo1966fluctuation,
mori1965continued,ottobre2011asymptotic,
pavliotis2014stochastic,
zwanzig2001nonequilibrium}. More specifically, when $K_i$ is given by
\begin{align} \label{form:K_i}
K_i(t)=\sum_{\ell=1}^{k_i} \lambda_{i,\ell}^2 e^{-\alpha_{i,\ell} t},\quad t\ge 0,
\end{align}
for some positive constants $\lambda_{i,\ell}$, $\alpha_{i,\ell}$, $\ell=1,\dots,k_i$,  
following the framework of \cite{baczewski2013numerical,doob1942brownian,duong2022accurate,
goychuk2012viscoelastic,
ottobre2011asymptotic,
pavliotis2014stochastic}, we can rewrite \eqref{eqn:GLE:N-particle:integral} as the following system
\begin{align} \label{eqn:GLE:N-particle}
\d\, x_i(t) &= v_i(t)\d t,\qquad i=1,\dots,N, \nt  \\
m\,\d\, v_i(t) & = -\gamma v_i(t) \d t -\grad \U(x_i(t))\d t +\sqrt{2\gamma} \,\d W_{i,0}(t)\\
&\qquad  - \sum_{j\neq i}\grad \G\big(x_i(t)-x_j(t)\big) \d t +\sum _{\ell=1}^{k_i} \lambda_{i,\ell}  z_{i,\ell}(t)\d t, \nt \\ 
\d\, z_{i,\ell}(t) &= -\alpha_{i,\ell} z_{i,\ell} (t)\d t-\lambda_{i,\ell}  v_i(t)\d t+\sqrt{2\alpha_{i,\ell} }\,\d W_{i,\ell} (t),\quad \ell=1,\dots,k_i.\nt 
\end{align}
See \cite[Proposition 8.1]{pavliotis2014stochastic} for a detailed discussion of this formulation.

Denoting $\zb_i=(z_{i,1},\dots,z_{i,k_i})\in (\mathbb{R}^d)^{k_i}, i=1,\ldots, N
$, we introduce the Hamiltonian function $H_N$ defined as
\begin{align} \label{form:H_N}
H_N(\xb,\vb,\zb_1,\dots,\zb_N) = \frac{1}{2}m|\vb|^2+ \sum_{i=1}^N U(x_i)+\close\sum_{1\le i< j\le N}\close\G(x_i-x_j)+\frac{1}{2}\sum_{i=1}^N |\zb_i|^2.
\end{align}
The corresponding Gibbs measure is given by
\begin{align} \label{form:pi_N}
\pi_N(\xb,\vb,\zb_1,\dots,\zb_N)=\frac{1}{Z_N}\exp\big\{-H_N(\xb,\vb,\zb_1,\dots,\zb_N)\big\}\d \xb\d\vb\d\zb_1\dots\d\zb_N,
\end{align}
where $Z_N$ is the normalization constant. Under suitable assumptions on the potentials, cf. Assumption \ref{cond:U} and Assumption \ref{cond:G} below, one can rely on the Hamiltonian structure to show that \eqref{eqn:GLE:N-particle} is always well-posed. That is a strong solution of \eqref{eqn:GLE:N-particle} exists and is unique for all finite time. Furthermore, by a routine computation \cite[Proposition 8.2]{pavliotis2014stochastic}, it is not difficult to see that the Gibbs measure $\pi_N$ as in \eqref{form:pi_N} is an invariant probability measure of \eqref{eqn:GLE:N-particle} and that there is
dissipation towards such a measure. On the other hand, as observed elsewhere in \cite{conrad2010construction, cooke2017geometric, 
grothaus2015hypocoercivity,
herzog2017ergodicity, lu2019geometric}, in the presence of the nonlinearities, the Hamiltonian \eqref{form:H_N} does not produce an energy estimate of the form
\begin{align*}
\frac{\d}{\d t}\E [V(t)]\le -c\, \E [V(t)]+C,\quad t\ge 0,
\end{align*}
which is needed to obtain geometric ergodicity. In this paper, we tackle the problem by exploiting the technique of Lyapunov function, cf. Definition \ref{def:Lyapunov}, and successfully establish the uniqueness of $\pi_N$ as well as an exponential convergent rate toward $\pi_N$. We note that our result covers important examples of singular potentials such as Lennard-Jones functions and Coulomb functions. We refer the reader to Theorem \ref{thm:ergodicity:N-particle} for a precise statement of this result and to Section \ref{sec:ergodicity} for its proof.

Historically, in the absence of repulsive forces, \eqref{eqn:GLE:N-particle} is reduced to the following   single-particle GLE 
\begin{align} \label{eqn:GLE:no-G}
\d\, x(t) &= v(t)\d t, \nt  \\
m\,\d\, v(t) & = -\gamma v(t) \d t -\grad \U(x(t))\d t +\sum _{i=1}^k \lambda_i z_i(t)\d t +\sqrt{2\gamma} \,\d W_0(t),\\ 
\d\, z_i(t) &= -\alpha_i z_i(t)\d t-\lambda_i v(t)\d t+\sqrt{2\alpha_i}\,\d W_i(t),\quad i=1,\dots,k,\nt 
\end{align}
whose large-time asymptotic has been studied extensively \cite{glatt2020generalized,herzog2023gibbsian,
ottobre2011asymptotic,pavliotis2014stochastic}. Particularly, mixing rates for the kernel instances of finite sum of exponentials were established in \cite{ottobre2011asymptotic,pavliotis2014stochastic} via the weak Harris Theorem \cite{hairer2011yet,meyn2012markov}. Analogously, kernel instances as an infinite sum of exponentials were explored in \cite{glatt2020generalized}. In this work, the uniqueness of invariant probability measures was obtained by employing the so-called asymptotic coupling argument. Recently in \cite{herzog2023gibbsian}, a Gibbsian approach was adopted from \cite{ weinan2002gibbsian,weinan2001gibbsian,mattingly2002exponential}  to study more general kernels that need not be a sum of exponentials.  

Turning back to our ergodicity results for singular potentials, the proof of Theorem \ref{thm:ergodicity:N-particle} below relies on two ingredients: an irreducibility condition and a suitable Lyapunov function \cite{ hairer2011yet,mattingly2002ergodicity,
meyn2012markov}. Whereas the irreducibility is relatively standard and can be addressed by adapting to the argument in e.g., \cite[Corollary 5.12]{herzog2017ergodicity} and \cite[Proposition 2.5]{mattingly2002ergodicity}, the construction of Lyapunov functions is highly nontrivial requiring a deeper understanding of the dynamics. Notably, in case the viscous drag is positive ($\gamma>0$), we draw upon the ideas developed in \cite{herzog2017ergodicity,lu2019geometric} tailored to our settings, cf. Lemma \ref{lem:Lyapunov:V_N^1}. On the other hand, when $\gamma=0$, it is worth to point out that the method therein is unfortunately not applicable owing to a lack of dissipation in $\vb$. To circumvent the issue, we effectively build up a novel Lyapunov function specifically designed for \eqref{eqn:GLE:N-particle}, cf. Lemma \ref{lem:Lyapunov:V_N^2}. The main idea of the construction is to realize the dominating effects at high energy ($H_N\gg 1$). We do so by employing a heuristic asymptotic scaling allowing for determining the leading order terms in ``bad" regions, i.e., when $|\xb|\to 0$ and $|\vb|,|\zb|\to\infty$. In turn, this will be crucial in the derivation of our Lyapunov functions that ultimately will be invoked to conclude ergodicity when $\gamma=0$. The heuristic argument will be presented out in Section \ref{sec:ergodicity:single-particle:gamma>0}, whereas the proofs of Lemma \ref{lem:Lyapunov:V_N^1} and Lemma \ref{lem:Lyapunov:V_N^2} are supplied in Section \ref{sec:ergodicity:N-particle}.

\subsection{Small mass limit} \label{sec:intro:small-mass}

In the second main topic of the paper, we investigate the small mass limit for the process $\xb(t)=\xb_m(t)$ in \eqref{eqn:GLE:N-particle}. Namely, by taking $m$ to zero on the right-hand side of the $v_i$-equation in \eqref{eqn:GLE:N-particle}, we aim to derive a process $\qb(t)$ taking values in $(\rbb^d)^N$ such that $\xb_m(t)$ can be well-approximated by $\qb(t)$ on any finite time window, i.e.,
\begin{align} \label{lim:small-mass:heuristic}
\sup_{t\in[0,T]}|\xb_m(t)-\qb(t)|\to 0,\quad \text{as  }m\to 0,
\end{align}
where the limit holds in an appropriate sense. In the literature, such a statement is also known as the Smoluchowski-Kramer approximation \cite{freidlin2004some, kramers1940brownian,smoluchowski1916drei}.

We note that in the absence of the singularities, there is a vast literature on limits of the form \eqref{lim:small-mass:heuristic} for various settings of the single-particle GLE as well as other second order systems. For examples, numerical simulations were performed in \cite{hottovy2012noise,hottovy2012thermophoresis}. Rigorous results in this direction for state-dependent drift terms appear in the work of  \cite{cerrai2020averaging,hottovy2015smoluchowski,hottovy2012noise,
pardoux2003poisson}. Similar results are also established for Langevin dynamics \cite{herzog2016small,duong2017variational}, finite-dimensional single-particle GLE \cite{lim2019homogenization,lim2020homogenization}, as well as infinite-dimensional single-particle GLE \cite{nguyen2018small,shi2021small}. Analogous study for the stochastic wave equation was central in the work of \cite{cerrai2006smoluchowski,cerrai2006smoluchowski2,
cerrai2017smoluchowski,cerrai2020convergence,
cerrai2014smoluchowski,cerrai2016smoluchowski,
nguyen2022small}. On the other hand, small mass limits in the context of repulsive forces are much less studied, but see the recent paper \cite{choi2022quantified} for a quantification of the small-mass limit of the kinetic Vlasov–Fokker–Planck equations with singularities. In our work, we investigate this problem under a general set of conditions on the nonlinear potentials. More specifically, let us introduce the following system
\begin{align} \label{eqn:GLE:N-particle:m=0}
\gamma \d q_i(t) & =  -\grad U(q_i(t))\d t- \sum_{j\neq i}\grad \G\big(q_i(t)-q_j(t)\big) \d t+\sqrt{2\gamma}\d W_{i,0}(t) \nt \\
&\qquad\qquad -\sum_{\ell=1}^{k_i} \lambda_{i,\ell}^2 q_i(t)\d t+\sum_{\ell=1}^{k_i}\lambda_{i,\ell} f_{i,\ell}(t)\d t, \qquad i=1,\dots,N,  \\
\d f_{i,\ell}(t) &= -\alpha_{i,\ell} f_{i,\ell}(t)+\lambda_{i,\ell}\, \alpha_{i,\ell}\, q_i(t)\d t+\sqrt{2\alpha_{i,\ell}}\d W_{i,\ell}(t),\quad \ell=1,\dots, k_i.\nt
\end{align}
Our second main result states that the process $\qb(t)$ satisfies the following limit in probability for all $\xi,\,T>0$
\begin{align} \label{lim:small-mass:|x_m-q|:heuristic}
\P\Big(\sup_{t\in[0,T]}|\xb_m(t)-\qb(t)|>\xi\Big)\to 0,\quad \text{as  }m\to 0.
\end{align}
The precise statement of \eqref{lim:small-mass:|x_m-q|:heuristic} is provided in Theorem \ref{thm:small-mass:N-particle} while its detailed proof is supplied in Section \ref{sec:small-mass}.

In order to derive the limiting system \eqref{eqn:GLE:N-particle:m=0}, we will adopt the framework developed in \cite{herzog2016small,nguyen2018small} dealing with the same issue in the absence of singular potentials. This involves exploiting the structure of the $z_i$-equation in \eqref{eqn:GLE:N-particle} while making use of Duhamel's formula and an integration by parts. In turn, this allows for completely decoupling $z_i$ from $v_i$, ultimately arriving at \eqref{eqn:GLE:N-particle:m=0}. See Section \ref{sec:small-mass:heuristic} for a further discussion of this point. The proof of \eqref{lim:small-mass:|x_m-q|:heuristic} draws upon the argument in \cite{herzog2016small, nguyen2018small,ottobre2011asymptotic} tailored to our settings. Namely, we first reduce the general problem to the special case when the nonlinearities are assumed to be Lipschitz. This requires a careful analysis on the auxiliary memory variables $\zb_i(t)$, $i=1,\dots,N$, as well as the velocity process $\vb_m(t)$. We then proceed to remove the Lipschitz constraint by making use of crucial moment bounds on the limiting process $\qb(t)$. In turn, this relies on a delicate estimate on \eqref{eqn:GLE:N-particle:m=0} via suitable Lyapunov functions, which are also of independent interest. Similarly to the ergodicity results, we note that the limit \eqref{lim:small-mass:|x_m-q|:heuristic} is applicable to a wide range of repulsive forces, e.g., the Lennard-Jones and Coulomb functions. To the best of the authors' knowledge, the type of limit \eqref{lim:small-mass:|x_m-q|:heuristic} that we establish in this work seems to be the first in this direction for stochastic systems with memory and singularities. The explicit argument for \eqref{lim:small-mass:|x_m-q|:heuristic} will be carried out while making use of a series of auxiliary results in Section~\ref{sec:small-mass}. The proof of Theorem \ref{thm:small-mass:N-particle} will be presented in Section \ref{sec:small-mass:proof-of-theorem}.

Finally, we remark that a crucial property that is leveraged in this work is the choice of memory kernels as a finite sum of exponentials, cf. \eqref{form:K_i}. This allows for the convenience of employing Markovian framework to study asymptotic analysis for \eqref{eqn:GLE:N-particle:integral} \cite{mori1965continued, ottobre2011asymptotic,
pavliotis2014stochastic,
zwanzig2001nonequilibrium} under the impact of interacting repulsive forces. For SDEs with non-exponentially decaying kernels, e.g., sub-exponential or power-law, with smooth nonlinearities though, we refer the reader to \cite{baeumer2015existence,desch2011p,
glatt2020generalized,herzog2023gibbsian,
nguyen2018small}.

\subsection{Organization of the paper}\label{sec:intro:organization}
The rest of the paper is organized as follows: in Section \ref{sec:result}, we introduce the notations as well as the assumptions that we make on the nonlinearities. We also state the main results of the paper, including Theorem \ref{thm:ergodicity:N-particle} on geometric ergodicity of \eqref{eqn:GLE:N-particle} and Theorem \ref{thm:small-mass:N-particle} on the validity of the approximation of \eqref{eqn:GLE:N-particle} by \eqref{eqn:GLE:N-particle:m=0} in the small mass regime. In Section \ref{sec:ergodicity}, we address the construction of Lyapunov functions for \eqref{eqn:GLE:N-particle} and prove Theorem \ref{thm:ergodicity:N-particle}. In Section \ref{sec:small-mass}, we detail a series of auxiliary results that we employ to prove the convergence of \eqref{eqn:GLE:N-particle} toward \eqref{eqn:GLE:N-particle:m=0}. We also conclude Theorem \ref{thm:small-mass:N-particle} in this section. In Appendix \ref{sec:appendix}, we provide useful estimates on singular potentials that are exploited to establish the main results.

\section{Assumptions and main results}
 \label{sec:result}

Throughout, we let $(\Omega, \mathcal{F}, (\mathcal{F}_t)_{t\geq 0},  \P)$ be a filtered probability space satisfying the usual conditions \cite{karatzas2012brownian} and $(W_{i,j}(t))$, $i=1,\dots, N$, $j=0,\dots,k_i$, be i.i.d standard $d$-dimensional Brownian Motions on $(\Omega, \mathcal{F},\P)$ adapted to the filtration $(\mathcal{F}_t)_{t\geq 0}$.  

In Section \ref{sec:result:assumption}, we detail sufficient conditions on the nonlinearities $U$ and $G$ that we will employ throughout the analysis. We also formulate the well-posedness through Proposition \ref{prop:GLE:N-particle:well-posed}. In Section \ref{sec:result:ergodicity}, we state the first main result through Theorem~\ref{thm:ergodicity:N-particle} giving the uniqueness of the invariant Gibbs measure $\pi_N$ defined in \eqref{form:pi_N}, as well as the exponential convergent rate toward $\pi_N$ in suitable Wasserstein distances. In Section \ref{sec:result:small-mass}, we provide our second main result through Theorem \ref{thm:small-mass:N-particle} concerning the validity of the small mass limit of \eqref{eqn:GLE:N-particle} on any finite time window.

\subsection{Main assumptions} \label{sec:result:assumption} For notational convenience, we denote the inner product and the norm in $\rbb^d$ by $\la\cdot,\cdot\ra$ and $|\cdot|$, respectively. Concerning the potential $U$, we will impose the following condition \cite{ottobre2011asymptotic,
pavliotis2014stochastic}.

\begin{assumption} \label{cond:U} \textup{(i)} $U\in C^\infty(\rbb^d;[1,\infty))$ satisfies
\begin{equation} \label{cond:U:U'(x)=O(x^lambda)} 
|U(x)|\le a_1(1+|x|^{\lambda+1}),\quad  |\grad \U(x)|\le a_1(1+|x|^{\lambda}),\quad x\in\rbb^d,
\end{equation}
and
\begin{align}\label{cond:U:U''(x)=O(x^lambda-1)} 
 |\grad^2 \U(x)|\le a_1(1+|x|^{\lambda-1}),
\end{align}
for some constants $a_1>0$ and $\lambda\ge 1$.

\textup{(ii)} Furthermore, there exist positive constants $a_2,a_3$ such that
\begin{equation} \label{cond:U:x.U'(x)>-x^(lambda+1)}
\la \grad\U(x), x\ra \ge a_2|x|^{\lambda+1}-a_3,\quad x\in\rbb^d.
\end{equation}

\textup{(iii)} If $\gamma=0$, then $\lambda=1 $.

%
%
%
%
%
\end{assumption}

\begin{remark} \label{rem:U} The first two conditions (i) and (ii) in Assumption \ref{cond:U} are quite popular and can be found in many previous literature for the Langevin dynamics \cite{glatt2020generalized, mattingly2002ergodicity,nguyen2018small}. In particular, $U(x)$ essentially behaves like $|x|^{\lambda+1}$ at infinity. On the other hand, for the generalized Langevin counterpart, in the absence of the viscous drag, i.e., $\gamma=0$ in $v-$equation in \eqref{eqn:GLE:N-particle}, we have to impose condition (iii) \cite{ottobre2011asymptotic} requiring $U$ be essentially a quadratic potential. We note however that this condition is unnecessary for the well-posedness. Rather, it is to guarantee the existence of suitable Lyapunov functions, so as to ensure geometric ergodicity of \eqref{eqn:GLE:N-particle}. See the proofs of Lemma \ref{lem:Lyapunov:V_2} and Lemma \ref{lem:Lyapunov:V_N^2} for a further discussion of this point. 
\end{remark}

Concerning the singular potential $G$, we will make the following condition \cite{bolley2018dynamics,herzog2017ergodicity,
lu2019geometric}.

\begin{assumption} \label{cond:G} \textup{(i)} $G\in C^\infty(\rbb^d\setminus\{0\};\rbb)$ satisfies $\G(x)\to \infty$ as $|x|\to 0$. Furthermore, there exists a positive constant $\beta_1\ge 1$ such that for all $x\in\rbb^d\setminus\{0\}$
\begin{align} 
|\G(x)|&\le a_1\Big(1+|x|+\frac{1}{|x|^{\beta_1}}\Big) ,\label{cond:G:G<1/|x|^beta}\\
|\grad \G(x)| &\le a_1\Big(1+\frac{1}{|x|^{\beta_1}}\Big),\label{cond:G:grad.G(x)<1/|x|^beta}\\
\text{and} \quad |\grad^2 \G(x)|& \le a_1\Big(1+\frac{1}{|x|^{\beta_1+1}}\Big),\label{cond:G:grad^2.G(x)<1/|x|^beta}
\end{align}
where $a_1$ is the constant as in Assumption~\ref{cond:U}.

\textup{(ii)} There exist positive constants $\beta_2\in[0,\beta_1)$, $a_4,\,a_5$ and $a_6$ such that 
\begin{equation} \label{cond:G:|grad.G(x)+q/|x|^beta_1|<1/|x|^beta_2}
\Big|\grad G(x) +a_4\frac{x}{|x|^{\beta_1+1}}\Big| \le \frac{a_5}{|x|^{\beta_2}}+a_6, \quad x\in\rbb^d\setminus\{0\}.
\end{equation}

\end{assumption}

\begin{remark} \label{rem:G} (i) Although \eqref{cond:G:|grad.G(x)+q/|x|^beta_1|<1/|x|^beta_2} is slightly odd looking, this condition essentially states that $\grad G$ can be expressed as
\begin{align*}
\grad G(x)= -c\frac{x}{|x|^{\beta_1+1}}+\text{lower order terms}.
\end{align*}
hence the requirement $\beta_2<\beta_1$. Also, $\grad G$ has the order of $|x|^{-\beta_1}$ as $|x|\to 0$, i.e.,
\begin{align*}
\frac{c}{|x|^{\beta_1}}-C\le |\grad G(x)|\le \frac{C}{|x|^{\beta_1}}+C,
\end{align*}
which need not imply \eqref{cond:G:|grad.G(x)+q/|x|^beta_1|<1/|x|^beta_2}. Condition \eqref{cond:G:|grad.G(x)+q/|x|^beta_1|<1/|x|^beta_2} will be employed in the Lyapunov proofs in Section \ref{sec:ergodicity}.

(ii) A routine calculation shows that both the Coulomb functions
\begin{align} \label{form:Coulomb}
G(x)=\begin{cases}
-\log|x|,& d=2,\\
\frac{1}{|x|^{d-2}},& d\ge 3,
\end{cases}
\end{align}
 and the Lennard-Jones functions 
\begin{align} \label{form:Lennard-Jones}
G(x) = \frac{c_0}{|x|^{12}}-\frac{c_1}{|x|^6},
\end{align} 
satisfy \eqref{cond:G:|grad.G(x)+q/|x|^beta_1|<1/|x|^beta_2}. Particularly, in case of \eqref{form:Lennard-Jones}, we have
\begin{align*}
\grad G(x) = -12c_0\frac{x}{|x|^{14}}+6c_1\frac{x}{|x|^{8}}, 
\end{align*}
which verifies \eqref{cond:G:|grad.G(x)+q/|x|^beta_1|<1/|x|^beta_2} with $\beta_1=13$ and $\beta_2=7$. Another well-known example for the case $\beta_1=1$ is the log function $G(x)=-\log|x|$ whereas the case $\beta_1>1$ includes the instance $G(x)=|x|^{-\beta_1+1}$.

(iiI) Without loss of generality, we may assume that
\begin{align} \label{cond:U+G>0}
\sum_{i=1}^N U(x_i)+\close\sum_{1\le i<j\le N}\close G(x_i-x_j)\ge 0.
\end{align}
Otherwise, we may replace $U(x)$ by $U(x)+c$ for some sufficiently large constant $c$, which does not affect \eqref{eqn:GLE:N-particle}.

\end{remark}

Under the above two assumptions, we are able to establish the geometric ergodicity of \eqref{eqn:GLE:N-particle}, cf. Theorem \ref{thm:ergodicity:N-particle} below, in any dimension $d\ge 1$. They are also sufficient for the purpose of investigating the small mass limits when either $d\ge 2$ or $\beta_1>1$. On the other hand, when $d=1$ and $\beta_1=1$ (e.g., log potentials), we will impose the following extra condition on $G$.

\begin{assumption} \label{cond:G:d=1}
Let $G$ and $\beta_1$ be as in Assumption \ref{cond:G}. In dimension $d=1$, if $\beta_1=1$, then 
\begin{align} \label{cond:G:d=1:a_4}
a_4> \frac{1}{2},
\end{align}
where $a_4$ is the positive constant from condition \eqref{cond:G:|grad.G(x)+q/|x|^beta_1|<1/|x|^beta_2}.
\end{assumption}

\begin{remark} \label{rem:G:small-mass:d=1:beta_1=1} As it turns out, in dimension $d=1$, log potentials ($\beta_1=1$) induce further difficulty for the small mass limit. To circumvent the issue, we have to impose Assumption \ref{cond:G:d=1}, which requires that the repulsive force be strong enough so as to establish suitable energy estimates. See the proof of Lemma \ref{lem:GLE:N-particle:m=0:supnorm} for a detailed explanation of this point. 
 It is also worth to mention that the threshold $a_4>1/2$ is a manifestation of Lemma \ref{lem:<|x_i-x_j|^s,|x_i-x_ell|^s>} and is perhaps far from optimality. Thus, the small mass limit in the case $d=1$ and $\beta_1=1$, cf. Theorem \ref{thm:small-mass:N-particle} below, remains an open problem for all $a_4\in(0,1/2]$ .

\end{remark}

Having introduced sufficient conditions on the potentials, we turn to the issue of well-posedness for \eqref{eqn:GLE:N-particle}. The domain where the displacement process $\xb_m(t)$ evolves on is denoted by $\D$ and is defined as  \cite{bolley2018dynamics,herzog2017ergodicity,
lu2019geometric}
\begin{align} \label{form:D}
\D=\{\xb=(x_1,\dots,x_N) \in (\rbb^d)^N: x_i\neq x_j \text{ if } i\neq j\}.
\end{align}
Then, we define the phase space for the solution of \eqref{eqn:GLE:N-particle} as follow:
\begin{align} \label{form:X}
\X=\D\times (\rbb^d)^N \times \prod_{i=1}^N (\rbb^d)^{k_i}.
\end{align}

The first result of this paper is the following well-posedness result ensuring the existence and uniqueness
of strong solutions to system \eqref{eqn:GLE:N-particle}.

\begin{proposition} \label{prop:GLE:N-particle:well-posed}  Under Assumption \ref{cond:U} and Assumption \ref{cond:G}, for every initial condition $X_0=(\xb(0),\vb(0)$, $\zb_{1}(0)$,$\dots$,$\zb_N(0))\in \X$, system \eqref{eqn:GLE:N-particle} admits a unique strong solution $X_m(t;X_0)=\big(\xb_m(t)$, $\vb_m(t)$, $\zb_{1,m}(t)$, ..., $\zb_{N,m}(t)\big)\in\X$.
\end{proposition}

The proof of Proposition \ref{prop:GLE:N-particle:well-posed} is a consequence of the existence of Lyapunov functions below in Lemma \ref{lem:Lyapunov:V_N^1} and Lemma \ref{lem:Lyapunov:V_N^2}. The argument can be adapted from the well-posedness proof of \cite[Section 3]{glatt2020generalized} tailored to our settings. Alternatively, the result can also be proved by using the Hamiltonian as in \eqref{form:H_N} to establish suitable moment bounds. It however should be noted that Proposition \ref{prop:GLE:N-particle:well-posed} itself is highly nontrivial as well-posedness does not automatically follow a standard argument for SDEs \cite{khasminskii2011stochastic} based on locally Lipschitz continuity, owing to the presence of singular potentials. Nevertheless, the issue can be tackled by constructing appropriate Lyapunov functions.

As a consequence of the well-posedness, we can thus introduce the Markov transition probabilities of the solution $X_m(t)$ by
\begin{align*}
P_t^m(X_0,A):=\P(X_m(t;X_0)\in A),
\end{align*}
which are well-defined for $t\ge 0$, initial condition $X_0\in\X$ and Borel sets $A\subset\X$. Letting $\B_b(\X)$ denote the set of bounded Borel measurable functions $f:\X \rightarrow \rbb$, the associated Markov semigroup $P_t^m:\B_b(\X)\to\B_b(\X)$ is defined and denoted by
\begin{align*}
P_t^m f(X_0)=\E[f(X_m(t;X_0))], \,\, f\in \B_b(\X).
\end{align*}

\subsection{Geometric ergodicity} \label{sec:result:ergodicity}
We now turn to the topic of large-time properties of equation~\eqref{eqn:GLE:N-particle}. Recall that a probability measure $\mu$ on Borel subsets of $\X$ is called {\bf invariant} for the semigroup $P_t^m$ if for every $f\in \B_b(\X)$
\begin{align*}
\int_{\X} f(X) (P_t^m)^*\mu(\d X)=\int_{\X} f(X)\mu(\d X),
\end{align*}
where $(P_t^m)^*\mu$ is defined as \cite{hairer2011yet}
\begin{align*}
(P_t^m)^*\mu(A) = \int_{\X} P_t^m(X,A)\mu (\d X),
\end{align*}
for all Borel sets $A\subset \X$. Next, we denote by $\L_{m,\gamma}^N$ the generator associated with \eqref{eqn:GLE:N-particle}. One defines $\L_{m,\gamma}^N$ for any $\f\in C^2(\X;\rbb)$ by
\begin{align} 
\L_{m,\gamma}^N\f & =\sum_{i=1}^N\Big[ \la \partial_{x_i} \f, v_i\ra +\frac{1}{m} \Big\la \partial_{v_i} \f, -\gamma v_i-\grad U(x_i) -\sum_{j\neq i} \grad \G(x_i-x_j)+\sum_{\ell=1}^{k_i}\lambda_{i,\ell} z_{i,\ell}\Big\ra \label{form:L_N}\\
&\qquad  +\frac{\gamma}{m^2} \triangle_{v_i} \f + \sum _{\ell=1}^{k_i}\la \partial_{z_{i,\ell}} \f, -\alpha_{i,\ell} z_{i,\ell}-\lambda_{i,\ell} v_i\ra+\sum _{\ell=1}^{k_i} \alpha_{i,\ell}\triangle_{z_{i,\ell}}\f \Big].  \notag
\end{align}

In order to show that $\pi_N(X)\d X$ defined in \eqref{form:pi_N} is invariant for \eqref{eqn:GLE:N-particle}, it suffices to show that $\pi_N(X)$ is a solution of the stationary Fokker-Planck equation
\begin{align} \label{eqn:L^*.pi_N=0}
(\L_{m,\gamma}^N)^*\pi_N(X)=0,
\end{align}
where $(\L_{m,\gamma}^N)^*$ is the dual of $\L_{m,\gamma}^N$, i.e.,
\begin{align*}
\int_{\X} \L_{m,\gamma}^Nf_1(X)\cdot f_2(X) \d X = \int_{\X} f_1(X)\cdot (\L_{m,\gamma}^N)^*f_2(X)\d X,
\end{align*}
for any $f_1,f_2\in C^2_c(\X;\rbb)$. In the absence of the singular potential $G$, this approach was previously employed in \cite{pavliotis2014stochastic} for a finite-dimensional GLE and in \cite{glatt2020generalized} for an infinite-dimensional GLE. With regard to \eqref{eqn:L^*.pi_N=0}, we may simply adapt to the proof of \cite[Proposition 8.2]{pavliotis2014stochastic} tailored to our setting with the appearance of $G$.

Concerning the unique ergodicity of $\pi_N$, we will work with suitable Wasserstein distances allowing for the convenience of measuring the convergent rate toward equilibrium. For a measurable function $V:\X\to(0,\infty)$, we introduce the following weighted supremum norm
\begin{align*}
\|\f\|_V:=\sup_{X\in\X}\frac{|\f(X)|}{1+V(X)}.
\end{align*}
We denote by $\M_V$ the collection of probability measures $\mu$ on Borel subsets of $\X$ such that
\begin{align*}
\int_{\X}V(X)\mu(\d X)<\infty.
\end{align*}
Let $\W_V$ be the corresponding weighted total variation distance in $\M_V$ associated with $\|\cdot\|_V$, given by
\begin{align*}
\W_V(\mu_1,\mu_2)=\sup_{\|\f\|_V\le 1}\Big|\int_{\X} \f(X)\mu_1(\d X)-\int_{\X} \f(X)\mu_2(\d X)\Big|.
\end{align*}
We remark that $\W_V$ is a Wasserstein distance. Indeed, by the dual Kantorovich Theorem 
\begin{align*}
\W_V(\mu_1,\mu_2)= \inf \E \Big[\big( (1+V(X_1))+(1+V(X_2))\big)\boldsymbol{1}\{X_1\neq X_2\}\Big],
\end{align*}
where the infimum runs over all pairs of random variables $(X_1,X_2)$ such that $X_1\sim \mu_1$ and $X_2\sim\mu_2$.
We refer the reader to the monograph \cite{villani2021topics} for a detailed account of Wasserstein distances and optimal transport problems. With this setup, we now state the first main result of the paper, establishing the unique ergodicity of $\pi_N$ defined in \eqref{form:pi_N} as well as the exponential convergent rate toward $\pi_N$.

\begin{theorem} \label{thm:ergodicity:N-particle}
 Under Assumption \ref{cond:U} and Assumption \ref{cond:G}, for every $m>0$ and $\gamma\ge 0$, the probability measure $\pi_N$ defined in \eqref{form:pi_N} is the unique invariant probability measure for \eqref{eqn:GLE:N-particle}. Furthermore, there exists a function $V\in C^2(\X;[1,\infty))$ such that for all $\mu\in\M_V$, 
\begin{align} \label{ineq:geometric-ergodicity}
\W_V\big((P_t^m)^*\mu, \pi_N\big)\le Ce^{-ct}\W_V(\mu,\pi_N),\quad t\ge 0,
\end{align}
for some positive constants $c$ and $C$ independent of $\mu$ and $t$.
\end{theorem}

In order to establish Theorem \ref{thm:ergodicity:N-particle}, we will draw upon the framework of \cite{bolley2018dynamics, hairer2011yet,herzog2017ergodicity,
lu2019geometric,
mattingly2002ergodicity,
meyn2012markov} tailored to our settings. The argument relies on two crucial ingredients: a suitable Lyapunov function, cf. Definition \ref{def:Lyapunov}, and a minorization condition, cf. Definition \ref{def:minorization}. Since it is not difficult to see that the system \eqref{eqn:GLE:N-particle} satisfies hypoellipticity \cite{ottobre2011asymptotic,
pavliotis2014stochastic}, we may employ a relatively standard argument \cite{
mattingly2002ergodicity,pavliotis2014stochastic} to establish the minorization. On the other hand, constructing Lyapunov functions is quite involved requiring a deeper understanding of the dynamics in the presence of the singular potentials. Particularly, while the construction in the case $\gamma>0$ can be adapted to those in the previous work of \cite{herzog2017ergodicity,
lu2019geometric}, the absence of the viscous drag $\gamma=0$ induces further difficulty owing to the interaction between the singular potentials, the velocity $\vb$ and the auxiliary variables $\{\zb_i\}_{i=1,\dots,N}$. To overcome this issue, we will follow \cite{athreya2012propagating, cooke2017geometric,herzog2015noiseI,
herzog2015noiseII, herzog2017ergodicity,lu2019geometric} and perform the technique of \emph{asymptotic scaling} to determine the leading order terms in the dynamics at large energy states. All of this will be carried out in Section \ref{sec:ergodicity}. The proof of Theorem \ref{thm:ergodicity:N-particle} will be given in Section \ref{sec:ergodicity:proof-of-theorem}.

\subsection{Small mass limit} \label{sec:result:small-mass}
We now consider the topic of small mass limit and rigorously compare the solution of \eqref{eqn:GLE:N-particle} with that of  \eqref{eqn:GLE:N-particle:m=0} as $m\to 0$. 

As mentioned in Section \ref{sec:intro:small-mass}, the derivation of \eqref{eqn:GLE:N-particle:m=0} follows the framework of  \cite{cerrai2020averaging,
herzog2016small,lim2019homogenization,
lim2020homogenization} for finite-dimensional second-order systems as well as \cite{nguyen2018small,shi2021small} for infinite-dimensional systems. The trick employed involves an integration by parts on the $z$-equation of \eqref{eqn:GLE:N-particle}, so as to decouple the velocity from the other processes. As a result, it induces extra drift terms appearing on the right-hand side of the $q$-equation in \eqref{eqn:GLE:N-particle:m=0}. Also, the initial conditions of \eqref{eqn:GLE:N-particle:m=0} are closely related to those of \eqref{eqn:GLE:N-particle}. For the sake of clarity, we defer the explanation in detail to Section \ref{sec:small-mass:heuristic}.

Due to the presence of singular potentials, the first issue arising from \eqref{eqn:GLE:N-particle:m=0} is the well-posedness. To this end, we introduce the new phase space for the solutions of \eqref{eqn:GLE:N-particle:m=0} defined as
\begin{align*}
\Q=\D\times \prod_{i=1}^{N} (\rbb^d)^{k_i},
\end{align*}
where we recall $\D$ being the state space for the displacement as in \eqref{form:D}. The existence and uniqueness of a strong solution in $\Q$ for \eqref{eqn:GLE:N-particle:m=0} are guaranteed in the following auxiliary result. 
\begin{proposition} \label{prop:well-posedness:m=0}  Under Assumption \ref{cond:U} and Assumption \ref{cond:G}, for every initial condition $Q_0=(\qb(0)$, $\fb_{1}(0)$,$\dots$,$\fb_N(0))\in \Q$, system \eqref{eqn:GLE:N-particle:m=0} admits a unique strong solution $Q(t;Q_0)=\big(\qb(t)$, $\fb_{1}(t)$, ..., $\fb_{N}(t)\big)\in\Q$.
\end{proposition}

Similar to the proof of Proposition \ref{prop:GLE:N-particle:well-posed}, the argument of Proposition \ref{prop:well-posedness:m=0} follows from the energy estimate established in Lemma \ref{lem:GLE:N-particle:m=0:supnorm} below in Section \ref{sec:small-mass:estimate-m=0}. In turn, the result in Lemma \ref{lem:GLE:N-particle:m=0:supnorm} relies on suitable Lyapunov function specifically designed for \eqref{eqn:GLE:N-particle:m=0}. It will also be employed to study the small mass limit.

We now state the second main result of the paper giving the validity of the approximation of \eqref{eqn:GLE:N-particle} by \eqref{eqn:GLE:N-particle:m=0} in the small mass regime.

\begin{theorem} \label{thm:small-mass:N-particle}
Suppose that Assumption \ref{cond:U}, Assumption \ref{cond:G} and Assumption \ref{cond:G:d=1} hold. For every $(\xb(0)$, $\vb(0)$, $\zb_{1}(0)$,..., $\zb_N(0))\in \X$, let $X_m(t)=\big(\xb_m(t)$, $\vb_m(t)$, $\zb_{1,m}(t)$, ..., $\zb_{N,m}(t)\big)$ be the solution of \eqref{eqn:GLE:N-particle} and let $Q(t)=\big(\qb(t)$, $\fb_{1}(t)$,$\dots$, $\fb_{N}(t)\big)$ be the solution of~\eqref{eqn:GLE:N-particle:m=0} with the following initial condition
\begin{align} \label{cond:initial-condition:(q,f)}
q_i(0) = x_i(0),\quad f_{i,\ell}(0)=z_{i,\ell}(0)+\lambda_{i,\ell} x_i(0),\quad i=1,\dots,N, \quad \ell=1,\dots,k_i.
\end{align}
Then, for every $T,\,\xi>0$, it holds that
\begin{align} \label{lim:small-mass:probability}
\P\Big\{\sup_{t\in[0,T]}|\xb_m(t)-\qb(t)|>\xi\Big\}\rightarrow 0,\quad m\rightarrow 0.
\end{align}

\end{theorem}

 In order to establish Theorem \ref{thm:small-mass:N-particle}, we will adapt to the approach in \cite{herzog2016small,nguyen2018small} tailored to our settings. The argument can be summarized as follows: we first truncate the nonlinearities in \eqref{eqn:GLE:N-particle} and \eqref{eqn:GLE:N-particle:m=0} while making use of conditions \eqref{cond:U:U''(x)=O(x^lambda-1)} and \eqref{cond:G:grad^2.G(x)<1/|x|^beta}. This results in Lipschitz systems, thus allowing for proving the small mass limit of the truncated systems. We then exploit the moment estimates on \eqref{eqn:GLE:N-particle:m=0}, cf. Lemma \ref{lem:GLE:N-particle:m=0:supnorm}, to remove the Lipschitz assumption. The explicit argument will be carried out in a series of results in Section \ref{sec:small-mass}.
 
Finally, we remark that the convergence in \eqref{lim:small-mass:probability} only holds in probability, but not in $L^p$. Following \cite[Theorem 2.2]{higham2002strong}, the latter is a consequence of an $L^p$ estimate on $\xb_m(t)$ that is uniform with respect to the mass, i.e.,
 \begin{align}\label{lim:small-mass:L^p}
 \limsup_{m\to 0}\sup_{t\in[0,T]}|\xb_m(t)|^p<\infty.
 \end{align}
Due to the singularities, \eqref{lim:small-mass:L^p} is not available in this work, and would require further insight. Therefore, the small mass convergence in $L^p$ remains an open problem.

\section{Geometric Ergodicity} \label{sec:ergodicity}

Throughout the rest of the paper, $c$ and $C$ denote generic positive constants that may change from line to line. The main parameters that they depend on will appear between parenthesis, e.g., $c(T,q)$ is a function of $T$ and $q$. In this section, since we do not take $m\to 0$, we will drop the subscript $m$ in $\xb_m,\vb_m$ and elsewhere for notational convenience.

In this section, we establish the unique ergodicity and the exponential convergent rate toward the Gibbs measure $\pi_N$ defined in \eqref{form:pi_N} for \eqref{eqn:GLE:N-particle}. The argument will rely on the construction of suitable Lyapunov functions while making use of a standard irreducibility condition. For the reader's convenience, we first recall the definitions of these notions below.

\begin{definition} \label{def:Lyapunov}
A function $V\in C^2(\X;[1,\infty))$ is called a \emph{Lyapunov} function for \eqref{eqn:GLE:N-particle} if the followings hold:

(i) $V(X)\to \infty$ whenever $|X|+\sum_{1\le i<j\le N}|x_i-x_j|^{-1}\to \infty$ in $\X$; and

(ii) for all $X\in\X$, 
\begin{align} \label{ineq:Lyapunov}
\mathcal{L}_{m,\gamma}^N V(X)\le -c \mathcal{L}_{m,\gamma}^N V(X)+D, 
\end{align}
for some constants $c>0$ and $D\ge 0$ independent of $X$.

\end{definition}

\begin{definition} \label{def:minorization}
Let $V$ be a Lyapunov function as in Definition \ref{def:Lyapunov}. Denote 
\begin{align*}
\X_R =\big\{X\in\X:V(X)\le R\big\}.
\end{align*}
The system \eqref{eqn:GLE:N-particle} is said to satisfy a \emph{minorization} condition if for all $R$ sufficiently large, there exist positive constants $t_R,\,c_R$, a probability measure $\nu_R$ such that $\nu_R(\X_R)=1$ and for every $X\in\X_R$ and any Borel set $A\subset \X$,
\begin{align} \label{ineq:minorization}
P_{t_R}^m(X,A)\ge c_R\nu_R(A).
\end{align}
\end{definition}

In Section \ref{sec:ergodicity:single-particle}, we provide a heuristic argument on the construction of a Lyapunov function for the instance of single-particle GLE. In particular, we show that our construction works well for this simpler setting, both in the case $\gamma>0$ and $\gamma=0$ through Lemma \ref{lem:Lyapunov:V_1} and Lemma \ref{lem:Lyapunov:V_2}, respectively. In Section \ref{sec:ergodicity:N-particle}, we adapt the construction in the single-particle to the general $N$-particle system \eqref{eqn:GLE:N-particle}, and establish suitable Lyapunov function so as to achieve the desired dissipative estimates. Together with an auxiliary result on minorization, cf. Lemma \ref{lem:minorization}, we will conclude the proof of Theorem \ref{thm:ergodicity:N-particle} in Section \ref{sec:ergodicity:proof-of-theorem}.

\subsection{Single-particle GLE} \label{sec:ergodicity:single-particle} Following the framework developed in \cite{herzog2017ergodicity,lu2019geometric}, we will build up intuition for the construction of Lyapunov functions for the full system \eqref{eqn:GLE:N-particle} by investigating a simpler equation in the absence of interacting forces. More specifically, we introduce the following single-particle GLE ($N=1$) 
\begin{align} \label{eqn:GLE}
\d\, x(t) &= v(t)\d t, \nt  \\
m\,\d\, v(t) & = -\gamma v(t) \d t -\grad \U(x(t))\d t - \grad \G(x(t)) \d t +\sum _{i=1}^k \lambda_i z_i(t)\d t +\sqrt{2\gamma} \,\d W_0(t),\\ 
\d\, z_i(t) &= -\alpha_i z_i(t)\d t-\lambda_i v(t)\d t+\sqrt{2\alpha_i}\,\d W_i(t),\quad i=1,\dots,k.\nt 
\end{align}
Let $\L_{m,\gamma}$ be the generator associated with \eqref{eqn:GLE}. That is for $\gamma\ge 0$,
\begin{align} 
\L_{m,\gamma}\f & = \la \partial_x \f, v\ra + \frac{1}{m}\bigg\la \partial_v \f, -\gamma v-\grad U(x) -\grad \G(x)+\sum_{i=1}^k\lambda_i z_i\bigg\ra \label{form:L}\\
&\qquad  +\frac{\gamma}{m^2} \triangle_v \f + \sum_{i=1}^k\la \partial_{z_i} \f, -\alpha_i z_i-\lambda_i v\ra+\sum _{i=1}^k \alpha_i\triangle_{z_i}\f,  \notag
\end{align}
where $\f=\f(x,v,z_1,\ldots, z_k)\in C^2(\rbb^{(k+2)d})$. Denote $H$
\begin{align} \label{form:H}
H\big(x,v,z_1,\dots,z_k\big) = \U(x)+\G(x)+ \frac{1}{2}m|v|^2+\frac{1}{2}\sum_{i=1}^k |z_i|^2.
\end{align}
Note that $\L_{m,\gamma}\f$ can be written as
\begin{equation}
\label{form2:H}
\L_{m,\gamma} \f=J\nabla H\cdot\nabla \f-\sigma\nabla H\cdot\nabla \f+\mathrm{div}(\sigma\nabla\f),    
\end{equation}
where
\begin{align*}
\nabla H&=\begin{pmatrix}
   \partial_x H\\
   \partial_v H\\
   \partial_{z_1}H\\
   \vdots\\
   \partial_{z_k}H
\end{pmatrix}=\begin{pmatrix}
   \nabla U(x)+\nabla G(x)\\
   mv\\
   z_1\\
   \vdots\\
   z_k
\end{pmatrix},\quad \nabla\f=\begin{pmatrix}
   \partial_x \f\\
   \partial_v \f\\
   \partial_{z_1}\f\\
   \vdots\\
   \partial_{z_k}\f
\end{pmatrix}, \\
J &:= \frac{1}{m}\begin{pmatrix}
0&I &0&\cdots&0\\
-I&0&\lambda_1&\cdots&\lambda_k\\
0&-\lambda_1&0&\cdots&0\\
\vdots&\vdots&\vdots&\ddots&\vdots\\
0&-\lambda_k&0&\cdots&0
\end{pmatrix},\quad\text{and}\quad \sigma=\sigma_\gamma :=\begin{pmatrix}
0&0&0&\cdots&0\\
0&\frac{\gamma}{m^2}&0&\cdots&0\\
0&0&\alpha_1&\cdots&0\\
\vdots&\vdots&\vdots&\ddots&\vdots\\
0&0&0&\cdots&\alpha_k
\end{pmatrix}.
\end{align*} 
We notice that $J$ is anti-symmetric while $\sigma$ is positive semi-definite. This formulation will be more convenient in the subsequent computations. In particular, if $\f$ is linear in $v$ and $z$ then the last term in \eqref{form2:H} vanishes.

Now, there are two cases to be considered depending on the value of the viscous constant $\gamma$, of which, the first case is when $\gamma>0$.  
\subsubsection{Positive viscous constant $\gamma>0$} \label{sec:ergodicity:single-particle:gamma>0} In this case, we observe that system \eqref{eqn:GLE} is almost the same as the following single-particle Langevin equation without the auxiliary memory variables
\begin{align} \label{eqn:Langevin}
\d\, x(t) &= v(t)\d t, \nt  \\
m\,\d\, v(t) & = -\gamma v(t) \d t -\grad \U(x(t))\d t - \grad \G(x(t)) \d t+\sqrt{2\gamma} \,\d W_0(t).
\end{align}
Before discussing the Lyapunov construction for \eqref{eqn:GLE}, it is illuminating to recapitulate the heuristic argument for \eqref{eqn:Langevin} from \cite{herzog2017ergodicity,lu2019geometric,
mattingly2002ergodicity}. As a first ansatz, one can facilitate the Hamiltonian of \eqref{eqn:Langevin} given by
\begin{align*}
\Ht(x,v)=\frac{1}{2}m|v|^2+U(x)+G(x).
\end{align*}
Denoting by $\Lt$ the generator associated with \eqref{eqn:Langevin}, i.e.,
\begin{align} \label{form:L_tilde} 
\Lt\f  = \la \partial_x \f, v\ra + \frac{1}{m}\big\la \partial_v \f, -\gamma v-\grad U(x) -\grad \G(x)\big\ra +\frac{\gamma}{m^2} \triangle_v \f ,
\end{align}
a routine computation shows that
\begin{align*}
\Lt \Ht \le -c|v|^2+D.
\end{align*}
While this is sufficient for the well-posedness of \eqref{eqn:Langevin}, it does not produce the dissipative effect on $x$ when $|x|\to\infty$ and $|x|\to 0$, so as to establish geometric ergodicity. In the absence of the singular potential $G$, one can exploit the trick in \cite{mattingly2002ergodicity} by considering a perturbation of the form
\begin{align} \label{form:Htilde+epsilon<x,v>}
\Ht(x,v)+\varepsilon \la x,v\ra,
\end{align} 
where $\varepsilon$ is sufficiently small. Applying $\Lt$ to the above function while making use of condition \eqref{cond:U:x.U'(x)>-x^(lambda+1)}, it is not difficult to see that 
\begin{align*}
\Lt (\Ht(x,v)+\varepsilon\la x,v\ra) \le -c (|v|^2+U(x))+C.
\end{align*}
On the other hand, in the work of \cite{herzog2017ergodicity}, the following function was introduced
\begin{align*}
\exp\big\{b(\Ht(x,v)+\psi(x,v))\big\},
\end{align*}
where the perturbation $\psi(x,v)$ satisfies
\begin{align*}
\Lt\psi(x,v)\le -C,
\end{align*}
for a sufficiently large constant $C$. In order to construct $\psi$, it is crucial to determine the leading order terms in $\Lt$ when $U(x)+G(x)$ is large while $v$ is being fixed. Following the idea previously presented in \cite[Section 3.2]{herzog2017ergodicity}, in this situation, $\Lt$ is approximated by
\begin{align*}
\Lt\approx \A_1=-\grad (U(x)+ G(x))\cdot\grad_v.
\end{align*}
In turn, this suggests $\psi$ satisfies
\begin{align*}
\A_1\psi\le -C.
\end{align*}
A candidate for the above inequality is given by \cite{herzog2017ergodicity}
\begin{align*}
\psi(x,v)\propto\frac{\la v,\grad (U(x)+G(x))\ra}{|\grad (U(x)+G(x))|^2}.
\end{align*}
We refer the reader to \cite{herzog2017ergodicity} for the derivation of $\psi$ in more detail. While this choice of $\psi$ works well for Lennard-Jones and Riesz potentials, it is not applicable to the class of log functions, of which, the Coulomb potential in dimension $d=2$ is a well-known example. 

With regard to the specific instance of Coulomb potentials \eqref{form:Coulomb}, in the work of \cite{lu2019geometric}, the authors tackle the Lyapunov issue by employing the scaling transformation
\begin{align*}
x = \kappa^{-1}\xhat,\quad v=\vhat,
\end{align*}
for $\kappa>0$. In dimension $d=2$ (with $G(x)=-\log|x|$), $\Lt$ as in \eqref{form:L_tilde} can be recast as
\begin{align*}
\Lt &=  v\cdot \grad_x +\frac{1}{m}\Big(-\gamma v-\grad U(x)+\frac{x}{|x|^2}  \Big)\cdot \grad_v+\frac{\gamma}{m^2}\triangle_{v} \\
& = \kappa\,\vhat \cdot \grad_{\xhat} +\frac{1}{m}\Big(-\gamma \vhat -\grad U(\kappa^{-1}\xhat)+\kappa \frac{\xhat}{|\xhat|^2} \Big)\cdot \grad_{\vhat}+\frac{\gamma}{m^2}\triangle_{\vhat}.
\end{align*}
Recalling condition \eqref{cond:U:U'(x)=O(x^lambda)},  since $|\grad U(\kappa^{-1} \xhat)| \approx \kappa^{-\lambda}|\xhat|^{\lambda}$, which is negligible as $\kappa$ is large, we observe that
\begin{align*}
\Lt &\approx  \kappa\,\vhat \cdot \grad_{\xhat}+\frac{\kappa}{m}\frac{\xhat}{|\xhat|^2}\cdot \grad _{\vhat},\quad \kappa\to\infty.
\end{align*}
Likewise, in dimension $d\ge 3$ (with $G(x)=|x|^{2-d}$),
\begin{align*}
\Lt & = \kappa\,\vhat \cdot \grad_{\xhat} +\frac{1}{m}\Big(-\gamma \vhat -\grad U(\kappa^{-1}\xhat)+2\kappa^{d-1} \frac{\xhat}{|\xhat|^d} \Big)\cdot \grad_{\vhat}+\frac{\gamma}{m^2}\triangle_{\vhat}\\
&\approx \frac{2}{m}\kappa^{d-1}\frac{\xhat}{|\xhat|^2}\cdot \grad _{\vhat}, \quad k\to \infty.
\end{align*}
So, taking $\kappa$ to infinity (i.e., taking $|x|\to0$) indicates
\begin{align*}
\Lt\approx \A_2= \begin{cases} v\cdot\grad_x+\frac{x}{|x|^2}\cdot \grad_v,&d=2,\\
\frac{x}{|x|^d}\cdot \grad_v,& d\ge 3. \end{cases}
\end{align*}
A typical choice of $\psi$ satisfying $\A_2\psi\le -C$ is given by
\begin{align*}
\psi(x,v)\propto -\frac{\la x,v\ra}{|x|}.
\end{align*}
Together with \eqref{form:Htilde+epsilon<x,v>}, we deduce that a Lyapunov function $V$ for \eqref{eqn:Langevin} with Coulomb potentials has the following form \cite{lu2019geometric}
\begin{align*}
V\propto\Ht(x,v)+\varepsilon_1\la x,v\ra-\varepsilon_2\frac{\la x,v\ra}{|x|},
\end{align*}
for some positive constants $\varepsilon_1$ and $\varepsilon_2$ sufficiently small.

Turning back to \eqref{eqn:GLE} in the case $\gamma>0$, motivated by the above discussion, for $\varepsilon>0$, we
introduce the function $V_1$ given by
\begin{align} \label{form:V:gamma>0}
V_1\big(x,v,z_1,\dots,z_k\big):= H\big(x,v,z_1,\dots,z_k\big)+m\varepsilon\la x,v\ra-m\varepsilon\frac{\la x,v\ra}{|x|},
\end{align}
where $H$ is defined in \eqref{form:H}. Since \eqref{eqn:GLE} only differs from \eqref{eqn:Langevin} by the appearance of the linear memory variables $z_i$'s, it turns out that $V_1$ is indeed a Lyapunov function for \eqref{eqn:GLE}. This is summarized in the following lemma.

\begin{lemma} \label{lem:Lyapunov:V_1} Under Assumption \ref{cond:U} and Assumption \ref{cond:G}, let $V_1$ be defined as in \eqref{form:V:gamma>0}. For each $\gamma>0$ and $m>0$, there exists a positive constant $\varepsilon$ sufficiently small such that $V_1$ is a Lyapunov function for \eqref{eqn:GLE}.
\end{lemma}
\begin{proof}
Letting $\L_{m,\gamma}$ and $H$ respectively be defined in~\eqref{form:L} and \eqref{form:H}, It\^o's formula yields
\begin{align} \label{eqn:L.H}
\L_{m,\gamma} H &=-\sigma\nabla H\cdot\nabla H+\mathrm{div}(\sigma \nabla H)\notag
\\&=-\gamma  |v|^2- \sum_{i=1}^k \alpha_i|z_i|^2+\frac{1}{m}\cdot\frac{d}{2}\gamma +\frac{d}{2}\sum_{i=1}^k \alpha_i.
\end{align}
Similarly, we compute
\begin{align}\label{eqn:L.m<x,v>}
\L_{m,\gamma} \big(m\la x,v\ra\big)= m|v|^2-\gamma\la x,v\ra -\la \grad \U(x)+\grad \G(x),x\ra+\sum_{i=1}^k\lambda_i\la z_i,x\ra.
\end{align}
To estimate the right-hand side above, we recall from condition \eqref{cond:U:x.U'(x)>-x^(lambda+1)} that
\begin{align*}
-\la \grad\U(x),x\ra \le -a_2|x|^{\lambda+1}+a_3.
\end{align*}
With regard to $\grad \G$, condition \eqref{cond:G:grad.G(x)<1/|x|^beta} implies that
\begin{align*}
|\la \grad \G(x),x\ra |\le \frac{a_1}{|x|^{\beta_1-1}}+a_1|x|.
\end{align*}
Concerning the cross terms $\la x,v\ra$ and $\la z_i,x\ra$, for $\varepsilon\in (0,1)$, we employ Cauchy-Schwarz inequality to deduce
\begin{align*}
-\gamma\varepsilon \la x,v\ra +\varepsilon\sum_{i=1}^k\lambda_i\la z_i,x\ra \le \gamma\varepsilon^{1/2}|v|^2+\varepsilon^{1/2} \sum_{i=1}^k\lambda_i^2|z_i|^2 + (k+\gamma)\varepsilon^{3/2}|x|^2.
\end{align*}
Together with~\eqref{eqn:L.m<x,v>}, we find 
\begin{align} \label{ineq:L_gamma.<x,v>:gamma}
\L_{m,\gamma} \big(\varepsilon m\la x,v\ra\big) 
&\le (\varepsilon m+\gamma\varepsilon^{1/2}) |v|^2+\varepsilon^{1/2}\sum_{i=1}^k\lambda_i^2 |z_i|^2 + (k+\gamma)\varepsilon^{3/2}|x|^2\notag\\
&\qquad-a_2\varepsilon|x|^{\lambda+1}+a_3\varepsilon+\varepsilon\frac{a_1}{|x|^{\beta_1-1}}+\varepsilon a_1|x|,
\end{align}
whence,
\begin{align}\label{ineq:L_gamma.<x,v>}
\L_{m,\gamma} \big(\varepsilon m\la x,v\ra\big) 
&\le C\Big(\varepsilon^{1/2} |v|^2+\varepsilon^{1/2}\sum_{i=1}^k |z_i|^2 + \varepsilon^{3/2}|x|^2+\varepsilon\frac{1}{|x|^{\beta_1-1}}+1\Big)-a_2\varepsilon|x|^{\lambda+1}, 
\end{align}
for some positive constant $C$ independent of $\varepsilon$.

Turning to $-\la x,v\ra/|x| $, it holds that
\begin{align}
&\L_{m,\gamma}\bigg(\!-m\frac{\la x,v\ra}{|x|}\bigg) \notag\\
& = -m\frac{|v|^2}{|x|}+m\frac{|\la x,v\ra|^2}{|x|^3}+\gamma \frac{\la x,v\ra}{|x|}+ \frac{\la\grad U(x), x\ra}{|x|}+\frac{\la \grad \G(x),x\ra}{|x|}-\sum_{i=1}^k\lambda_i\frac{\la z_i,x\ra}{|x|}  .\label{eqn:L.m<x,v>/|x|}
\end{align}
It is clear that
\begin{align*}
 -m\frac{|v|^2}{|x|}+m\frac{|\la x,v\ra|^2}{|x|^3}\le 0.
\end{align*}
From condition~\eqref{cond:U:U'(x)=O(x^lambda)}, we readily have
\begin{align*}
 \frac{\la\grad U(x), x\ra}{|x|} \le |\grad \U(x)|\le a_1(1+|x|^\lambda).
\end{align*}
 Also, recalling \eqref{cond:G:|grad.G(x)+q/|x|^beta_1|<1/|x|^beta_2},  
\begin{align} \label{cond:G:<grad.G(x),x>/|x|}
\frac{\la \grad \G(x),x\ra}{|x|} = -\frac{a_4}{|x|^{\beta_1}}+ \frac{\la \grad \G(x)+a_4\frac{x}{|x|^{\beta_1+1}},x\ra}{|x|}& \le  -\frac{a_4}{|x|^{\beta_1}}+\frac{a_5}{|x|^{\beta_2}}+a_6\nt\\
&\le -\frac{a_4}{2|x|^{\beta_1}}+C.
\end{align}
In the last estimate above, we subsumed $|x|^{-\beta_2}$ into $-|x|^{-\beta_1}$ thanks to the fact that $\beta_2\in[0,\beta_1)$ by virtue of the condition \eqref{cond:G:|grad.G(x)+q/|x|^beta_1|<1/|x|^beta_2}. Altogether, we deduce that
\begin{align} \label{ineq:L_gamma.m<x,v>/|x|}
\L_{m,\gamma}\bigg(\!-m\varepsilon\frac{\la x,v\ra}{|x|}\bigg) 
&\le C\varepsilon\Big(|v|+\sum_{i=1}^k |z_i|+|x|^\lambda+ 1 \Big)-\varepsilon\frac{a_4}{2|x|^{\beta_1}}\nt\\
& \le C\Big(\varepsilon^2|v|^2+\varepsilon^2\sum_{i=1}^k |z_i|^2+\varepsilon|x|^\lambda+ 1 \Big)-\varepsilon\frac{a_4}{2|x|^{\beta_1}} .
\end{align}

Now, we combine estimates \eqref{ineq:L_gamma.<x,v>} and \eqref{ineq:L_gamma.m<x,v>/|x|} together with identities \eqref{form:V:gamma>0} and  \eqref{eqn:L.H} to infer
\begin{align*} 
\L_{m,\gamma} V_1 &= \L_{m,\gamma}\bigg(H+m\varepsilon\la x,v\ra-m\varepsilon\frac{\la x,v\ra}{|x|}\bigg) \notag \\
&\le  -c |v|^2- c\sum_{i=1}^k |z_i|^2-c\varepsilon |x|^{\lambda+1}-\varepsilon\frac{c}{|x|^{\beta_1}}+C\\
&\qquad+  C\Big(\varepsilon^{1/2} |v|^2+\varepsilon^{1/2}\sum_{i=1}^k |z_i|^2 + \varepsilon^{3/2}|x|^2+\varepsilon|x|^\lambda+\varepsilon\frac{1}{|x|^{\beta_1-1}}\Big),
\end{align*}
for some positive constants $c,C$ independent of $\varepsilon$. By taking $\varepsilon$ sufficiently small, we observe that the positive non-constant terms on the above right-hand side are dominated by the negative terms. In particular, since $\lambda\ge 1$, $\varepsilon^{3/2}|x|^2$ can be subsumed into $-\varepsilon|x|^{\lambda+1}$. As a consequence, we arrive at
\begin{align} \label{ineq:L_gamma.V_1}
\L_{m,\gamma} V_1 &\le  -c |v|^2- c\sum_{i=1}^k |z_i|^2-c\varepsilon |x|^{\lambda+1}-\varepsilon\frac{c}{|x|^{\beta_1}}+C.
\end{align}
Since $U$ and $G$ are bounded by $|x|^{\lambda+1}+|x|^{-\beta_1}$, \eqref{ineq:L_gamma.V_1} produces the desired Lyapunov property of $V_1$ for \eqref{eqn:GLE}. The proof is thus finished.
\end{proof}

\subsubsection{Zero viscous constant $\gamma=0$} \label{sec:ergodicity:single-particle:gamma=0}

We now turn to the instance $\gamma=0$. In this case, since there is no viscous drag on the right-hand side of the $v$-equation in \eqref{eqn:GLE}, the function $V_1$ defined in \eqref{form:V:gamma>0} does not produce the dissipation in $v$ for large $v$. To circumnavigate this issue, we note that the $z_i$-equation in \eqref{eqn:GLE} still depends on $v$. So, we may exploit this fact to transfer the dissipation from, say $z_1$ to $v$. More specifically, let us consider adding a small perturbation to $V_1$ as follows: 
\begin{align*}
V_1+m\varepsilon\la v,z_1\ra.
\end{align*} 
Denote by $\L_{m,0}$ the generator associated with \eqref{eqn:GLE} when $\gamma=0$. That is, from \eqref{form:L}, we have
\begin{align} 
\L_{m,0}\f & = \la \partial_x \f, v\ra + \frac{1}{m}\bigg\la \partial_v \f,-\grad U(x) -\grad \G(x)+\sum_{i=1}^k\lambda_i z_i\bigg\ra \label{form:L:gamma=0}\\
&\qquad  + \sum_{i=1}^k\la \partial_{z_i} \f, -\alpha_i z_i-\lambda_i v\ra+\sum _{i=1}^k \alpha_i\triangle_{z_i}\f.  \notag
\end{align}
Applying $\L_{m,0}$ to the cross term $\la v,z_1\ra$, although $\L_{m,0}\la v,z_1\ra$ provides the required dissipative effect in $v$, it also induces a cross product $\la \grad G(x),z_1\ra$, which has the order of $|z_1|/|x|^{-\beta_1}$ by virtue of condition \eqref{cond:G:grad.G(x)<1/|x|^beta}. That is
\begin{align} \label{ineq:L_(m,0)<v,z_1>}
\L_{m,0}\la v ,z_1\ra \le -c|v|^2+ \frac{|z_1|}{|x|^{\beta_1}}.
\end{align}
How to annihilate the effect caused by this extra term is the main difficulty that we face in the case $\gamma=0$.

To circumvent the issue, from the proof of Lemma \ref{lem:Lyapunov:V_1}, particularly the estimate \eqref{ineq:L_gamma.m<x,v>/|x|}, we see that 
\begin{align*}
\L_{m,0}\Big(-\frac{\la x,v\ra}{|x|}\Big) \propto -\frac{1}{|x|^{\beta_1}}+ \f(x,v,z_1,\dots,z_k),
\end{align*}
where $\f(x,v,z_1,\dots,z_k)$ consists of lower order terms. This suggests that we look for a perturbation of the form
\begin{align*}
-\frac{\la x,v\ra}{|x|}\psi(x,v,z_1),
\end{align*} 
where $\psi$ satisfies
\begin{align} \label{cond:psi:gamma=0}
\psi\ge c|z_1|,\quad\text{and}\quad |\L_{m,0}\psi|=O(|v|+ |\zb|).
\end{align}
From \eqref{ineq:L_(m,0)<v,z_1>} and \eqref{cond:psi:gamma=0}, the terms $|z_1|/|x|^{\beta_1}$ and $O(|v|+|\zb|)$ suggest that the issue  is where $|v|$, $|\zb|$ are large and $|x|$ is small. To derive $\psi$, it is important to understand the dynamics in $\L_{m,0}$ in this ``bad region". So, we introduce the following scaling transformation
\begin{align*}
(x,v,z_1,\dots,z_k)=(\kappa^{-a}\xhat,\kappa \vhat,\kappa \zhat_1,\kappa\zhat_2,\dots,\kappa\zhat_k),
\end{align*}
for some positive constant $a>1$. Recalling $\L_{m,0}$ as in \eqref{form:L:gamma=0}, under this scaling, we find that
\begin{align*}
\L_{m,0}&= v\cdot \grad_{x}+\frac{1}{m}\Big(-\grad U(x)-\grad G(x)+\sum_{i=1}^k \lambda_iz_i\Big)\cdot \grad_v+ \sum_{i=1}^k \big(-\alpha_iz_i-\lambda_i v\big)\cdot\grad_{z_i}+\sum_{i=1}^k \alpha_i \triangle_{z_i} \\
&= \kappa^{a+1}\vhat\cdot \grad _{\xhat} +\frac{1}{m\kappa}\Big(-\grad U(\kappa^{-a} \xhat)-\grad G(\kappa^{-a} \xhat)+\kappa\sum_{i=1}^k \lambda_i\zhat_i\Big)\cdot \grad_{\vhat}\\
&\qquad +\sum_{i=1}^k \big(-\alpha_i \zhat_i-\lambda_i \vhat\big)\cdot\grad_{\zhat_i}+\kappa^{-2}\sum_{i=1}^k \alpha_i \triangle_{\zhat_i}.
\end{align*}
Recalling condition \eqref{cond:U:U'(x)=O(x^lambda)} and condition \eqref{cond:G:grad.G(x)<1/|x|^beta}, suppose heuristically that
\begin{align*}
-\grad U(\kappa^{-a} \xhat)-\grad G(\kappa^{-a} \xhat) \approx -\kappa^{-a\lambda}\grad U(\xhat)-\kappa^{a\beta_1}\grad G(\xhat),
\end{align*}
implying,
\begin{align*}
\L_{m,0}&\approx \kappa^{a+1} \vhat\cdot\grad_{\xhat} -\kappa^{-a\lambda -1}\grad U(\xhat)\cdot \grad_{\vhat}-\kappa^{a\beta_1-1 } \grad G(\xhat)\cdot \grad_{\vhat}+ \sum_{i=1}^k \zhat_i\cdot \grad_{\vhat}\\
&\qquad-\sum_{i=1}^k (\zhat_i+\vhat\big)\cdot \grad_{\zhat_i}+\kappa^{-2}\sum_{i=1}^k \triangle_{\zhat_i}.
\end{align*}
Observe that
\begin{align*}
\L_{m,0}&\approx \kappa^{a+1} \vhat\cdot\grad_{\xhat}-\kappa^{a\beta_1-1 } \grad G(\xhat)\cdot \grad_{\vhat},\quad \kappa\to \infty.
\end{align*}
In other words, when $\kappa$ is large, the dominant balance of terms in the above transformation is contained in 
\begin{align*}
\A_{m,0} = v\cdot\grad_x-\grad G(x)\cdot \grad _v. 
\end{align*}
Together with the requirement \eqref{cond:psi:gamma=0}, a typical choice for $\psi$ is given by
\begin{align*}
\psi=\sqrt{a_1|z_1|^2+\tfrac{1}{2}m|v|^2+G(x)+U(x)+a_2},
\end{align*}
for some positive constants $a_1,a_2$ to be chosen later. In the above, we note that the appearance of $U(x)$ is to ensure the expression under the square root is positive. In summary, the candidate Lyapunov function for \eqref{eqn:GLE} looks like
\begin{align*}
V_1+m\varepsilon\la v,z_1\ra-m\varepsilon \frac{\la x,v\ra}{|x|}\psi.
\end{align*}  
In Lemma \ref{lem:Lyapunov:V_2} below, we will see that by picking $a_1,a_2$ carefully, we will achieve the Lyapunov effect for \eqref{eqn:GLE}. We finish this discussion by introducing the following function $V_2$ defined for $\varepsilon\in(0,1),\, R>1$,
\begin{align}
&V_2\big(x,v,z_1,\dots,z_k\big) \nt\\
&:=H\big(x,v,z_1,\dots,z_k\big)+ \varepsilon R m\la x,v\ra+\varepsilon R^2 m\la v,z_1\ra-\varepsilon m\frac{\la x,v\ra }{|x|}\sqrt{Q_R}  \label{form:V_2:gamma=0},
\end{align}
where
\begin{align} \label{form:Q_R}
Q_R = R^4|z_1|^2+m|v|^2+2U(x)+2G(x)+R.
\end{align}
\begin{lemma} \label{lem:Lyapunov:V_2} Under Assumption \ref{cond:U} and Assumption \ref{cond:G}, let $V_2$ be defined as in \eqref{form:V_2:gamma=0}. For $\gamma=0$ and every $m>0$, there exist positive constants $\varepsilon$ small and $R$ large enough such that $V_2$ is a Lyapunov function for \eqref{eqn:GLE}.
\end{lemma}
\begin{proof} Firstly, we note that when $\gamma=0$, by virtue of Assumption~\ref{cond:U}, $\lambda=1$. From the estimate \eqref{ineq:L_gamma.<x,v>:gamma}, we immediately obtain
\begin{align*}
\L_{m,0} \Big(\varepsilon R m\la x,v\ra\Big) 
&\le CR\Big(\varepsilon |v|^2+\varepsilon^{1/2}\sum_{i=1}^k |z_i|^2 + \varepsilon^{3/2}|x|^2+\varepsilon\frac{1}{|x|^{\beta_1-1}}+1\Big)-a_2\varepsilon R|x|^{2}.
\end{align*}
Also, \eqref{eqn:L.H} is reduced to
\begin{align*}
\L_{m,0} H = - \sum_{i=1}^k \alpha_i|z_i|^2+\frac{d}{2}\sum_{i=1}^k \alpha_i.
\end{align*}
As a consequence, we obtain
\begin{align*}
\L_{m,0}\big(H+  \varepsilon Rm\la x,v\ra\big)
&\le   CR\Big(\varepsilon |v|^2+\varepsilon^{1/2}\sum_{i=1}^k |z_i|^2 + \varepsilon^{3/2}|x|^2+\varepsilon\frac{1}{|x|^{\beta_1-1}}+1\Big)-a_2\varepsilon R|x|^{2}\\
&\qquad- \sum_{i=1}^k \alpha_i|z_i|^2+\frac{d}{2}\sum_{i=1}^k \alpha_i.
\end{align*}
By taking $\varepsilon$ sufficiently small, it follows that
\begin{align}
&\L_{m,0}\big(H+ m \varepsilon\la x,v\ra\big) \notag\\
&\le -c\varepsilon R|x|^2  - c\sum_{i=1}^k |z_i|^2 +C\varepsilon R m |v|^2+ C\varepsilon R\frac{1}{|x|^{\beta_1-1}}+CR,\label{ineq:L_0(H+<x,v>-<x,v>/|x|)}
\end{align}
for some positive constants $c,\,C$ independent of $\varepsilon$ and $R$. 

Next, we consider the cross term $\la v,z_1\ra$ on the right-hand side of \eqref{form:V_2:gamma=0}. It\^o's formula yields (recalling $\gamma=0$)
\begin{align} \label{eqn:L_0<v,z_1>}
&\L_{m,0} \big(\varepsilon m\la v,z_1\ra\big) \notag\\
 & = -\varepsilon\la\grad \U(x)+\grad \G(x),z_1\ra+\varepsilon\sum_{i=1}^k \lambda_i\la z_i,z_1\ra-\varepsilon\alpha_1m\la v,z_1\ra-\lambda_1\varepsilon m|v|^2.
\end{align}
We invoke condition~\eqref{cond:U:U'(x)=O(x^lambda)} with Cauchy-Schwarz inequality to infer
\begin{align*}
-\varepsilon\la\grad \U(x),z_1\ra \le \varepsilon a_1(1+|x|)|z_1| \le a_1\varepsilon^{3/2}|x|^2+a_1\varepsilon^{1/2}|z_1|^2 +a_1\varepsilon|z_1| .
\end{align*}
Similarly,
\begin{align*}
\varepsilon\sum_{i=1}^k \lambda_i\la z_i,z_1\ra-\varepsilon\alpha_1m\la v,z_1\ra\le \varepsilon\sum_{i=1}^k \lambda_i^2|z_i|^2+\varepsilon k|z_1|^2+\varepsilon^{1/2}\alpha_1^2m|z_1|^2+\varepsilon^{3/2}m|v|^2.
\end{align*}
Concerning the cross term $\la \grad\G(x),z_1\ra$, recall from \eqref{cond:G:grad.G(x)<1/|x|^beta} that
\begin{align*}
|\grad\G(x)|\le \frac{a_1}{|x|^{\beta_1}}+a_1.
\end{align*}
As a consequence,
\begin{align*}
-\varepsilon\la \grad\G(x),z_1\ra &\le \varepsilon a_1\frac{|z_1|}{|x|^{\beta_1}} +\varepsilon a_1|z_1|\le \varepsilon a_1\frac{|z_1|}{|x|^{\beta_1}} +\frac{1}{2} \varepsilon (|z_1|^2+a_1^2).
\end{align*}
Altogether, we deduce that the bound
\begin{align}  
&\L_{m,0} \big(R^2\varepsilon m\la v,z_1\ra\big) \notag\\
&\le  - \lambda_1\varepsilon R^2 m|v|^2+C\varepsilon R^2  \frac{|z_1|}{|x|^{\beta_1}}  +C\varepsilon^{3/2} R^2\big(|v|^2 +|x|^{2}\big)+C\varepsilon^{1/2} R^2\sum_{i=1}^k |z_i|^2+C\varepsilon R^2, \label{ineq:L_0<v,z_1>}
\end{align}
holds for some positive constant $C$ independent of $R$ and $\varepsilon$.

Turning to the last term on the right-hand side of \eqref{form:V_2:gamma=0}, a routine calculation together with \eqref{eqn:L.m<x,v>/|x|} gives
\begin{align}
&\L_{m,0}\bigg(\!-\varepsilon m\frac{\la x,v\ra}{|x|}\sqrt{Q_R}    \bigg) \notag\\
 &=\varepsilon \bigg(-m\frac{|v|^2|x|^2-|\la x,v\ra|^2}{|x|^3}+ \frac{\la\grad U(x)-\sum_{i=1}^k\lambda_i z_i, x\ra}{|x|}+\frac{\la \grad \G(x),x\ra}{|x|}\bigg)\sqrt{Q_R} \notag\\
 &\qquad - \varepsilon m\frac{\la x,v\ra}{|x|}\cdot\frac{\la v,\sum_{i=1}^k\lambda_i z_i\ra}{\sqrt{Q_R} }\notag  +\varepsilon R^6 m\frac{\la x,v\ra}{|x|}\cdot\frac{  \alpha_1|z_1|^2+\lambda_1\la v,z_1\ra}{\sqrt{Q_R} }\notag\\
 &\qquad +\frac{1}{2}\alpha_1  \varepsilon R^6 m\frac{\la x,v\ra}{|x|}\cdot\frac{1}{\sqrt{Q_R} } \Big(  d- \frac{R^6|z_1|^2}{Q_R}\Big)\notag\\
 &=I_1-I_2+I_3+I_4. \label{eqn:L_0(<v,x>/|x|.sqrt(Q_R)}
\end{align}
Concerning $I_4$, we recall from \eqref{form:Q_R} that
\begin{equation}
\label{eq: QR1}
Q_R\ge R^6|z_1|^2+m|v|^2,
\end{equation}
whence
\begin{align*}
I_4\le \frac{1}{2}\alpha_1  \varepsilon R^6 d \sqrt{m} .
\end{align*}
With regard to $I_1$, we invoke conditions \eqref{cond:U:U'(x)=O(x^lambda)} and \eqref{cond:G:G<1/|x|^beta} to see that
\begin{align*}
Q_R\le R^6|z_1|^2+m|v|^2+a_1\Big(1+|x|^2+\frac{1}{|x|^{\beta_1}}+|x|\Big)+R^2.
\end{align*} 
As a consequence,
\begin{align*}
\varepsilon\frac{\la\grad U(x)-\sum_{i=1}^k\lambda_i z_i, x\ra}{|x|}\sqrt{Q_R} 
&\le \varepsilon\Big[a_1(1+|x|)+\sum_{i=1}^k \lambda_i |z_i|\Big]\sqrt{Q_R}\\
&\le \frac{1}{2}\varepsilon\Big[a_1(1+|x|)+\sum_{i=1}^k \lambda_i |z_i|\Big]^2+\frac{1}{2}\varepsilon Q_R\\
&\le C\varepsilon \Big(R^6\sum_{i=1}^k |z_i|^2+|v|^2+|x|^2+\frac{1}{|x|^{\beta_1}}+R^2\Big).
\end{align*}
On the other hand, estimate \eqref{cond:G:<grad.G(x),x>/|x|} implies the bound
\begin{align*}
\varepsilon\frac{\la \grad \G(x),x\ra}{|x|}\sqrt{Q_R} 
&\le \varepsilon\Big(-\frac{a_4}{2|x|^{\beta_1}}+C\Big)\sqrt{Q_R}\\
&= -\varepsilon\frac{a_4}{2}\cdot\frac{1}{|x|^{\beta_1}}\sqrt{Q_R}+C\varepsilon\sqrt{Q_R}\\
&\le -c \varepsilon \frac{R^3|z_1|+R}{|x|^{\beta_1}}+C\varepsilon \Big(R^6 |z_1|^2+|v|^2+|x|^2+\frac{1}{|x|^{\beta_1}}+R^2\Big).
\end{align*}
It follows that
\begin{align*}
I_1&\le-c \varepsilon \frac{R^3|z_1|+R}{|x|^{\beta_1}}+C\varepsilon \Big(R^6\sum_{i=1}^k |z_i|^2+|v|^2+|x|^2+\frac{1}{|x|^{\beta_1}}+R^2\Big).
\end{align*}
With regard to $I_2$ on the right-hand side of \eqref{eqn:L_0(<v,x>/|x|.sqrt(Q_R)}, since $Q_R\ge m|v|^2$, we find
\begin{align*}
-I_2=-\varepsilon m\frac{\la x,v\ra}{|x|}\cdot\frac{\la v,\sum_{i=1}^k\lambda_i z_i\ra}{\sqrt{Q_R} }\le \varepsilon\sqrt{m}|v|\sum_{i=1}^k \lambda_i|z_i|\le C\varepsilon\Big(|v|^2+\sum_{i=1}^k |z_i|^2\Big).
\end{align*}
Turning to $I_3$, we estimate as follows:
\begin{align*}
I_3&=\varepsilon R^6 m\frac{\la x,v\ra}{|x|}\cdot\bigg(\frac{  \alpha_1|z_1|^2}{\sqrt{Q_R} }+\frac{  \lambda_1\la v,z_1\ra}{\sqrt{Q_R} }\bigg)\\
&\le \varepsilon  R^3 m |v|\cdot \alpha_1|z_1|+\varepsilon R^6\sqrt{m}|v|\cdot \lambda_1|z_1|\\
&\le C\varepsilon(|v|^2+R^{12}|z_1|^2),
\end{align*}
where in the first inequality we have used $\sqrt{Q_R}\geq R^3 |z_1|$ (which follows from \eqref{eq: QR1}) and $\sqrt{Q_R}\geq \sqrt{m}|v|$.
Now, we collect the estimates on $I_j$, $j=1,\dots,4$, together with expression \eqref{eqn:L_0(<v,x>/|x|.sqrt(Q_R)} to infer (recalling $\varepsilon<1<R$)
\begin{align} \label{ineq:L_0(<v,x>/|x|.sqrt(Q_R)}
&\L_{m,0}\bigg(\!-\varepsilon m\frac{\la x,v\ra}{|x|}\sqrt{Q_R}    \bigg) \notag\\
&\le -c \varepsilon \frac{R^3|z_1|+R}{|x|^{\beta_1}}+C\varepsilon \Big(R^{12}\sum_{i=1}^k |z_i|^2+|v|^2+|x|^2+\frac{1}{|x|^{\beta_1}}+R^6\Big) \nt \\
&\le  -c \varepsilon \frac{R^3|z_1|+R}{|x|^{\beta_1}}+C\varepsilon \Big(R^{12}\sum_{i=1}^k |z_i|^2+|v|^2+|x|^2+R^6\Big).
\end{align}
In the last implication above, we subsumed $C\varepsilon|x|^{-\beta_1}$ into $-c\varepsilon R|x|^{-\beta_1}$, by taking $R$ large enough.

Turning back to $V_2$ given by \eqref{form:V_2:gamma=0}, we combine \eqref{ineq:L_0(H+<x,v>-<x,v>/|x|)}, \eqref{ineq:L_0<v,z_1>} and \eqref{ineq:L_0(<v,x>/|x|.sqrt(Q_R)} to arrive at the estimate
\begin{align*}
\L_{m,0} V_2 &\le -c\varepsilon R|x|^2  - c\sum_{i=1}^k |z_i|^2 +C\varepsilon R  |v|^2+ C\varepsilon R\frac{1}{|x|^{\beta_1-1}}+CR\\
&\qquad - c\varepsilon R^2 |v|^2+C\varepsilon R^2  \frac{|z_1|}{|x|^{\beta_1}}  +C\varepsilon^{3/2} R^2\big(|v|^2 +|x|^{2}\big)+C\varepsilon^{1/2} R^2\sum_{i=1}^k |z_i|^2+C\varepsilon R^2\\
&\qquad -c \varepsilon \frac{R^3|z_1|+R}{|x|^{\beta_1}}+C\varepsilon \Big(R^{12}\sum_{i=1}^k |z_i|^2+|v|^2+|x|^2+R^6\Big),
\end{align*}
i.e.,
\begin{align*} 
\L_{m,0} V_2 & \le -(c\varepsilon R^2 -C\varepsilon R-C\varepsilon^{3/2}R^2-C\varepsilon)  |v|^2 - (c-C\varepsilon^{1/2}R^2-C\varepsilon R^{12})\sum_{i=1}^k |z_i|^2\nt \\
&\qquad - ( c\varepsilon R-C\varepsilon^{3/2}R^2-C\varepsilon)|x|^2- c\varepsilon R \frac{1}{|x|^{\beta_1}}-(c\varepsilon R^3-c\varepsilon R^2)\frac{|z_1|}{|x|^{\beta_1}} \nt \\
&\qquad+C\varepsilon R\frac{1}{|x|^{\beta_1-1}}+CR+C\varepsilon R^{6}.
\end{align*}
Since $c,C$ are independent of $\varepsilon$ and $R$, we may take $R$ sufficiently large and then shrink $\varepsilon$ to zero while making use of the fact that $\frac{1}{|x|^{\beta_1-1}}$ can be subsumed into $|x|^2+|x|^{\beta_1}$. It follows that
\begin{align} \label{ineq:L_0.V_2:gamma=0}
\L_{m,0} V_2 & \le -c\varepsilon R^2|v|^2 - c\sum_{i=1}^k |z_i|^2-  c\varepsilon R|x|^2- c\varepsilon R \frac{1}{|x|^{\beta_1}}+C\varepsilon R^{6}.
\end{align}
This produces the Lyapunov property of $V_2$ for \eqref{eqn:GLE} in the case $\gamma=0$, thereby finishing the proof.

\end{proof}

\subsection{N-particle GLEs} \label{sec:ergodicity:N-particle}

We now turn our attention to the full system \eqref{eqn:GLE:N-particle} and construct Lyapunov functions for \eqref{eqn:GLE:N-particle} based on the discussion of the single-particle GLE in Section \ref{sec:ergodicity:single-particle}.

We start with the case $\gamma>0$ and observe that there is a natural generalization of the function $V_1$ defined in \eqref{form:V:gamma>0} to an arbitrary number of particles $N\ge 2$ in an arbitrary number of dimensions $d\ge1$. More specifically, for $\varepsilon>0$, we introduce the following function $V^1_N$ given by
\begin{align} \label{form:V_N^1}
&V_N^1( \xb,\vb,\zb_1,\dots,\zb_N)  \nt \\
&=H_N(\xb,\vb,\zb_1,\dots,\zb_N)+ \varepsilon  m\la \xb,\vb\ra-\varepsilon m \sum_{i=1}^N \Big\la v_i,\sum_{j\neq i}\frac{ x_i-x_j}{|x_i-x_j|}\Big\ra,
\end{align}
where $H_N$ is the Hamiltonian as in \eqref{form:H_N}. In Lemma \ref{lem:Lyapunov:V_N^1}, stated and proven next, we assert that one may pick $\varepsilon$ sufficiently small to ensure the Lyapunov property of $V^1_N$.

\begin{lemma} \label{lem:Lyapunov:V_N^1}
Under Assumption \ref{cond:U} and Assumption \ref{cond:G}, let $V_N^1$ be the function defined as \eqref{form:V_N^1}. For all $\gamma> 0$ and $m>0$, there exists a positive constant $\varepsilon$ sufficiently small such that $V_N^1$ is a Lyapunov function for \eqref{eqn:GLE:N-particle}.
\end{lemma}
\begin{proof}
We first consider the Hamiltonian $H_N$ given by \eqref{form:H_N}. Applying $\L_{m,\gamma}^N$ as in \eqref{form:L_N} to $H_N$ gives
\begin{align} \label{eqn:L_gamma^N.H_N}
\L_{m,\gamma}^N H_N= -\gamma|\vb|^2-\sum_{i=1}^N \sum_{\ell=1}^{k_i}\alpha_{i,\ell}|z_{i,\ell}|^2+\frac{1}{2m}\gamma N d+\frac{1}{2}\sum_{i=1}^N\sum_{\ell=1}^{k_i}\alpha_{i,\ell}.
\end{align}
Next, a routine calculation on the cross term $\la \xb,\vb\ra$ gives
\begin{align} \label{eqn:L_gamma^N.<x,v>}
\L_{m,\gamma}^N \big(\varepsilon m\la \xb,\vb\ra\big) &=\varepsilon m|\vb|^2-\varepsilon\gamma\la \xb,\vb\ra+\varepsilon\sum_{i=1}^N \Big\la x_i, \sum_{\ell=1}^{k_i}z_{i,\ell}\Big\ra\\
& \qquad -\varepsilon\sum_{i=1}^N \la x_i,\grad U(x_i)\ra-\varepsilon\close\sum_{1\le i<j\le N}\close\la x_i-x_j,\grad G(x_i-x_j)\ra. \notag
\end{align}
Recalling \eqref{cond:U:x.U'(x)>-x^(lambda+1)}, we readily have
\begin{align*}
-\sum_{i=1}^N \la x_i,\grad U(x_i)\ra \le -a_2\sum_{i=1}^N |x_i|^{\lambda+1}+Na_3.
\end{align*}
Also, from \eqref{cond:G:grad.G(x)<1/|x|^beta}, it holds that
\begin{align*}
-\sum_{1\le i<j\le N}\la x_i-x_j,\grad G(x_i-x_j)\ra \le a_1\Big(N^2+\sum_{1\le i<j\le N}\frac{1}{|x_i-x_j|^{\beta_1-1}}\Big).
\end{align*}
Concerning the other cross terms on the right-hand side of \eqref{eqn:L_gamma^N.<x,v>} we invoke Cauchy-Schwarz inequality to infer
\begin{align*}
-\varepsilon\gamma\la \xb,\vb\ra+\varepsilon\sum_{i=1}^N \Big\la x_i, \sum_{\ell=1}^{k_i}z_{i,\ell}\Big\ra
&\le  C\Big( \varepsilon^{1/2}|\vb|^2+\varepsilon^{1/2}\sum_{i=1}^N|\zb_i|^2+\varepsilon^{3/2}|\xb|^2\Big)
\end{align*}
for some positive constant $C$ independent of $\varepsilon$. We now collect the above estimates together with expression \eqref{eqn:L_gamma^N.<x,v>} to obtain
\begin{align}
\L_{m,\gamma}^N\big( \varepsilon m\la \xb,\vb\ra\big) &\le C\Big( \varepsilon^{1/2}|\vb|^2+\varepsilon^{1/2}\sum_{i=1}^N|\zb_i|^2+\varepsilon^{3/2}|\xb|^2+\varepsilon\close\sum_{1\le i<j\le N}\frac{1}{|x_i-x_j|^{\beta_1-1}}\Big) \notag\\
&\qquad  -a_2\varepsilon\sum_{i=1}^N |x_i|^{\lambda+1} +C. \label{ineq:L_gamma^N.<x,v>}
\end{align}

Next, we turn to the last term on the right-hand side of \eqref{form:V_N^1}. A routine calculation produces
\begin{align} \label{eqn:L_gamma^N.<v,(x_i-x_j)/|x_i-x_j|>}
&\L_{m,\gamma}^N\Big(-m \sum_{i=1}^N \Big\la v_i,\sum_{j\neq i}\frac{ x_i-x_j}{|x_i-x_j|}\Big\ra\Big) \\
&= -m\close\sum_{1\le i<j\le N} \frac{|v_i-v_j|^2}{|x_i-x_j|} +m\close\sum_{1\le i<j\le N} \frac{|\la v_i-v_j,x_i-x_j\ra|^2}{|x_i-x_j|^3} \notag \\
&\qquad+\gamma\close \sum_{1\le i<j\le N} \frac{\la v_i-v_j,x_i-x_j\ra}{|x_i-x_j|} +\sum_{1\le i<j\le N} \frac{\la \grad U(x_i) -\grad U(x_j),x_i-x_j\ra}{|x_i-x_j|}           \notag\\
&\qquad +\sum_{i=1}^N\Big\la  \sum_{j\neq i}\frac{x_i-x_j}{|x_i-x_j|}  ,\sum_{\ell\neq i}\grad G(x_i-x_\ell)    \Big\ra -\sum_{i=1}^N\Big\la  \sum_{j\neq i}\frac{x_i-x_j}{|x_i-x_j|}  ,\sum_{\ell=1}^{k_i} \lambda_{i,\ell}z_{i,\ell}    \Big\ra .    \notag
\end{align}
We proceed to estimate the above right-hand side while making use of Cauchy-Schwarz inequality. It is clear that
\begin{align*}
-m\close\sum_{1\le i<j\le N} \frac{|v_i-v_j|^2}{|x_i-x_j|} +m\close\sum_{1\le i<j\le N} \frac{|\la v_i-v_j,x_i-x_j\ra|^2}{|x_i-x_j|^3}\le 0,
\end{align*}
which is negligible. Also,
\begin{align*}
&\gamma\close \sum_{1\le i<j\le N} \frac{\la v_i-v_j,x_i-x_j\ra}{|x_i-x_j|}  -\sum_{i=1}^N\Big\la  \sum_{j\neq i}\frac{x_i-x_j}{|x_i-x_j|}  ,\sum_{\ell=1}^{k_i} \lambda_{i,\ell}z_{i,\ell}    \Big\ra\\
&\le \gamma (N-1)\sum_{i=1}^N |v_i|+(N-1)\sum_{i=1}^N\sum_{\ell=1}^{k_i} \lambda_{i,\ell}|z_{i,\ell}|  .
\end{align*}
Concerning the cross terms involving $\grad U$, we invoke condition \eqref{cond:U:U'(x)=O(x^lambda)} and obtain
\begin{align*}
\sum_{1\le i<j\le N} \frac{\la \grad U(x_i) -\grad U(x_j),x_i-x_j\ra}{|x_i-x_j|}      
&\le (N-1)a_1\Big(N+\sum_{i=1}^N|x_i|^\lambda\Big). 
\end{align*}
With regard to the cross terms involving $\grad G$, we recast them as follows:
\begin{align*}
\sum_{i=1}^N\Big\la  \sum_{j\neq i}\frac{x_i-x_j}{|x_i-x_j|}  ,\sum_{\ell\neq i}\grad G(x_i-x_\ell)    \Big\ra
&=-a_4\sum_{i=1}^N\Big\la  \sum_{j\neq i}\frac{x_i-x_j}{|x_i-x_j|}  ,\sum_{\ell\neq i}\frac{x_i-x_\ell}{|x_i-x_\ell|^{\beta_1+1}} \Big\ra\\
&+ \sum_{i=1}^N\Big\la  \sum_{j\neq i}\frac{x_i-x_j}{|x_i-x_j|}  ,\sum_{\ell\neq i}\grad G(x_i-x_\ell)+ a_4\frac{x_i-x_\ell}{|x_i-x_\ell|^{\beta_1+1}}   \Big\ra.
\end{align*}
In view of Lemma~\ref{lem:|x_i-x_i|}, we readily have
\begin{align*}
-a_4\sum_{i=1}^N\Big\la  \sum_{j\neq i}\frac{x_i-x_j}{|x_i-x_j|}  ,\sum_{\ell\neq i}\frac{x_i-x_\ell}{|x_i-x_\ell|^{\beta_1+1}} \Big\ra \le -2a_4 \sum_{1\le i<j\le N}\frac{1}{|x_i-x_j|^{\beta_1}}.
\end{align*}
On the other hand, condition~\eqref{cond:G:|grad.G(x)+q/|x|^beta_1|<1/|x|^beta_2} implies the bound
\begin{align*}
&  \sum_{i=1}^N\Big\la  \sum_{j\neq i}\frac{x_i-x_j}{|x_i-x_j|}  ,\sum_{\ell\neq i}\grad G(x_i-x_\ell)+ a_4\frac{x_i-x_\ell}{|x_i-x_\ell|^{\beta_1+1}}   \Big\ra\\
&\qquad\le (N-1)\sum_{i=1}^N\sum_{\ell\neq i}\Big|\grad G(x_i-x_\ell)+ a_4\frac{x_i-x_\ell}{|x_i-x_\ell|^{\beta_1+1}}   \Big|\\
&\qquad\le (N-1)\Big[ 2\sum_{1\le i<\ell\le N}\frac{a_5}{|x_i-x_\ell|^{\beta_2}}+N(N-1)a_6\Big].
\end{align*}
In the above, $a_4,a_5$ and $a_6$ are the constants as in condition \eqref{cond:G:|grad.G(x)+q/|x|^beta_1|<1/|x|^beta_2}. Since $\beta_2\in[0,\beta_1)$, we observe that $|x_i-x_\ell|^{-\beta_2}$ can be subsumed into $-|x_i-x_\ell|^{-\beta_1}$. It follows that
\begin{align} \label{ineq:sum<x_i-x_j,grad.G(x_i-x_j)>}
\sum_{i=1}^N\Big\la  \sum_{j\neq i}\frac{x_i-x_j}{|x_i-x_j|}  ,\sum_{\ell\neq i}\grad G(x_i-x_\ell)    \Big\ra 
 &\le -a_4 \sum_{1\le i<j\le N}\frac{1}{|x_i-x_j|^{\beta_1}}+C.
\end{align}
From the identity \eqref{eqn:L_gamma^N.<v,(x_i-x_j)/|x_i-x_j|>}, we infer the estimate
\begin{align} \label{ineq:L_gamma^N.<v,(x_i-x_j)/|x_i-x_j|>}
&\L_{m,\gamma}^N\Big(- \varepsilon m\sum_{i=1}^N \Big\la v_i,\sum_{j\neq i}\frac{ x_i-x_j}{|x_i-x_j|}\Big\ra\Big) \\
&\le   -a_4\varepsilon\close \sum_{1\le i<j\le N}\frac{1}{|x_i-x_j|^{\beta_1}}+C\varepsilon\Big(1+\sum_{i=1}^N |v_i|+ \sum_{i=1}^N |x_i|^\lambda+\sum_{i=1}^N\sum_{\ell=1}^{k_i} |z_{i,\ell}| \Big)  . \nt 
\end{align}

Now, we collect \eqref{eqn:L_gamma^N.H_N}, \eqref{ineq:L_gamma^N.<x,v>}, \eqref{eqn:L_gamma^N.<v,(x_i-x_j)/|x_i-x_j|>} together with the expression \eqref{form:V_N^1} of $V_N^1$ and deduce
\begin{align}
\L_{m,\gamma}^N V_N^1 &\le -\gamma|\vb|^2-\sum_{i=1}^N \sum_{\ell=1}^{k_i}\alpha_{i,\ell}|z_{i,\ell}|^2 -a_2\varepsilon\sum_{i=1}^N |x_i|^{\lambda+1} -a_4\varepsilon\close \sum_{1\le i<j\le N}\frac{1}{|x_i-x_j|^{\beta_1}}+C\\
&\qquad + C\Big( \varepsilon^{1/2}|\vb|^2+\varepsilon^{1/2}\sum_{i=1}^N|\zb_i|^2+\varepsilon^{3/2}|\xb|^2+\varepsilon\close\sum_{1\le i<j\le N}\frac{1}{|x_i-x_j|^{\beta_1-1}}\Big) \notag\\
&\qquad+C\varepsilon\Big(\sum_{i=1}^N |v_i|+ \sum_{i=1}^N |x_i|^\lambda+\sum_{i=1}^N\sum_{\ell=1}^{k_i} |z_{i,\ell}| \Big) .\nt
\end{align}
In the above, we emphasize that $C$ is a positive constant independent of $\varepsilon$. Finally, by taking $\varepsilon$ sufficiently small, we may infer
\begin{align*}
\L_{m,\gamma}^N V_N^1 &\le -\frac{1}{2}\Big(\gamma|\vb|^2+\sum_{i=1}^N \sum_{\ell=1}^{k_i}\alpha_{i,\ell}|z_{i,\ell}|^2 +a_2\varepsilon\sum_{i=1}^N |x_i|^{\lambda+1} +a_4\varepsilon\close \sum_{1\le i<j\le N}\frac{1}{|x_i-x_j|^{\beta_1}}\Big)+C.
\end{align*}
This produces the desired Lyapunov property of $V_N^1$ for system~\eqref{eqn:GLE:N-particle} in the case $\gamma>0$, as claimed. 
\end{proof}

Turning to the case $\gamma=0$, analogous to the function $V_2$ defined in \eqref{form:V_2:gamma=0} for the single-particle system \eqref{eqn:GLE}, for $\varepsilon\in(0,1)$ and $R>1$, we introduce the function $V_N^2$ given by
\begin{align} \label{form:V_N^2}
V_N^2( \xb,\vb,\zb_1,\dots,\zb_N) 
&=H_N(\xb,\vb,\zb_1,\dots,\zb_N)+ \varepsilon Rm\la \xb,\vb\ra +\varepsilon R^2m \sum_{i=1}^N \la v_i,z_{i,1}\ra\\
&\quad-\varepsilon \Big(\sum_{i=1}^N m\Big\la v_i,\sum_{j\neq i}\frac{ x_i-x_j}{|x_i-x_j|}\Big\ra\Big)\sqrt{Q_R^N},\notag
\end{align}
where
\begin{align} \label{form:Q_R^N}
Q_R^N(\xb,\vb,\zb_1,\dots,\zb_N)=R^6\sum_{i=1}^N|z_{i,1}|^2+m|\vb|^2+2\sum_{i=1}^N U(x_i)+2\close \sum_{1\le i<j\le N}\close G(x_i-x_j)+R^2.
\end{align}
In Lemma \ref{lem:Lyapunov:V_N^2} below, we prove that $V_N^2$ is indeed a Lyapunov function for \eqref{eqn:GLE:N-particle} in the case $\gamma=0$.

\begin{lemma} \label{lem:Lyapunov:V_N^2} Under Assumption \ref{cond:U} and Assumption \ref{cond:G}, let $V_N^2$ be the function defined in \eqref{form:V_N^2}. In the case $\gamma= 0$, for every $m>0$, there exist positive constants $\varepsilon$ small and $R$ large enough such that $V_N^2$ is a Lyapunov function for \eqref{eqn:GLE:N-particle}.
\end{lemma}
\begin{proof}
We first consider the Hamiltonian $H_N$ as in \eqref{form:H_N}. Since $\gamma=0$, the identity \eqref{eqn:L_gamma^N.H_N} is reduced to 
\begin{align} \label{eqn:L_0^N.H_N}
\L_{m,0}^N H_N=-\sum_{i=1}^N \sum_{\ell=1}^{k_i}\alpha_{i,\ell}|z_{i,\ell}|^2+\frac{1}{2}\sum_{i=1}^N\sum_{\ell=1}^{k_i}\alpha_{i,\ell}.
\end{align}

With regard to the cross term $\la \xb,\vb\ra$, 
we note that since $\gamma=0$, by virtue of Assumption \ref{cond:U}, $\lambda=1$. 
We employ an argument similarly to that of the estimate \eqref{ineq:L_gamma^N.<x,v>} and obtain the bound
\begin{align}
\L_{m,0}^N\big( \varepsilon Rm\la \xb,\vb\ra\big) &\le   -a_2\varepsilon R|\xb|^{2}+ C R\Big( \varepsilon|\vb|^2+\varepsilon^{1/2}\sum_{i=1}^N|\zb_i|^2+\varepsilon^{3/2}|\xb|^2+1\Big) \label{ineq:L_0^N.<x,v>} \\
&\qquad +C\varepsilon R\close\sum_{1\le i<j\le N}\frac{1}{|x_i-x_j|^{\beta_1-1}}.\nt
\end{align}
In the above, $C$ is a positive constant independent of $\varepsilon$ and $R$.

Concerning the cross terms $\la v_i,z_{i,1}\ra$, $i=1,\dots,N$, on the right-hand side of \eqref{form:V_N^2}, applying It\^o's formula gives
\begin{align*}
\L_{m,0}^N \big(m\la v_i,z_{i,1}\big)\ra &=\Big\la -\grad U(x_i)-\sum_{j\neq i}\grad G(x_i-x_j),z_{i,1}\Big\ra +\sum_{\ell=1}^{k_i}\lambda_{i,\ell}\la z_{i,\ell}, z_{i,1}\ra\\
&\qquad -\alpha_{i,1}m\la v_i,z_{i,1}\ra- \lambda_{i,1}m|v_i|^2.
\end{align*}
From the condition \eqref{cond:U:U'(x)=O(x^lambda)} ($\lambda=1$), we have
\begin{align*}
-\varepsilon\la \grad U(x_i),z_{i,1}\ra \le   a_1 \varepsilon(|x|+1)|z_{i,1}|\le C(\varepsilon^{3/2}|x_i|^2+\varepsilon^{1/2}|z_{i,1}|^2+\varepsilon). 
\end{align*}
In the last estimate above, we employed Cauchy-Schwarz inequality. Likewise,
\begin{align*}
\varepsilon\sum_{\ell=1}^{k_i}\lambda_{i,\ell}\la z_{i,\ell}, z_{i,1}\ra -\varepsilon\alpha_{i,1}m\la v_i,z_{i,1}\ra \le C\varepsilon^{1/2}|\zb_i|^2+C\varepsilon^{3/2}|v_i|^2.
\end{align*}
Also, condition \eqref{cond:G:grad.G(x)<1/|x|^beta} implies
\begin{align*}
-\sum_{i=1}^N\Big\la\sum_{j\neq i}\grad G(x_i-x_j),z_{i,1}\Big\ra &\le a_1\sum_{i=1}^N|z_{i,1}|\Big(2\close\sum_{1\le i<j\le N}\frac{1}{|x_i-x_j|^{\beta_1}}+N^2\Big)\\
&\le C\sum_{i=1}^N|z_{i,1}|\sum_{1\le i<j\le N}\frac{1}{|x_i-x_j|^{\beta_1}}+C\sum_{i=1}^N|\zb_{i}|^2+C.
\end{align*} 
It follows that for $\varepsilon\in (0,1)$ and $R>1$,
\begin{align}\label{ineq:L_0^N.<v,z>}
\L_{m,0}^N \Big(\varepsilon R^2m\sum_{i=1}^N\la v_i,z_{i,1}\ra\Big) &\le -\varepsilon R^2m \sum_{i=1}^N \lambda_{i,1}|v_i|^2+C R^2 \Big( \varepsilon^{3/2}( |\xb|^2+|\vb|^2)+\varepsilon^{1/2}\sum_{i=1}^N |\zb_i|^2+\varepsilon\Big) \nt \\
&\qquad +C\varepsilon R^2\Big(\sum_{i=1}^N|\zb_{i}|\sum_{1\le i<j\le N}\frac{1}{|x_i-x_j|^{\beta_1}}+\sum_{i=1}^N|\zb_{i}|^2+1\Big).
\end{align}
In the above, we emphasize again that the positive constant $C$ does not depend on $\varepsilon$ and $R$.

Next, we consider the last term on the right-hand side of \eqref{form:V_N^2}. Observe that
\begin{align*}
\sum_{i=1}^N \Big\la v_i,\sum_{j\neq i}\frac{ x_i-x_j}{|x_i-x_j|}\Big\ra = \sum_{1\le i<j\le N}\close\frac{\la v_i-v_j,x_i-x_j\ra}{|x_i-x_j|}.
\end{align*}
So,
\begin{align} \label{eqn:L_0^N.<v,(x_i-x_j)>sqrt(Q_R^N)}
&\L_{m,0}^N \Big(-\varepsilon\sum_{i=1}^N m\Big\la v_i,\sum_{j\neq i}\frac{ x_i-x_j}{|x_i-x_j|}\Big\ra\Big) \sqrt{Q_R^N}\Big)\nt\\
&=-\varepsilon \sqrt{Q_R^N}\,\L_{m,0}^N\Big( \sum_{i=1}^N m \Big\la v_i,\sum_{j\neq i}\frac{ x_i-x_j}{|x_i-x_j|}\Big\ra\Big)-\varepsilon m\close \sum_{1\le i<j\le N}\close\frac{\la v_i-v_j,x_i-x_j\ra}{|x_i-x_j|} \L_{m,0}^N\sqrt{Q_R^N}  \nt \\
&=I_1+I_2.
\end{align}
Concerning $I_1$, from the estimate \eqref{ineq:L_gamma^N.<v,(x_i-x_j)/|x_i-x_j|>} (with $\lambda=1$), we have
\begin{align*}
I_1&= \sqrt{Q_R^N}\,\L_{m,0}^N \Big(-\varepsilon\sum_{i=1}^N m \Big\la v_i,\sum_{j\neq i}\frac{ x_i-x_j}{|x_i-x_j|}\Big\ra\Big)  \\
&\le   -a_4\varepsilon \sqrt{Q_R^N}\close \sum_{1\le i<j\le N}\frac{1}{|x_i-x_j|^{\beta_1}}+C\varepsilon\Big(1+ |\vb|+ |\xb|+\sum_{i=1}^N\sum_{\ell=1}^{k_i} |z_{i,\ell}| \Big)\sqrt{Q_R^N}  . \nt 
\end{align*}
Recalling $Q_R^N$ given by \eqref{form:Q_R^N}
\begin{align*}
Q_R^N=R^6\sum_{i=1}^N|z_{i,1}|^2+m|\vb|^2+2\sum_{i=1}^N U(x_i)+2\close \sum_{1\le i<j\le N}\close G(x_i-x_j)+R^2,
\end{align*}
it is clear that
\begin{align*}
\sqrt{Q_R^N}\ge \Big(R^6\sum_{i=1}^N |z_{i,1}|^2+R^2\Big)^{1/2}\ge  c R^3\sum_{i=1}^N |z_{i,1}|+cR,
\end{align*}
whence
\begin{align*}
-a_4\varepsilon \sqrt{Q_R^N}\close \sum_{1\le i<j\le N}\frac{1}{|x_i-x_j|^{\beta_1}} \le -c\varepsilon R^3\sum_{i=1}^N |z_{i,1}|\close \sum_{1\le i<j\le N}\frac{1}{|x_i-x_j|^{\beta_1}} - c\varepsilon R \close \sum_{1\le i<j\le N}\frac{1}{|x_i-x_j|^{\beta_1}} .
\end{align*}
On the other hand, in view of conditions \eqref{cond:U:U'(x)=O(x^lambda)} and \eqref{cond:G:G<1/|x|^beta}, it holds that
\begin{align*}
\sqrt{Q_R^N} \le C\Big(R^3\sum_{i=1}^N |z_{i,1}|+\sqrt{m}|\vb|+|\xb|+\Big(\sum_{1\le i<j\le N}\frac{1}{|x_i-x_j|^{\beta_1}}\Big)^{1/2}+R\Big).
\end{align*}
It follows that
\begin{align*}
&\varepsilon\Big(1+|\vb|+ |\xb|+\sum_{i=1}^N\sum_{\ell=1}^{k_i} |z_{i,\ell}| \Big)\sqrt{Q_R^N} \\
& \le C\,\varepsilon \Big( R^6 \sum_{i=1}^N |\zb_i|^2+|\vb|^2+|\xb|^2+\close \sum_{1\le i<j\le N}\frac{1}{|x_i-x_j|^{\beta_1}}+R^2\Big).
\end{align*}
As a consequence,
\begin{align} \label{ineq:L_0^N.<v,(x_i-x_j)>sqrt(Q_R^N):I_1}
I_1&\le   -a_4\varepsilon \sqrt{Q_R^N}\close \sum_{1\le i<j\le N}\frac{1}{|x_i-x_j|^{\beta_1}}+C\varepsilon\Big(1+ |\vb|+ |\xb|+\sum_{i=1}^N\sum_{\ell=1}^{k_i} |z_{i,\ell}| \Big)\sqrt{Q_R^N} \nt \\
&\le -c\varepsilon R^3\sum_{i=1}^N |z_{i,1}|\close \sum_{1\le i<j\le N}\frac{1}{|x_i-x_j|^{\beta_1}} - c\varepsilon R \close \sum_{1\le i<j\le N}\frac{1}{|x_i-x_j|^{\beta_1}} \nt \\
&\qquad +C\,\varepsilon \Big( R^6 \sum_{i=1}^N |\zb_i|^2+|\vb|^2+|\xb|^2+\close \sum_{1\le i<j\le N}\frac{1}{|x_i-x_j|^{\beta_1}}+R^2\Big),
\end{align}
holds for some positive constants $c,\, C$ independent of $\varepsilon$ and $R$.

With regard to $I_2$ on the right-hand side of \eqref{eqn:L_0^N.<v,(x_i-x_j)>sqrt(Q_R^N)}, It\^o's formula yields the identity
\begin{align*}
\L_{m,0}^N\sqrt{Q_R^N}  &= \frac{1}{\sqrt{Q_R^N}}\sum_{i=1}^N\Big[ \Big\la v_i,\sum_{\ell=1}^{k_i}\lambda_{i,\ell} z_{i,\ell}\Big\ra-R^6\alpha_{i,1} |z_{i,1}|^2-R^6\la z_{i,\ell},v_i\ra+\tfrac{1}{2}\alpha_{i,1}d-\frac{1}{2Q_R^N} \alpha_{i,1}|z_{i,1}|^2\Big].
\end{align*}
Since $$\sum_{i=1}^N U(x_i)+\sum_{1\le i<j\le N}G(x_i-x_j)\ge 0,$$
we deduce the bound
\begin{align*}
|\L_{m,0}^N\sqrt{Q_R^N}| \le C R^6 \frac{|\vb|\sum_{i=1}^N |\zb_i|+\sum_{i=1}^N |\zb_i|^2+1}{(R^6\sum_{i=1}^N |\zb_i|^2+|\vb|^2+1)^{1/2}}  \le C R^6 \Big(\sum_{i=1}^N |\zb_i|+1\Big).
\end{align*}
It follows that $I_2$ satisfies
\begin{align*}
I_2 =-\varepsilon m\close \sum_{1\le i<j\le N}\close\frac{\la v_i-v_j,x_i-x_j\ra}{|x_i-x_j|} \L_{m,0}^N\sqrt{Q_R^N} 
&\le C\varepsilon\close \sum_{1\le i<j\le N}|v_i-v_j| \cdot R^6 \Big(\sum_{i=1}^N |\zb_i|+1\Big)  \\
&\le C\varepsilon R^6\sum_{i=1}^N|v_i| \Big(\sum_{i=1}^N |\zb_i|+1\Big),
\end{align*}
whence
\begin{align}\label{ineq:L_0^N.<v,(x_i-x_j)>sqrt(Q_R^N):I_2}
I_2&\le C\varepsilon \Big(|\vb|^2+ R^{12} \sum_{i=1}^N |\zb_i|^2+ R^{12}\big).
\end{align}
Now, we collect \eqref{ineq:L_0^N.<v,(x_i-x_j)>sqrt(Q_R^N):I_1}, \eqref{ineq:L_0^N.<v,(x_i-x_j)>sqrt(Q_R^N):I_2} together with \eqref{eqn:L_0^N.<v,(x_i-x_j)>sqrt(Q_R^N)} to arrive at the bound
\begin{align*}
&\L_{m,0}^N \Big(-\varepsilon\sum_{i=1}^N m\Big\la v_i,\sum_{j\neq i}\frac{ x_i-x_j}{|x_i-x_j|}\Big\ra\Big) \sqrt{Q_R^N}\Big)\\
&\le -c\varepsilon R^3\sum_{i=1}^N |z_{i,1}|\close \sum_{1\le i<j\le N}\frac{1}{|x_i-x_j|^{\beta_1}} - c\varepsilon R \close \sum_{1\le i<j\le N}\frac{1}{|x_i-x_j|^{\beta_1}} \nt \\
&\qquad +C\varepsilon \Big( R^6 \sum_{i=1}^N |\zb_i|^2+|\vb|^2+|\xb|^2+\close \sum_{1\le i<j\le N}\frac{1}{|x_i-x_j|^{\beta_1}}+R^2\Big)\\
&\qquad + C\varepsilon \Big(|\vb|^2+ R^{12} \sum_{i=1}^N |\zb_i|^2+ R^{12}\big),
\end{align*}
whence
\begin{align} \label{ineq:L_0^N.<v,(x_i-x_j)>sqrt(Q_R^N)}
&\L_{m,0}^N \Big(-\varepsilon\sum_{i=1}^N m\Big\la v_i,\sum_{j\neq i}\frac{ x_i-x_j}{|x_i-x_j|}\Big\ra\Big) \sqrt{Q_R^N}\Big) \nt\\
&\qquad\le  -c\varepsilon R^3\sum_{i=1}^N |z_{i,1}|\close \sum_{1\le i<j\le N}\frac{1}{|x_i-x_j|^{\beta_1}} - c\varepsilon R \close \sum_{1\le i<j\le N}\frac{1}{|x_i-x_j|^{\beta_1}} \nt \\
&\qquad\qquad +C\varepsilon \Big( R^{12} \sum_{i=1}^N |\zb_i|^2+|\vb|^2+|\xb|^2+\close \sum_{1\le i<j\le N}\frac{1}{|x_i-x_j|^{\beta_1}}+R^{12}\Big).
\end{align}
In the above, we emphasize that $c,C$ are independent of $\varepsilon$ and $R$.

Turning back to $V_N^2$ given by \eqref{form:V_N^2}, from the estimates \eqref{eqn:L_0^N.H_N}, \eqref{ineq:L_0^N.<x,v>}, \eqref{ineq:L_0^N.<v,z>} and \eqref{ineq:L_0^N.<v,(x_i-x_j)>sqrt(Q_R^N)}, we obtain
\begin{align*}
\L_{m,0}^N V_N^2 &\le -I_3+I_4,
\end{align*}
where
\begin{align*}
I_3&= -\sum_{i=1}^N \sum_{\ell=1}^{k_i}\alpha_{i,\ell}|z_{i,\ell}|^2 -a_2\varepsilon R|\xb|^{2} -\varepsilon R^2 m\sum_{i=1}^N \lambda_{i,1}|v_i|^2\\
&\qquad -c\varepsilon R^3\sum_{i=1}^N |z_{i,1}|\close \sum_{1\le i<j\le N}\frac{1}{|x_i-x_j|^{\beta_1}} - c\varepsilon R \close \sum_{1\le i<j\le N}\frac{1}{|x_i-x_j|^{\beta_1}} ,
\end{align*}
and
\begin{align*}
I_4&=\frac{1}{2}\sum_{i=1}^N\sum_{\ell=1}^{k_i}\alpha_{i,\ell} + C R\Big( \varepsilon|\vb|^2+\varepsilon^{1/2}\sum_{i=1}^N|\zb_i|^2+\varepsilon^{3/2}|\xb|^2+1\Big)\\
&\qquad+C\varepsilon R^2\Big(\sum_{i=1}^N|z_{i,1}|\sum_{1\le i<j\le N}\frac{1}{|x_i-x_j|^{\beta_1}}+\sum_{i=1}^N|\zb_{i}|^2+1\Big)\\
&\qquad +C\varepsilon R\close\sum_{1\le i<j\le N}\frac{1}{|x_i-x_j|^{\beta_1-1}}+C R^2 \Big( \varepsilon^{3/2}( |\xb|^2+|\vb|^2)+\varepsilon^{1/2}\sum_{i=1}^N |\zb_i|^2+\varepsilon\Big) \nt \\
&\qquad  +C\varepsilon \Big( R^{12} \sum_{i=1}^N |\zb_i|^2+|\vb|^2+|\xb|^2+\close \sum_{1\le i<j\le N}\frac{1}{|x_i-x_j|^{\beta_1}}+R^{12}\Big).
\end{align*}
Since the constants $c,C$ are independent of $\varepsilon, R$, we may infer
\begin{align*}
\L_{m,0}^N V_N^2 &\le -c\sum_{i=1}^N |\zb_{i}|^2 -c\varepsilon R|\xb|^{2} -c\varepsilon R^2 |\vb|^2\\
&\qquad -c\varepsilon R^3\sum_{i=1}^N |z_{i,1}|\close \sum_{1\le i<j\le N}\frac{1}{|x_i-x_j|^{\beta_1}} - c\varepsilon R \close \sum_{1\le i<j\le N}\frac{1}{|x_i-x_j|^{\beta_1}} \\
&\qquad +C\varepsilon R(1+R\varepsilon^{1/2})|\vb|^2+ C\varepsilon(1+R^2\varepsilon^{1/2})|\xb|^2 +C\varepsilon^{1/2}R^{12}\sum_{i=1}^N |\zb_i|^2\\
 &\qquad  +C\varepsilon R^2\sum_{i=1}^N|z_{i,1}|\sum_{1\le i<j\le N}\frac{1}{|x_i-x_j|^{\beta_1}}+C\varepsilon\close\sum_{1\le i<j\le N}\frac{1}{|x_i-x_j|^{\beta_1}}+CR^{12}.
\end{align*}
Now, by first taking $R$ sufficiently large and then shrinking $\varepsilon$ small enough, we observe that all the positive (non-constant) terms on the above right-hand side are dominated by the negative terms. That is, the following holds
\begin{align*}
\L_{m,0}^N V_N^2 &\le -c\sum_{i=1}^N |\zb_{i}|^2 -c\varepsilon R|\xb|^{2} -c\varepsilon R^2 |\vb|^2 - c\varepsilon R \close \sum_{1\le i<j\le N}\frac{1}{|x_i-x_j|^{\beta_1}} +CR^{12}.
\end{align*}
This produces the desired Lyapunov property of $V_N^2$ for \eqref{eqn:GLE:N-particle} in the case $\gamma=0$. The proof is thus finished.

\end{proof}

\subsection{Proof of Theorem \ref{thm:ergodicity:N-particle}} \label{sec:ergodicity:proof-of-theorem}

The proof of Theorem \ref{thm:ergodicity:N-particle} is based on the Lyapunov functions constructed in Section~\ref{sec:ergodicity:N-particle} and a local minorization, cf. Definition \ref{def:minorization}, on the transition probabilities $P^m_t(X,\cdot)$. The latter property is summarized in the following auxiliary result.

\begin{lemma} \label{lem:minorization}
 Under Assumption \ref{cond:U} and Assumption \ref{cond:G}, system \eqref{eqn:GLE:N-particle} satisfies the minorization condition as in Definition \ref{def:minorization}.
\end{lemma}

The proof of Lemma \ref{lem:minorization} is relatively standard and can be found in literature for the Langevin dynamics \cite{herzog2017ergodicity,ottobre2011asymptotic}. For the sake of completeness, we briefly sketch the argument without going into details. 

\begin{proof}[Sketch of the proof of Lemma \ref{lem:minorization}]
First of all, by verifying a Hormander's condition, see \cite[page 1639]{ottobre2011asymptotic}, we note that the operator $\partial_t+\L_{m,\gamma}^N$ ($\gamma\ge 0$) is hypoelliptic \cite{hormander1967hypoelliptic}. This implies that the transition probabilities $P_t^m$ has a smooth density in $\X$. Furthermore, we may adapt the proof of \cite[Proposition 2.5]{herzog2017ergodicity} to study the control problem for \eqref{eqn:GLE:N-particle}. In particular, by the Stroock-Varadhan support Theorem, it can be shown that $\P_t^m(X,A)>0$ for $t>0,X\in\X$ and every open set $A\subset \X$. We then employ the same argument as those in \cite[Corollary 5.12]{herzog2017ergodicity} and \cite[Lemma 2.3]{mattingly2002ergodicity} to conclude the inequality \eqref{ineq:minorization}, thereby concluding the minorization.
\end{proof}

We are now in a position to conclude Theorem \ref{thm:ergodicity:N-particle}. The proof employs the
weak-Harris theorem proved in \cite[Theorem 1.2]{hairer2011yet}. 

\begin{proof}[Proof of Theorem \ref{thm:ergodicity:N-particle}] First of all, it is clear that $\pi_N$ defined in \eqref{form:pi_N} is an invariant measure for \eqref{eqn:GLE:N-particle}. Next, we observe that from Lemma \ref{lem:Lyapunov:V_N^1} and Lemma \ref{lem:Lyapunov:V_N^2}, the functions $V^1_N$ and $V^2_N$ respectively defined in \eqref{form:V_N^1} and \eqref{form:V_N^2} are Lyapunov functions for \eqref{eqn:GLE:N-particle} in the case $\gamma>0$ and $\gamma=0$. It follows that \cite[Assumption 1]{hairer2011yet} holds. On the other hand, Lemma \ref{lem:minorization} verifies \cite[Assumption 2]{hairer2011yet}. In view of \cite[Theorem 1.2]{hairer2011yet}, we conclude the uniqueness of $\pi_N$ as well as the exponential convergent rate \eqref{ineq:geometric-ergodicity}.

\end{proof}

\section{Small mass limit} \label{sec:small-mass}

We turn to the topic of the small mass limit for \eqref{eqn:GLE:N-particle}. In Section \ref{sec:small-mass:heuristic}, we provide a heuristic argument on how we derive the limiting system \eqref{eqn:GLE:N-particle:m=0} as $m\to 0$. We also outline the main steps of the proof of Theorem \ref{thm:small-mass:N-particle} in this section. In Section \ref{sec:small-mass:truncating}, we prove a partial result on the small mass limit assuming the nonlinearities are globally Lipschitz. In Section \ref{sec:small-mass:estimate-m=0}, we establish useful moment estimates on the limiting system \eqref{eqn:GLE:N-particle:m=0}. Lastly, in Section \ref{sec:small-mass:proof-of-theorem}, we establish Theorem \ref{thm:small-mass:N-particle} while making use of the auxiliary results from Section \ref{sec:small-mass:truncating} and Section \ref{sec:small-mass:estimate-m=0}.

\subsection{Heuristic argument for the limiting system \eqref{eqn:GLE:N-particle:m=0}} \label{sec:small-mass:heuristic}
In this subsection, we provide a heuristic argument detailing how we derive the limiting system~\eqref{eqn:GLE:N-particle:m=0} as well as its initial conditions from the original system \eqref{eqn:GLE:N-particle}. The argument draws upon recent works in~\cite{herzog2016small,nguyen2018small} where similar issues were dealt with in the absence of singular potentials. We formally set $m=0$ on the left-hand side of the $v_i-$equation in \eqref{eqn:GLE:N-particle} while substituting the $v_i(t)$ term on the right-hand side by $\d x_i(t)$ to obtain
\begin{align*}
\gamma \d x_i(t)=\Big[-\grad U(x_i(t))-\sum_{j\neq i}\grad G(x_i(t)-x_j(t))+\sum_{\ell=1}^{k_i}\lambda_{i,\ell} z_{i,\ell}(t)\Big]\d t+\sqrt{2\gamma}dW_0(t).
\end{align*}
Next, considering the $z_{i,\ell}-$equation in~\eqref{eqn:GLE:N-particle}, by Duhamel's formula, $z_{i,\ell}(t)$ may be written as
\begin{align} \label{form:Duhamel:z_i,ell:v_i}
z_{i,\ell}(t)=e^{-\alpha_{i,\ell} t}z_{i,\ell}(0)-\lambda_{i,\ell}\int_0^t e^{-\alpha_{i,\ell}(t-r)}v_i(r)\d r+\sqrt{2\alpha_{i,\ell}}\int_0^t e^{-\alpha_{i,\ell}(t-r)}\d W_{i,\ell}(r).
\end{align}
We note that the above expression still depends on $v_i(t)$. Nevertheless, this can be circumvented by employing an integration by parts as follows: 
\begin{align*}
\int_0^t e^{-\alpha_{i,\ell}(t-r)}v_i(r)\d r = x_i(t)-e^{-\alpha_{i,\ell} t}x_i(0)+\alpha_{i,\ell}\int_0^t e^{-\alpha_{i,\ell}(t-r)}x_i(r)\d r.
\end{align*}
Alternatively, we note that the above identity can be derived by applying It\^o's formula to $e^{\alpha_{i,\ell} t}x_i(t)$. Plugging back into \eqref{form:Duhamel:z_i,ell:v_i}, we find
\begin{align} \label{eqn:Duhamel:z_i,ell:x_i}
z_{i,\ell}(t)&=e^{-\alpha_{i,\ell} t}(z_{i,\ell}(0)+\lambda_{i,\ell} x_i(0))-\lambda_{i,\ell} x_i(t)-\lambda_{i,\ell}\alpha_{i,\ell}\int_0^t e^{-\alpha_{i,\ell}(t-r)}x_i(r)\d r \nt\\
&\qquad\qquad\qquad+\sqrt{2\alpha_{i,\ell}}\int_0^t e^{-\alpha_{i,\ell}(t-r)}\d W_{i,\ell}(r),
\end{align}
whence
\begin{align*}
z_{i,\ell}(t)+\lambda_{i,\ell} x_i(t)&=e^{-\alpha_{i,\ell} t}(z_{i,\ell}(0)+\lambda_{i,\ell} x_i(0))-\lambda_{i,\ell}\alpha_{i,\ell}\int_0^t e^{-\alpha_{i,\ell}(t-r)}x_i(r)\d r\\
&\qquad\qquad\qquad+\sqrt{2\alpha_{i,\ell}}\int_0^t e^{-\alpha_{i,\ell}(t-r)}\d W_{i,\ell}(r).
\end{align*}
Setting $f_{i,\ell}(t):=z_{i,\ell}(t)+\lambda_{i,\ell} x_i(t)$, we observe that (using Duhamel's formula again)
\begin{align*}
\d f_{i,\ell}(t) &= -\alpha_{i,\ell} f_{i,\ell}(t)+\lambda_{i,\ell}\, \alpha_{i,\ell}\, x_i(t)\d t+\sqrt{2\alpha_{i,\ell}}\d W_{i,\ell}(t),\quad \ell=1,\dots, k_i,\\
 f_{i,\ell}(0)&=z_{i,\ell}(0)+\lambda_{i,\ell} x_i(0).
\end{align*}
This together with setting $q_i(t):=x_i(t)$ deduces the limiting system \eqref{eqn:GLE:N-particle:m=0} as well as the corresponding shifted initial conditions as in Theorem \ref{thm:small-mass:N-particle}.

Next, for the reader's convenience, we summarize the idea of the proof of Theorem \ref{thm:small-mass:N-particle}. The argument
essentially consists of three steps as follows \cite{herzog2016small,lim2019homogenization,
lim2020homogenization}.

\textbf{Step 1}: we first truncate  the nonlinear potentials in \eqref{eqn:GLE:N-particle} and \eqref{eqn:GLE:N-particle:m=0} obtaining a truncated system whose coefficients are globally Lipschitz, and establish a convergence in probability in the small mass limit. This result appears in Proposition \ref{prop:small-mass:truncating} found in Section \ref{sec:small-mass:truncating}.

\textbf{Step 2}: Next, we establish an exponential moment bound on any finite time window for the limiting system \eqref{eqn:GLE:N-particle:m=0}. This is discussed in details in Section \ref{sec:small-mass:estimate-m=0}, cf. Lemma \ref{lem:GLE:N-particle:m=0:supnorm}.

\textbf{Step 3}: We prove Theorem \ref{thm:small-mass:N-particle} by removing the Lipschitz constraint from Proposition \ref{prop:small-mass:truncating} while making use of the energy estimates in Lemma \ref{lem:GLE:N-particle:m=0:supnorm}.

\subsection{Truncating \eqref{eqn:GLE:N-particle} and \eqref{eqn:GLE:N-particle:m=0}} \label{sec:small-mass:truncating} For $R>2$, let $\theta_R:[0,\infty)\to\rbb$ be a smooth function satisfying
\begin{align} \label{form:theta_R}
\theta_R(t) = \begin{cases} 
1,&  0\le t\le R,\\
\text{decreasing},& R\le t\le R+1,\\
0,& t\ge R+1.
\end{cases}
\end{align}
With the above cut-off $\theta_R$, we consider a truncating version of \eqref{eqn:GLE:N-particle} given by
\begin{align} \label{eqn:GLE:N-particle:truncating}
\d\, x_i(t) &= v_i(t)\d t,\qquad i=1,\dots,N, \nt  \\
m\d\, v_i(t) & = -\gamma v_i(t) \d t -\theta_R(|x_i(t)|)\grad \U(x_i(t))\d t +\sqrt{2\gamma} \,\d W_{i,0}(t)\\
&\qquad  - \sum_{j\neq i}\theta_R\big(|x_i(t)-x_j(t)|^{-1}\big)\grad \G\big(x_i(t)-x_j(t)\big) \d t +\sum _{\ell=1}^{k_i} \lambda_{i,\ell}  z_{i,\ell}(t)\d t, \nt \\ 
\d\, z_{i,\ell}(t) &= -\alpha_{i,\ell} z_{i,\ell} (t)\d t-\lambda_{i,\ell}  v_i(t)\d t+\sqrt{2\alpha_{i,\ell} }\,\d W_{i,\ell} (t),\quad \ell=1,\dots,k_i,\nt 
\end{align}
as well as the following truncated version of \eqref{eqn:GLE:N-particle:m=0}
\begin{align} \label{eqn:GLE:N-particle:m=0:truncating}
\gamma \d q_i(t) & =  -\theta_R(|q_i(t)|)\grad U(q_i(t))\d t- \sum_{j\neq i}\theta_R\big(|q_i(t)-q_j(t)|^{-1}\big)\grad \G\big(q_i(t)-q_j(t)\big) \d t \nt \\
&\qquad\qquad -\sum_{i=1}^{k_i} \lambda_{i,\ell}^2 q_i(t)\d t+\sum_{\ell=1}^{k_i}\lambda_{i,\ell} f_{i,\ell}(t)\d t+\sqrt{2\gamma}\d W_{i,0}(t), \nt \\
\d f_{i,\ell}(t) &= -\alpha_{i,\ell} f_{i,\ell}(t)+\lambda_{i,\ell}\, \alpha_{i,\ell}\, q_i(t)\d t+\sqrt{2\alpha_{i,\ell}}\d W_{i,\ell}(t),\quad \ell=1,\dots, k_i,\\
q_i(0)& = x_i(0),\quad f_{i,\ell}(0)=z_{i,\ell}(0)+\lambda_{i,\ell} x_i(0).\nt 
\end{align}

We now show that system \eqref{eqn:GLE:N-particle:truncating} can be approximated by \eqref{eqn:GLE:N-particle:m=0:truncating} on any finite time window in the small mass regime.

\begin{proposition} \label{prop:small-mass:truncating}
Under Assumption \ref{cond:U} and Assumption \ref{cond:G}, given any initial condition $(\xb(0)$, $\vb(0)$, $\zb_{1}(0)$,$\dots$, $\zb_N(0))\in \X$ and $R>2$, let $\big(\xb_m^R(t)$, $\vb_m^R(t)$, $\zb_{1,m}^R(t)$,$\dots$, $\zb_{N,m}^R(t)\big)$ and $\big(\qb^R(t)$, $\fb_{1}^R(t)$,$\dots$, $\fb_{N}^R(t)\big)$ respectively solve~\eqref{eqn:GLE:N-particle:truncating} and~\eqref{eqn:GLE:N-particle:m=0:truncating}. Then, for every $T>0$, the following holds 
\begin{align} \label{lim:small-mass:truncating}
\E\Big[\sup_{0\leq t\leq T}\big|\xb_m^R(t)-\qb^R(t)\big|^4\Big]\le m\cdot C,\quad\text{as}\quad m\rightarrow 0,
\end{align}
for some positive constant $C=C(T,R)$ independent of $m$.
\end{proposition}

In order to establish Proposition \ref{prop:small-mass:truncating}, it is crucial to derive useful moment bounds on the velocity process $\vb_m^R(t)$.  More precisely, we have the following result.

\begin{lemma}\label{lem:small-mass:truncating:m.v_m->0}
Under Assumption \ref{cond:U} and Assumption \ref{cond:G}, given any initial condition $(\xb(0)$, $\vb(0)$, $\zb_{1}(0)$,$\dots$, $\zb_N(0))\in \X$  and $R>2$, let $\big(\xb_m^R(t)$, $\vb_m^R(t)$, $\zb_{1,m}^R(t)$,$\dots$, $\zb_{N,m}^R(t)\big)$ be the solution of~\eqref{eqn:GLE:N-particle:truncating}. Then, for every $T>0$, $n>1$ and $\varepsilon>0$, it holds that 
\begin{align} \label{lim:small-mass:truncating:m.v_m->0}
m^{n}\E\Big[\sup_{0\leq t\leq T}\big|\vb_m^R(t)\big|^n\Big]\le m^{\frac{n}{2}-\varepsilon} C,\quad\text{as}\quad m\rightarrow 0,
\end{align}
for some positive constant $C=C(T,n,R,\varepsilon)$ independent of $m$.
\end{lemma}

For the sake of clarity, the proof of Lemma \ref{lem:small-mass:truncating:m.v_m->0} will be deferred to the end of this subsection. In what follows, we will assume Lemma \ref{lem:small-mass:truncating:m.v_m->0} holds and prove Proposition \ref{prop:small-mass:truncating}. The argument is adapted from the proof of \cite[Proposition 9]{nguyen2018small} tailored to our settings.

\begin{proof}[Proof of Proposition \ref{prop:small-mass:truncating}]
From the $(x_i,v_i)-$equations in \eqref{eqn:GLE:N-particle:truncating}, we find
\begin{align*}
m\d v_i^R(t)+\gamma \d x_i^R(t) &=  -\theta_R(x_i^R(t))\grad \U(x_i^R(t))\d t +\sqrt{2\gamma} \,\d W_{i,0}(t)\\
&\qquad  - \sum_{j\neq i}\theta_R\big(|x_i^R(t)-x_j^R(t)|^{-1}\big)\grad \G\big(x_i^R(t)-x_j^R(t)\big) \d t +\sum _{\ell=1}^{k_i} \lambda_{i,\ell}  z_{i,\ell}^R(t)\d t.
\end{align*}
Substituting $z_{i,\ell}^R(t)$ by the expression \eqref{eqn:Duhamel:z_i,ell:x_i} into the above equation  produces
\begin{align}
&m\d v_i^R(t)+\gamma \d x_i^R(t)  \nt\\
 &=  -\theta_R(|x_i^R(t)|)\grad \U(x_i^R(t))\d t +\sqrt{2\gamma} \,\d W_{i,0}(t) - \sum_{j\neq i}\theta_R\big(|x_i^R(t)-x_j^R(t)|^{-1}\big)\grad \G\big(x_i^R(t)-x_j^R(t)\big) \d t \nt \\
&\qquad+\sum_{\ell=1}^{k_i} \lambda_{i,\ell}e^{-\alpha_{i,\ell} t}\big[z_{i,\ell}(0)+ \lambda_{i,\ell} x_i(0)\big]-\Big(\sum_{\ell=1}^{k_i}\lambda_{i,\ell}^2\Big) x_i^R(t) \nt \\
&\qquad-\sum_{\ell=1}^{k_i}\lambda_{i,\ell}^2\alpha_{i,\ell}\int_0^t e^{-\alpha_{i,\ell}(t-r)}x_i^R(r)\d r +\sum_{\ell=1}^{k_i}\lambda_{i,\ell}\sqrt{2\alpha_{i,\ell}}\int_0^t e^{-\alpha_{i,\ell}(t-r)}\d W_{i,\ell}(r)\d t. \label{eqn:m.d.v_i^R+d.x_i^r}
\end{align} 
Similarly, from the $f_{i,\ell}-$equation of system \eqref{eqn:GLE:N-particle:m=0:truncating}, we have
\begin{align*}
f_{i,\ell}^R(t)&=e^{-\alpha_{i,\ell} t}(z_{i,\ell}(0)+\lambda_{i,\ell} x_i(0))-\lambda_{i,\ell}\alpha_{i,\ell}\int_0^t e^{-\alpha_{i,\ell}(t-r)}q_i^R(r)\d r\\
&\qquad\qquad\qquad+\sqrt{2\alpha_{i,\ell}}\int_0^t e^{-\alpha_{i,\ell}(t-r)}\d W_{i,\ell}(r).
\end{align*}
Plugging into $q_i-$equation of \eqref{eqn:GLE:N-particle:m=0:truncating} yields
\begin{align}
&\gamma \d q_i^R(t) \nt\\
 &=  -\theta_R(|q_i^R(t)|)\grad \U(q_i^R(t))\d t +\sqrt{2\gamma} \,\d W_{i,0}(t  - \sum_{j\neq i}\theta_R\big(|q_i^R(t)-q_j^R(t)|^{-1}\big)\grad \G\big(q_i^R(t)-q_j^R(t)\big) \d t \nt \\
&\qquad+\sum_{\ell=1}^{k_i} \lambda_{i,\ell}e^{-\alpha_{i,\ell} t}\big[z_{i,\ell}(0)+ \lambda_{i,\ell} x_i(0)\big]-\Big(\sum_{\ell=1}^{k_i}\lambda_{i,\ell}^2\Big) q_i^R(t) \nt \\
&\qquad-\sum_{\ell=1}^{k_i}\lambda_{i,\ell}^2\alpha_{i,\ell}\int_0^t e^{-\alpha_{i,\ell}(t-r)}x_i^R(r)\d r +\sum_{\ell=1}^{k_i}\lambda_{i,\ell}\sqrt{2\alpha_{i,\ell}}\int_0^t e^{-\alpha_{i,\ell}(t-r)}\d W_{i,\ell}(r)\d t. \label{eqn:d.q_i^R}
\end{align} 
Setting $\xbar_i^R=x_i^R-q_i^R$, we subtract~\eqref{eqn:d.q_i^R} from \eqref{eqn:m.d.v_i^R+d.x_i^r} to obtain the identity
\begin{align}
m\d v_i^R(t)+\gamma \d \xbar_i^R(t)  
 &= -\Big[\theta_R(|x_i^R(t)|)\grad \U(x_i^R(t))-\theta_R(|q_i^R(t)|)\grad \U(q_i^R(t))\Big]\d t \nt \\
 &\qquad - \sum_{j\neq i}\Big[\theta_R\big(|x_i^R(t)-x_j^R(t)|^{-1}\big)\grad \G\big(x_i^R(t)-x_j^R(t)\big)\nt \\
 &\qquad\qquad\qquad-\theta_R\big(|q_i^R(t)-q_j^R(t)|^{-1}\big)\grad \G\big(q_i^R(t)-q_j^R(t)\big)\Big] \d t \nt \\
&\qquad -\Big(\sum_{\ell=1}^{k_i}\lambda_{i,\ell}^2\Big) \xbar_i^R(t)-\sum_{\ell=1}^{k_i}\lambda_{i,\ell}^2\alpha_{i,\ell}\int_0^t e^{-\alpha_{i,\ell}(t-r)}\xbar_i^R(r)\d r . \label{eqn:m.d.v_i^R+d.xbar_i^R}
\end{align}
By the choice of $\theta_R$ as in \eqref{form:theta_R} and conditions \eqref{cond:U:U''(x)=O(x^lambda-1)} and \eqref{cond:G:grad^2.G(x)<1/|x|^beta}, we invoke the Mean Value Theorem to infer
\begin{align*}
|\theta_R(|x|)\grad U(x)-\theta_R(|y|)\grad U(y)|\le C|x-y|,\quad x,y\in\rbb^d,
\end{align*}
and
\begin{align*}
|\theta_R\big(|x|^{-1}\big)\grad G(x)-\theta_R\big(|y|^{-1}\big)\grad G(y)|\le C|x-y|,\quad x,y\in\rbb^d\setminus\{0\},
\end{align*}
for some positive constant $C=C(R)$. From \eqref{eqn:m.d.v_i^R+d.xbar_i^R}, we arrive at the a.s. bound
\begin{align*}
|\xbar_i^R(t)|^n\le C m^n|v_i^R(t)-v(0)|^n+C\int_0^t \sum_{j=1}^N|\xbar_j^R(r)|^n\d r,
\end{align*}
whence 
\begin{align*}
|\bar{\xb}^R(t)|^n\le C m^n|\vb^R(t) -\vb(0)|^n+C\int_0^t \sum_{j=1}^N|\bar{\xb}^R(r)|^n\d r.
\end{align*}
Gronwall's inequality implies that
\begin{align*}
\E\sup_{t\in[0,T]}|\bar{\xb}^R(t)|^n \le C m^n\E\sup_{t\in[0,T]}|\vb^R(t)-\vb(0)|^n.
\end{align*}
Setting $n=4$, by virtue of Lemma~\ref{lem:small-mass:truncating:m.v_m->0}, this produces the small mass limit \eqref{lim:small-mass:truncating}, as claimed.

\end{proof}

We now turn to the proof of Lemma \ref{lem:small-mass:truncating:m.v_m->0}.

\begin{proof}[Proof of Lemma \ref{lem:small-mass:truncating:m.v_m->0}] From the $v_i-$equation in \eqref{eqn:GLE:N-particle:truncating}, variation constant formula yields
\begin{align}
mv_i^R(t)&=me^{-\frac{\gamma}{m}t}v_i(0)-\int_0^t e^{-\frac{\gamma}{m}(t-r)}\theta_R(|x_i^R(r)|)\grad U(x_i^R(r))\d r \nt\\
&\qquad -\int_0^t e^{-\frac{\gamma}{m}(t-r)}\sum_{j\neq i}\Big[\theta_R\big(|x_i^R(r)-x_j^R(r)|^{-1}\big) \grad G(x_i^R(r)-x_j^R(r))\Big]\d r \nt\\
&\qquad +\sum_{\ell=1}^{k_i}\lambda_{i,\ell}\int_0^t  e^{-\frac{\gamma}{m}(t-r)}z_{i,\ell}^R(r)\d r+\sqrt{2\gamma}\int_0^t e^{-\frac{\gamma}{m}(t-r)}\d W_{i,0}(r). \label{eqn:Duhamel:v_m^R:a}
\end{align}
With regard to the integral involving $z_{i,\ell}$, since $z_{i,\ell}^R$ satisfies the third equation in \eqref{eqn:GLE:N-particle:truncating}, we have
\begin{align*}
z_{i,\ell}^R(t)=e^{-\alpha_{i,\ell} t}z_{i,\ell}(0)-\lambda_{i,\ell}\int_0^t e^{-\alpha_{i,\ell}(t-r)}v_i^R(r)\d r+\sqrt{2\alpha_{i,\ell}}\int_0^t e^{-\alpha_{i,\ell}(t-r)}\d W_{i,\ell}(r).
\end{align*}
Plugging back into \eqref{eqn:Duhamel:v_m^R:a}, we obtain the identity
\begin{align}
mv_i^R(t)&=me^{-\frac{\gamma}{m}t}v_i(0)-\int_0^t e^{-\frac{\gamma}{m}(t-r)}\theta_R(|x_i^R(r)|)\grad U(x_i^R(r))\d r \nt \\
&\qquad -\int_0^t e^{-\frac{\gamma}{m}(t-r)}\sum_{j\neq i}\Big[\theta_R\big(|x_i^R(r)-x_j^R(r)|^{-1}\big) \grad G(x_i^R(r)-x_j^R(r))\Big]\d r \nt\\
&\qquad +\sum_{\ell=1}^{k_i}\lambda_{i,\ell}\int_0^t  e^{-\frac{\gamma}{m}(t-r)}e^{-\alpha_{i,\ell} r}z_{i,\ell}(0)\d r-\sum_{\ell=1}^{k_i}\lambda_{i,\ell}^2\int_0^t  e^{-\frac{\gamma}{m}(t-r)}\int_0^r e^{-\alpha_{i,\ell}(r-s)}v_i^R(s)\d s \d r \nt \\
&\qquad+\sum_{\ell=1}^{k_i}\lambda_{i,\ell}\sqrt{2\alpha_{i,\ell}}\int_0^t  e^{-\frac{\gamma}{m}(t-r)}\int_0^r e^{-\alpha_{i,\ell}(r-s)}\d W_{i,\ell}(s)\d r+ \sqrt{2\gamma}\int_0^t e^{-\frac{\gamma}{m}(t-r)}\d W_{i,0}(r)\nt\\
&=  me^{-\frac{\gamma}{m}t}v_i(0)-I_1-I_2+I_3-I_4+I_5+I_6.\label{eqn:Duhamel:v_m^R:b}
\end{align}
Concerning $I_1$, we invoke condition \eqref{cond:U:U'(x)=O(x^lambda)} while making use of the choice of $\theta_R$ as in \eqref{form:theta_R} to obtain
\begin{align*}
|I_1|\le C\int_0^t e^{-\frac{\gamma}{m}(t-r)}\d r= m C,
\end{align*}
for some positive constant $C=C(R)$ independent of $m$. Likewise, we employ condition \eqref{cond:G:grad.G(x)<1/|x|^beta} to infer
\begin{align*}
|I_2|\le C\int_0^t e^{-\frac{\gamma}{m}(t-r)}\d r= m C.
\end{align*}
Similarly, it is clear that
\begin{align*}
|I_3| \le m\sum_{\ell=1}^{k_i}\lambda_{i,\ell}|z_{i,\ell}(0)|.
\end{align*}
Concerning $I_4$, observe that
\begin{align*}
|I_4|\le \sum_{\ell=1}^{k_i}\lambda_{i,\ell}^2\int_0^t  e^{-\frac{\gamma}{m}(t-r)}\d r\int_0^t \sup_{s\in[0,r]}|v_i^R(s)| \d r\le  m\sum_{\ell=1}^{k_i}\lambda_{i,\ell}^2\int_0^t \sup_{s\in[0,r]}|v_i^R(s)| \d r.
\end{align*}
With regard to the noise term $I_5$, note that 
\begin{align*}
&\Big|\int_0^t  e^{-\frac{\gamma}{m}(t-r)}\int_0^r e^{-\alpha_{i,\ell}(r-s)}\d W_{i,\ell}(s)\d r\Big|^n \\
&\le  \Big|\int_0^t  e^{-\frac{\gamma}{m}(t-r)}\d r\Big|^n  \sup_{r\in[0,T]}\Big|\int_0^r e^{-\alpha_{i,\ell}(r-s)}\d W_{i,\ell}(s)\Big|^n\\
&\le \frac{m^n}{\gamma^n}  \sup_{r\in[0,T]}\Big|\int_0^r e^{-\alpha_{i,\ell}(r-s)}\d W_{i,\ell}(s)\Big|^n.
\end{align*}
We employ Burkholder's inequality to infer
\begin{align*}
&\E\sup_{t\in[0,T]}\Big|\int_0^t  e^{-\frac{\gamma}{m}(t-r)}\int_0^r e^{-\alpha_{i,\ell}(r-s)}\d W_{i,\ell}(s)\d r\Big|^n\\
&\le \frac{m^n}{\gamma^n}\E\sup_{t\in[0,T]}\Big|\int_0^t e^{\alpha_{i,\ell}s}\d W_{i,\ell}(s)\Big|^n \le m^n  C(T,n).
\end{align*}
Turning to $I_6$,
we set
\begin{align*}
Y_{i,0}(t):=\int_0^t e^{-\beta(t-r)}\d W_{i,0}(r),\quad \beta:=\frac{\gamma}{m},
\end{align*}
and let $Z_{i,0}\sim N(0,1)$ be a random variable independent of $W_{i,0}(t)$. Then, the process
\begin{align*}
X_{i,0}(t):=Z_{i,0}e^{-\beta t}+\sqrt{2\beta}Y_{i,0}(t)
\end{align*}
is a stationary solution to 
\begin{align*}
\d X_{i,0}(t)=-\beta X_{i,0}(t)\d t+\sqrt{2\beta}\d W_{i,0}(t),\quad X_{i,0}(0)=Z.
\end{align*}
For $n\ge 1$, it holds by the definition of $X_{i,0}(t)$ that
\begin{align*}
\E\sup_{t\in[0,T]}|Y_{i,0}(t)|^n \le C\beta^{-\frac{n}{2}}\bigg( 1+\E\sup_{t\in[0,T]}|X_{i,0}(t)|^n \bigg)
\end{align*}
By \cite[Lemma B.1]{pavliotis2022derivative}, it holds for all $n>1$ that
\begin{align*}
 \E \sup_{t\in[0,T]}|X_{i,0}(t)|^n\le C\big( 1+\log(1+\beta T) \big)^{n/2}
\end{align*}
As a consequence, for all $\varepsilon>0$
\begin{align*}
\beta^{\frac{n}{2}-\varepsilon}\E\sup_{t\in[0,T]}|Y_{i,0}(t)|^n <C(T,n).
\end{align*}
In other words,
\begin{align*}
\E\sup_{t\in[0,T]}|Y_{i,0}(t)|^n < m^{\frac{n}{2}-\varepsilon} C(T,n).
\end{align*}

Now we collect the above estimates together with expression \eqref{eqn:Duhamel:v_m^R:b} to deduce that for all $n>1$ and $\varepsilon>0$
\begin{align*}
&m^n\E\sup_{t\in[0,T]}|\vb^R(t)|^n\\
& \le m^n C\Big( |\vb(0)|^n+\sum_{i=1}^N|\zb_i^R(0)|^n+ 1  \Big)+m^{\frac{n}{2}-\varepsilon}C+C m^n \int_0^T \E\sup_{r\in[0,t]}|\vb^R(r)|^n\d t,
\end{align*}
holds for some positive constant $C=C(T,n,R,\varepsilon)$ independent of $m$. In view of Gronwall's inequality, for all $n>1$, we arrive at \eqref{lim:small-mass:truncating:m.v_m->0}, as claimed. 
\end{proof}

\subsection{Estimates on \eqref{eqn:GLE:N-particle:m=0}} \label{sec:small-mass:estimate-m=0}

In this subsection, we provide several energy estimates on the limiting system \eqref{eqn:GLE:N-particle:m=0} on any finite time window. More precisely, we have the following result.

\begin{lemma} \label{lem:GLE:N-particle:m=0:supnorm}
Under Assumption \ref{cond:U}, Assumption \ref{cond:G} and Assumption \ref{cond:G:d=1}, for all $(\xb(0),\zb_{1}(0)$,..., $\zb_N(0))\in \D\times(\rbb^d)^{N}$, let $Q(t)=\big(\qb(t)$, $\fb_{1}(t)$,$\dots$, $\fb_{N}(t)\big)$ be the solution of~\eqref{eqn:GLE:N-particle:m=0} and $\beta_1$ be the constant as in Assumption \ref{cond:G} and Assumption \ref{cond:G:d=1}. For all $\varepsilon,\kappa>0$ sufficiently small and $T>0$, the followings hold:

\textup{(a)} If $\beta_1>1$,
\begin{align} \label{ineq:GLE:N-particle:m=0:supnorm:beta_1>1}
\E\Big[\exp\Big\{\sup_{t\in[0,T]}\Big(\kappa|\qb(t)|^2+\kappa\varepsilon\close\sum_{1\le i<j\le N} \frac{1}{|q_i(t)-q_j(t)|^{\beta_1-1}}\Big)\Big\}\Big]\le C,
\end{align}
for some positive constant $C=C(\kappa,\varepsilon, T,\xb(0),\zb_{1}(0),\dots,\zb_N(0))$.

\textup{(b)} Otherwise, if $\beta_1=1$,
\begin{align} \label{ineq:GLE:N-particle:m=0:supnorm:beta_1=1}
\E\Big[\exp\Big\{\sup_{t\in[0,T]}\Big(\kappa|\qb(t)|^2-\kappa\varepsilon\close\sum_{1\le i<j\le N}\close \log|q_i(t)-q_j(t)|\Big)\Big\}\Big]\le C.
\end{align}

\end{lemma}

The proof of Lemma \ref{lem:GLE:N-particle:m=0:supnorm} relies on two ingredients: the choice of Lyapunov functions specifically designed for \eqref{eqn:GLE:N-particle:m=0} and the exponential Martingale inequality. Later in Section \ref{sec:small-mass:proof-of-theorem}, we will particularly exploit Lemma \ref{lem:GLE:N-particle:m=0:supnorm} to remove the Lipschitz constraint in Proposition \ref{prop:small-mass:truncating} so as to conclude the main Theorem \ref{thm:small-mass:N-particle}.

\begin{proof}[Proof of Lemma \ref{lem:GLE:N-particle:m=0:supnorm}] (a) Suppose that $\beta_1>1$. For $\varepsilon>0$, we consider the functions $\Gamma_1$ given by
\begin{align} \label{form:Gamma_1}
\Gamma_1(\qb,\fb_1,\dots,\fb_N) &= \frac{1}{2}\gamma|\qb|^2 +\frac{1}{2}\sum_{i=1}^N \sum_{\ell=1}^{k_i}\frac{1}{\alpha_{i,\ell}}|f_{i,\ell}|^2
+\varepsilon\,\gamma\close\sum_{1\le i<j\le N}\frac{1}{ |q_i-q_j|^{\beta_1-1}}.
\end{align} 
We aim to show that $\Gamma_1(t)$ satisfies suitable energy estimate, allowing for establishing the moment bound in sup norm \eqref{ineq:GLE:N-particle:m=0:supnorm:beta_1>1}.

From \eqref{eqn:GLE:N-particle:m=0}, we employ It\^o's formula to obtain
\begin{align*}
&\d \Big(\frac{1}{2}\gamma|\qb(t)|^2 +\frac{1}{2}\sum_{i=1}^N \sum_{\ell=1}^{k_i}\frac{1}{\alpha_{i,\ell}}|f_{i,\ell}(t)|^2\Big) \\
 &= - \sum_{i=1}^N \Big(\sum_{\ell=1}^{k_i}\lambda_{i,\ell}^2\Big)|q_i(t)|^2\d t-\sum_{i=1}^N \la \grad U(q_i(t)),q_i(t)\ra\d t +N\gamma^2d\, \d t\\
&\qquad-\close\sum_{1\le i<j\le N}\close \la \grad G(q_i(t)-q_j(t)),q_i(t)-q_j(t)\ra\d t +\sum_{i=1}^N\sqrt{2\gamma}\la q_i(t),\d W_{i,0}(t)\ra\\
&\qquad -\sum_{i=1}^N \sum_{\ell=1}^{k_i}|f_{i,\ell}|^2\d t+ \alpha_{i,\ell}^2d\,\d t+ \sqrt{2\alpha_{i,\ell}}\la f_{i,\ell}(t),\d W_{i,\ell}(t)\ra .
\end{align*}
Recalling condition \eqref{cond:U:x.U'(x)>-x^(lambda+1)}, we readily have
\begin{align*}
-\sum_{i=1}^N \la \grad U(q_i(t)),q_i(t)\ra \le -Na_2\sum_{i=1}^N|q_i(t)|^{\lambda+1}+Na_3 .
\end{align*}
To bound the cross terms involving $G$, we invoke~\eqref{cond:G:grad.G(x)<1/|x|^beta} and obtain
\begin{align*}
-\close\sum_{1\le i<j\le N}\close \la \grad G(q_i(t)-q_j(t)),q_i(t)-q_j(t)\ra\le a_1\close\sum_{1\le i<j\le N}\close\Big(|q_i(t)-q_j(t)|+\frac{1}{|q_i(t)-q_j(t)|^{\beta_1-1}}\Big).
\end{align*}
Since $\sum_{1\le i<j\le N}|q_i-q_j|$ can be subsumed into $-|\qb|^2$, we infer from the above estimates that
\begin{align}
&\d \Big(\frac{1}{2}\gamma|\qb(t)|^2 +\frac{1}{2}\sum_{i=1}^N \sum_{\ell=1}^{k_i}\frac{1}{\alpha_{i,\ell}}|f_{i,\ell}(t)|^2\Big) \nt  \\
 &\le  - c|\qb(t)|^2\d t-c|\qb(t)|^{\lambda+1}\d t-\sum_{i=1}^N |\fb_i|^2 +C \d t+ a_1\close\sum_{1\le i<j\le N}\frac{1}{|q_i-q_j|^{\beta_1-1}}\d t \nt \\
&\qquad +\sum_{i=1}^N\sqrt{2\gamma}\la q_i(t),\d W_{i,0}(t)\ra+\sum_{i=1}^N \sum_{\ell=1}^{k_i} \sqrt{2\alpha_{i,\ell}}\la f_{i,\ell}(t),\d W_{i,\ell}(t)\ra, \label{ineq:d.(|q|^2+|f|^2)}
\end{align}
for some positive constant $c,\,C$ independent of $t$.

Turning to the last term on the right-hand side of \eqref{form:Gamma_1}, a routine computation gives
\begin{align}
&\d \Big( \varepsilon\,\gamma\close\sum_{1\le i<j\le N}\frac{1}{|q_i(t)-q_j(t)|^{\beta_1-1}}\Big) \nt \\
&= -\varepsilon(\beta_1-1)\sum_{i=1}^N \Big\la \sum_{j\neq i}\frac{q_i(t)-q_j(t)}{|q_i(t)-q_j(t)|^{\beta_1+1}} , -\grad U(q_i(t))\d t- \sum_{\ell\neq i}\grad \G\big(q_i(t)-q_\ell(t)\big) \d t \nt \\
&\qquad\qquad -\sum_{i=1}^{k_i} \lambda_{i,\ell}^2 q_i(t)\d t+\sum_{\ell=1}^{k_i}\lambda_{i,\ell} f_{i,\ell}(t)\d t+\sqrt{2\gamma}\d W_{i,0}(t) \Big\ra \nt \\
&\qquad -\varepsilon(\beta_1-1)(\beta_1+1-d)\close\sum_{1\le i<j\le N} \frac{1}{|q_i(t)-q_j(t)|^{\beta_1+1}}\d t. \label{eqn:d.1/|q_i-q_j|^(beta_1-1)}
\end{align}
Using Cauchy-Schwarz inequality, it is clear that
\begin{align*}
&-\varepsilon(\beta_1-1)\sum_{i=1}^N \Big\la \sum_{j\neq i}\frac{q_i(t)-q_j(t)}{|q_i(t)-q_j(t)|^{\beta_1+1}} , -\sum_{i=1}^{k_i} \lambda_{i,\ell}^2 q_i(t)+\sum_{\ell=1}^{k_i}\lambda_{i,\ell} f_{i,\ell}(t)\Big\ra \\
&\le C\varepsilon^{1/2}|\qb(t)|^2+C\varepsilon^{1/2}\sum_{i=1}^N |\fb_i(t)|^2+C\varepsilon^{3/2}\close\sum_{1\le i<j\le N}\frac{1}{|q_i(t)-q_j(t)|^{2\beta_1}}.
\end{align*}
Concerning the cross terms involving $U$ on the right-hand side of \eqref{eqn:d.1/|q_i-q_j|^(beta_1-1)}, we recall that $\grad U$ satisfies \eqref{cond:U:U''(x)=O(x^lambda-1)}. In light of the Mean Value Theorem, we infer
\begin{align*}
& -\varepsilon(\beta_1-1)\sum_{i=1}^N \Big\la \sum_{j\neq i}\frac{q_i(t)-q_j(t)}{|q_i(t)-q_j(t)|^{\beta_1+1}} , -\grad U(q_i(t))\Big\ra\\
&= \varepsilon(\beta_1-1)\close\sum_{1\le i<j\le N}\frac{\la q_i(t)-q_j(t),\grad U(q_i(t))-\grad U(q_j(t))\ra}{|q_i(t)-q_j(t)|^2}\\
&\le C\varepsilon\close\sum_{1\le i<j\le N}\big(|q_i(t)|^{\lambda-1}+|q_j(t)|^{\lambda-1}\big)\le C\varepsilon |\qb(t)|^{\lambda-1}.
\end{align*}
Turning to the cross terms involving $G$ on the right-hand side of \eqref{eqn:d.log|q_i-q_j|}, we recast them as follows:
\begin{align*}
&-\varepsilon(\beta_1-1)\sum_{i=1}^N \Big\la \sum_{j\neq i}\frac{q_i(t)-q_j(t)}{|q_i(t)-q_j(t)|^{\beta_1+1}} , - \sum_{\ell\neq i}\grad \G\big(q_i(t)-q_\ell(t)\big)\Big\ra\\
&=  -\varepsilon(\beta_1-1)\sum_{i=1}^N \Big\la \sum_{j\neq i}\frac{q_i(t)-q_j(t)}{|q_i(t)-q_j(t)|^{\beta_1+1}} , a_4  \sum_{\ell\neq i}\frac{q_i(t)-q_\ell(t)}{|q_i(t)-q_\ell(t)|^{\beta_1+1}}\Big\ra \\
&\qquad +\varepsilon(\beta_1-1)\sum_{i=1}^N \Big\la \sum_{j\neq i}\frac{q_i(t)-q_j(t)}{|q_i(t)-q_j(t)|^{\beta_1+1}} ,  \sum_{\ell\neq i}\grad \G\big(q_i(t)-q_\ell(t)\big)+a_4\frac{q_i(t)-q_\ell(t)}{|q_i(t)-q_\ell(t)|^{\beta_1+1}}\Big\ra.
\end{align*}
In view of Lemma~\ref{lem:<|x_i-x_j|^s,|x_i-x_ell|^s>}, cf. \eqref{ineq:|x_i-x_i|:<|x_i-x_j|^s,|x_i-x_ell|^s>},   we find
\begin{align*}
&-\varepsilon(\beta_1-1)\sum_{i=1}^N \Big\la \sum_{j\neq i}\frac{q_i(t)-q_j(t)}{|q_i(t)-q_j(t)|^{\beta_1+1}} , a_4  \sum_{\ell\neq i}\frac{q_i(t)-q_\ell(t)}{|q_i(t)-q_\ell(t)|^{\beta_1+1}}\Big\ra\\
&\le -\varepsilon(\beta_1-1)a_4 \cdot \frac{4}{N(N-1)^2}\sum_{1\le i <j\le N}\frac{1}{|q_i(t)-q_j(t)|^{2\beta_1}}.
\end{align*}
On the other hand, the condition \eqref{cond:G:|grad.G(x)+q/|x|^beta_1|<1/|x|^beta_2} implies
\begin{align*}
&\varepsilon(\beta_1-1)\sum_{i=1}^N \Big\la \sum_{j\neq i}\frac{q_i(t)-q_j(t)}{|q_i(t)-q_j(t)|^{\beta_1+1}} ,  \sum_{\ell\neq i}\grad \G\big(q_i(t)-q_\ell(t)\big)+a_4\frac{q_i(t)-q_\ell(t)}{|q_i(t)-q_\ell(t)|^{\beta_1+1}}\Big\ra \\
&\le \varepsilon(\beta_1-1)\sum_{i=1}^N \Big\la \sum_{j\neq i}\frac{1}{|q_i(t)-q_j(t)|^{\beta_1}} ,  \sum_{\ell\neq i}\Big(a_5\frac{1}{|q_i(t)-q_\ell(t)|^{\beta_2}}+a_6\Big)\Big\ra\\
&\le \varepsilon^{3/2} C \sum_{1\le i <j\le N}\frac{1}{|q_i(t)-q_j(t)|^{2\beta_1}}+ \varepsilon^{1/2} C \sum_{1\le i <j\le N}\frac{1}{|q_i(t)-q_j(t)|^{2\beta_2}}+C,
\end{align*}
for some positive constant $C$ independent of $t$ and $\varepsilon$. Since $\beta_2<\beta_1$, cf. \eqref{cond:G:|grad.G(x)+q/|x|^beta_1|<1/|x|^beta_2},  taking $\varepsilon$ small enough produces the bound
\begin{align}
&-\varepsilon(\beta_1-1)\sum_{i=1}^N \Big\la \sum_{j\neq i}\frac{q_i(t)-q_j(t)}{|q_i(t)-q_j(t)|^{\beta_1+1}} , - \sum_{\ell\neq i}\grad \G\big(q_i(t)-q_\ell(t)\big)\Big\ra \nt \\
&\le  -\varepsilon\,  c\close \sum_{1\le i <j\le N}\frac{1}{|q_i(t)-q_j(t)|^{2\beta_1}}. \label{ineq:<1/|q_i-q_j|^(beta_1),grad.G>}
\end{align}
Similarly, regarding the last term on the right-hand side of \eqref{eqn:d.1/|q_i-q_j|^(beta_1-1)}, since in Case 1, $\beta_1>1$, it is clear that $\varepsilon |q_i-q_j|^{-\beta_1-1} $ can also be subsumed into $-\varepsilon|q_i-q_j|^{-2\beta_1}$ as in \eqref{ineq:<1/|q_i-q_j|^(beta_1),grad.G>}. Altogether with the expression \eqref{eqn:d.1/|q_i-q_j|^(beta_1-1)}, we arrive at the estimate
\begin{align}
&\d \Big( \varepsilon\,\gamma\close\sum_{1\le i<j\le N}\frac{1}{|q_i(t)-q_j(t)|^{\beta_1-1}}\Big) \nt\\
&\le \varepsilon^{1/2} C |\qb(t)|^2\d t+ \varepsilon^{1/2} C|\qb(t)|^{\lambda-1}\d t+\varepsilon^{1/2} C\sum_{i=1}^N |\fb_i(t)|^2\d t+C\d t \nt \\
&\qquad-\varepsilon \,c \close\sum_{1\le i<j \le N}\frac{1}{|q_i(t)-q_j(t)|^{2\beta_1}} \d t - \varepsilon\sum_{i=1}^N \Big\la \sum_{j\neq i}\frac{q_i(t)-q_j(t)}{|q_i(t)-q_j(t)|^{\beta_1+1}} , \sqrt{2\gamma}\d W_{i,0}(t) \Big\ra.\label{ineq:d.1/|q_i-q_j|^{beta_1-1}}
\end{align}
Next, from the expression \eqref{form:Gamma_1} of $\Gamma_1$ and the estimates \eqref{ineq:d.(|q|^2+|f|^2)} and \eqref{ineq:d.1/|q_i-q_j|^{beta_1-1}}, we obtain (by taking $\varepsilon$ small enough)
\begin{align} \label{ineq:d.Gamma_1}
\d \Gamma_1(t) &\le - c|\qb(t)|^2\d t-c|\qb(t)|^{\lambda+1}\d t-c\sum_{i=1}^N |\fb_i|^2 -\varepsilon\, c \close\sum_{1\le i<j \le N}\frac{1}{|q_i(t)-q_j(t)|^{2\beta_1}} \d t+C \d t \nt \\
&\qquad+ \sum_{i=1}^N\sqrt{2\gamma}\la q_i(t),\d W_{i,0}(t)\ra+\sum_{i=1}^N \sum_{\ell=1}^{k_i} \sqrt{2\alpha_{i,\ell}}\la f_{i,\ell}(t),\d W_{i,\ell}(t)\ra \nt\\
& \qquad - \varepsilon\sum_{i=1}^N \Big\la \sum_{j\neq i}\frac{q_i(t)-q_j(t)}{|q_i(t)-q_j(t)|^{\beta_1+1}} , \sqrt{2\gamma}\d W_{i,0}(t) \Big\ra.
\end{align}
We emphasize that in \eqref{ineq:d.Gamma_1}, $c>0$ is independent of $\varepsilon$ whereas $C>0$ may still depend on $\varepsilon$.

Now, to establish~\eqref{ineq:GLE:N-particle:m=0:supnorm:beta_1>1}, we aim to employ the well-known exponential Martingale inequality applied to~\eqref{ineq:d.Gamma_1}. The argument below is similarly to that found in \cite[Lemma 5.1]{hairer2008spectral}. See also \cite{glatt2021long,glatt2022short}.

For $\kappa\in(0,1)$ to be chosen later, from \eqref{ineq:d.Gamma_1}, we observe that
 \begin{align} \label{ineq:d.kappa.Gamma_1}
\kappa\d \Gamma_1(t) &\le - c\kappa|\qb(t)|^2\d t-c\kappa |\qb(t)|^{\lambda+1}\d t-c\kappa\sum_{i=1}^N |\fb_i|^2 +C\kappa \d t \nt \\
&\qquad- \kappa\varepsilon\, c  \close\sum_{1\le i<j \le N}\frac{1}{|q_i(t)-q_j(t)|^{2\beta_1}} \d t+\d M_1(t).
\end{align}
In the above, the semi-Martingale term $M_1(t)$ is defined as
\begin{align*}
\d M_1(t) &= \kappa\sum_{i=1}^N\sqrt{2\gamma}\la q_i(t),\d W_{i,0}(t)\ra+\kappa\sum_{i=1}^N \sum_{\ell=1}^{k_i} \sqrt{2\alpha_{i,\ell}}\la f_{i,\ell}(t),\d W_{i,\ell}(t)\ra \nt\\
&  -\kappa \varepsilon\sum_{i=1}^N \Big\la \sum_{j\neq i}\frac{q_i(t)-q_j(t)}{|q_i(t)-q_j(t)|^{\beta_1+1}} , \sqrt{2\gamma}\d W_{i,0}(t) \Big\ra,
\end{align*}
whose variation process $\la M_1\ra(t)$ is given by
\begin{align*}
\d \la M_1 \ra (t) &=2\gamma \kappa^2 \Big|\sum_{i=1}^N\Big( q_i(t)- \varepsilon \sum_{j\neq i}\frac{q_i(t)-q_j(t)}{|q_i(t)-q_j(t)|^{\beta_1+1}}\Big)\Big|^2\d t+\kappa^2\sum_{i=1}^N \sum_{\ell=1}^{k_i} 2\alpha_{i,\ell}| f_{i,\ell}(t)|^2\d t. \nt
\end{align*}
Using Cauchy-Schwarz inequality, it is clear that
\begin{align*}
\d \la M_1 \ra (t) &\le  - \ctilde \kappa^2|\qb(t)|^2\d t-\ctilde\kappa^2\sum_{i=1}^N |\fb_i|^2 \d t- \kappa^2\varepsilon\ctilde  \close\sum_{1\le i<j \le N}\frac{1}{|q_i(t)-q_j(t)|^{2\beta_1}} \d t,
\end{align*}
for some positive constant $\ctilde$ independent of both $\kappa$ and $\varepsilon$. It follows from \eqref{ineq:d.kappa.Gamma_1} that
\begin{align*}
\kappa\d \Gamma_1(t) &\le C\kappa \d t-\frac{c}{\kappa}\la M_1\ra(t)\d t + \d M_1(t).
\end{align*}
Recalling the exponential Martingale inequality applying to $M_1(t)$,
\begin{align}\label{ineq:exponential-Martingale}
\P\Big(\sup_{t\ge 0}\Big[M_1(t)-\frac{c}{\kappa}\la M_1\ra(t)\Big] >r\Big)\le e^{-\frac{2c}{\kappa}r}, \quad r\ge 0,
\end{align}
we deduce that
\begin{align*}
\P\Big(\sup_{t\in[0,T]}\Big[\kappa\Gamma_1(t)-\kappa\Gamma_1(0)-\kappa Ct\Big] >r\Big)\le e^{-\frac{2c}{\kappa}r}.
\end{align*}
In particular, by choosing $\kappa$ sufficiently small, the above inequality implies
\begin{align} \label{ineq:E.sup.kappa.Gamma_1(t)}
\E\exp\Big\{\sup_{t\in[0,T]}\kappa\Gamma_1(t)\Big\}\le C,
\end{align}
for some positive constant $C=C(T,\kappa,\varepsilon,\xb(0),\zb_1(0),\dots,\zb_N(0) )$. Recalling $\Gamma_1$ as in \eqref{form:Gamma_1}, the estimate \eqref{ineq:E.sup.kappa.Gamma_1(t)} produces \eqref{ineq:GLE:N-particle:m=0:supnorm:beta_1>1}. Hence, part (a)  is established for $\beta_1>1$.

(b) Considering $\beta_1=1$, in this case, we introduce the function $\Gamma_2$ defined as
\begin{align} \label{form:Gamma_2}
\Gamma_2(\qb,\fb_1,\dots,\fb_N) &= \frac{1}{2}\gamma|\qb|^2 +\frac{1}{2}\sum_{i=1}^N \sum_{\ell=1}^{k_i}\frac{1}{\alpha_{i,\ell}}|f_{i,\ell}|^2
-\varepsilon\,\gamma\close\sum_{1\le i<j\le N}\close\log |q_i-q_j|.
\end{align}  
With regard to the log term on the above right-hand side, the following identity holds
\begin{align}
&\d \Big(-\varepsilon\gamma\close\sum_{1\le i<j\le N}\close\log |q_i(t)-q_j(t)|\Big)  \nt \\
&= -\varepsilon\sum_{i=1}^N \Big\la \sum_{j\neq i}\frac{q_i(t)-q_j(t)}{|q_i(t)-q_j(t)|^2} , -\grad U(q_i(t))\d t- \sum_{\ell\neq i}\grad \G\big(q_i(t)-q_\ell(t)\big) \d t \nt \\
&\qquad\qquad -\sum_{i=1}^{k_i} \lambda_{i,\ell}^2 q_i(t)\d t+\sum_{\ell=1}^{k_i}\lambda_{i,\ell} f_{i,\ell}(t)\d t+\sqrt{2\gamma}\d W_{i,0}(t) \Big\ra \nt \\
&\qquad -\varepsilon(d-2)\close\sum_{1\le i<j\le N} \frac{1}{|q_i(t)-q_j(t)|^2}\d t. \label{eqn:d.log|q_i-q_j|}
\end{align}
Similarly to the estimates in Case 1, we readily have
\begin{align}
&-\varepsilon\sum_{i=1}^N \Big\la \sum_{j\neq i}\frac{q_i(t)-q_j(t)}{|q_i(t)-q_j(t)|^{2}} , -\grad U(q_i(t))-\sum_{i=1}^{k_i} \lambda_{i,\ell}^2 q_i(t)+\sum_{\ell=1}^{k_i}\lambda_{i,\ell} f_{i,\ell}(t)\Big\ra \nt  \\
&\le C\varepsilon^{1/2}|\qb(t)|^2+C\varepsilon |\qb(t)|^{\lambda-1}+C\varepsilon^{1/2}\sum_{i=1}^N |\fb_i(t)|^2+C\varepsilon^{3/2}\close\sum_{1\le i<j\le N}\frac{1}{|q_i(t)-q_j(t)|^{2}}.\label{ineq:d.log|q_i-q_j|:U}
\end{align}
Concerning the cross terms involving $G$ on the right-hand side of \eqref{eqn:d.log|q_i-q_j|}. We employ an argument from Case 1 while making use of condition \eqref{cond:G:|grad.G(x)+q/|x|^beta_1|<1/|x|^beta_2} and the estimate \eqref{ineq:|x_i-x_i|:<|x_i-x_j|^s,|x_i-x_ell|^s>:s<1} to arrive at
\begin{align}
&-\varepsilon\sum_{i=1}^N \Big\la \sum_{j\neq i}\frac{q_i(t)-q_j(t)}{|q_i(t)-q_j(t)|^2} , - \sum_{\ell\neq i}\grad \G\big(q_i(t)-q_\ell(t)\big)\Big\ra \nt \\
&\le -2a_4 \varepsilon\close\sum_{1\le i<j \le N}\frac{1}{|q_i(t)-q_j(t)|^{2}}+C\varepsilon^{3/2}\close\sum_{1\le i<j \le N}\frac{1}{|q_i(t)-q_j(t)|^{2}}+\Ctilde.\label{ineq:<q_i-q_j,grad.G(q_i-q_j)>}
\end{align}
In the above, we emphasize that $C$ is independent of $\varepsilon$ even though $\Ctilde$ may still depend on $\varepsilon$. Turning to the last term on the right-hand side of \eqref{eqn:d.log|q_i-q_j|}, i.e.,
\begin{align*}
-\varepsilon(d-2)\close\sum_{1\le i<j\le N} \frac{1}{|q_i(t)-q_j(t)|^2}\d t,
\end{align*}
there are two cases to be considered, depending on the dimension $d$. In dimension $d\ge 2$, it is clear that the above expression is negative and thus is negligible. On the other hand, in dimension $d=1$, it is reduced to
\begin{align*}
\varepsilon\close\sum_{1\le i<j\le N} \frac{1}{|q_i(t)-q_j(t)|^2}\d t.
\end{align*} 
In view of Assumption \ref{cond:G:d=1}, we combine with \eqref{ineq:<q_i-q_j,grad.G(q_i-q_j)>} to obtain
\begin{align*}
&-\varepsilon\sum_{i=1}^N \Big\la \sum_{j\neq i}\frac{q_i(t)-q_j(t)}{|q_i(t)-q_j(t)|^2} , - \sum_{\ell\neq i}\grad \G\big(q_i(t)-q_\ell(t)\big)\Big\ra +\varepsilon\close\sum_{1\le i<j\le N} \frac{1}{|q_i(t)-q_j(t)|^2}\nt\\
&\le -(2a_4-1)\varepsilon\close\sum_{1\le i<j\le N} \frac{1}{|q_i(t)-q_j(t)|^2}+C\varepsilon^{3/2}\close\sum_{1\le i<j \le N}\frac{1}{|q_i(t)-q_j(t)|^{2}}+\Ctilde,
\end{align*}
whence (by taking $\varepsilon$ sufficiently small)
\begin{align*}
&-\varepsilon\sum_{i=1}^N \Big\la \sum_{j\neq i}\frac{q_i(t)-q_j(t)}{|q_i(t)-q_j(t)|^2} , - \sum_{\ell\neq i}\grad \G\big(q_i(t)-q_\ell(t)\big)\Big\ra +\varepsilon\close\sum_{1\le i<j\le N} \frac{1}{|q_i(t)-q_j(t)|^2}\nt\\
&\le -c\varepsilon\close\sum_{1\le i<j\le N} \frac{1}{|q_i(t)-q_j(t)|^2}+\Ctilde,
\end{align*}
From \eqref{eqn:d.log|q_i-q_j|} and \eqref{ineq:d.log|q_i-q_j|:U}, we find
\begin{align}
&\d \Big(-\close\sum_{1\le i<j\le N}\close\log |q_i(t)-q_j(t)|\Big)\nt \\
&\le -c\varepsilon \close\sum_{1\le i<j \le N}\frac{1}{|q_i(t)-q_j(t)|^2} \d t + C\varepsilon^{1/2}|\qb(t)|^2\d t+C\varepsilon |\qb(t)|^{\lambda-1}\d t\nt\\
&\qquad+C\varepsilon^{1/2}\sum_{i=1}^N |\fb_i(t)|^2\d t-\varepsilon\sum_{i=1}^N \Big\la \sum_{j\neq i}\frac{q_i(t)-q_j(t)}{|q_i(t)-q_j(t)|^2}, \sqrt{2\gamma}\d W_{i,0}(t) \Big\ra+\Ctilde\d t.\label{ineq:d.log|q_i-q_j|}
\end{align}
Next, we combine estimates \eqref{ineq:d.(|q|^2+|f|^2)}, \eqref{ineq:d.log|q_i-q_j|} with expression \eqref{form:Gamma_2} of $\Gamma_2$ to infer
\begin{align} \label{ineq:d.Gamma_2}
\d \Gamma_2(t) &\le - c|\qb(t)|^2\d t-c|\qb(t)|^{\lambda+1}\d t-c\sum_{i=1}^N |\fb_i|^2 -\varepsilon\, c \close\sum_{1\le i<j \le N}\frac{1}{|q_i(t)-q_j(t)|^{2}} \d t+C \d t \nt \\
&\qquad+ \sum_{i=1}^N\sqrt{2\gamma}\la q_i(t),\d W_{i,0}(t)\ra+\sum_{i=1}^N \sum_{\ell=1}^{k_i} \sqrt{2\alpha_{i,\ell}}\la f_{i,\ell}(t),\d W_{i,\ell}(t)\ra \nt\\
& \qquad - \varepsilon\sum_{i=1}^N \Big\la \sum_{j\neq i}\frac{q_i(t)-q_j(t)}{|q_i(t)-q_j(t)|^{2}} , \sqrt{2\gamma}\d W_{i,0}(t) \Big\ra.
\end{align}
We may now employ an argument similarly to the exponential Martingale approach as in Case 1 to deduce the bound for all $\kappa$ sufficiently small
\begin{align*} 
\E\exp\Big\{\sup_{t\in[0,T]}\kappa\Gamma_2(t)\Big\}\le C,
\end{align*}
for some positive constant $C=C(T,\kappa,\varepsilon,\xb(0),\zb_1(0),\dots,\zb_N(0) )$. Recalling $\Gamma_2$ defined in \eqref{form:Gamma_2}, this produces the estimate \eqref{ineq:GLE:N-particle:m=0:supnorm:beta_1=1}. The proof is thus finished.

\end{proof}

\subsection{Proof of Theorem \ref{thm:small-mass:N-particle}} 

\label{sec:small-mass:proof-of-theorem}

We are now in a position to conclude Theorem \ref{thm:small-mass:N-particle}. The argument follows along the lines of the proof of \cite[Theorem 4]{nguyen2018small} adapted to our setting. See also \cite[Theorem 2.4]{herzog2016small}. The key observation is that instead of controlling the exiting time of the process $\xb_m(t)$ as $m\to 0$, we are able to control $\qb(t)$ since $\qb(t)$ is independent of $m$.

\begin{proof}[Proof of Theorem \ref{thm:small-mass:N-particle}]
Let $\big(\xb_m(t)$, $\vb_m(t)$, $\zb_{1,m}(t)$,$\dots$, $\zb_{N,m}(t)\big)$ and $\big(\qb(t)$, $\fb_{1}(t)$,$\dots$, $\fb_{N}(t)\big)$ respectively solve~\eqref{eqn:GLE:N-particle} and~\eqref{eqn:GLE:N-particle:m=0}. For $R,\,m>0$, define the following stopping times
\begin{align} \label{form:stopping-time:sigma^R}
\sigma^R = \inf_{t\geq 0}\Big\{|\qb(t)|+\close\sum_{1\le i<j\le N}\close|q_i(t)-q_j(t)|^{-1}\geq R\Big\},
\end{align}
and
\begin{align*}
\sigma^R_m = \inf_{t\geq 0}\Big\{|\xb_m(t)|+\close\sum_{1\le i<j\le N}\close |x_{i,m}(t)-x_{j,m}(t)|^{-1}\geq R\Big\}.
\end{align*}
Fixing $T,\,\xi>0$, observe that
\begin{align} \label{ineq:P(|x-q|>xi)}
\P\Big(\sup_{t\in[0,T]}|\xb_m(t)-\qb(t)|>\xi\Big)&\leq \P\Big(\sup_{t\in[0,T]}|\xb_m(t)-\qb(t)|>\xi,\sigma^R\mi\sigma^R_m\geq T\Big) \nt \\
&\qquad+\P\big(\sigma^R\mi\sigma^R_m<T\big).
\end{align}

To control the first term on the above right-hand side, observe that
\begin{align*}
\P\big(0\leq t\leq \sigma^R\mi\sigma^R_m,\, \qb(t)=\qb^R(t),\, \xb_m(t)=\xb_m^R(t)\big)=1,
\end{align*}
where $\qb^R(t)$ and $\xb_m^R(t)$ are the first components of the solutions of \eqref{eqn:GLE:N-particle:m=0:truncating} and \eqref{eqn:GLE:N-particle:truncating}, respectively. As a consequence,
\begin{align} \label{ineq:P(|x-q|>xi,sigma>T)}
&\P\Big(\sup_{0\leq t\leq T}|\xb_m(t)-\qb(t)|>\xi,\sigma^R\mi\sigma^R_m\geq T\Big) \nt \\
&\leq \P\Big(\sup_{0\leq t\leq T}|\xb^R_m(t)-\qb^R(t)|>\xi\Big)\le  \frac{m}{\xi^4}\cdot C(T,R).
\end{align} 
In the last estimate above, we employed Proposition~\ref{prop:small-mass:truncating} while making use of Markov's inequality. 

Turning to $\P(\sigma^R\mi\sigma^R_m<T)$, we note that
\begin{align}
&\P\big(\sigma^R\mi\sigma^R_m<T\big)\nt \\
&\leq \P\Big(\sup_{t\in[0,T]}|\xb_m^R(t)-\qb^R(t)|\leq \frac{\xi}{R},\sigma^R\mi\sigma^R_m < T\Big) +\P\Big(\sup_{t\in[0,T]}|\xb_m^R(t)-\qb^R(t)|> \frac{\xi}{R}\Big) \nt \\
&\leq  \P\Big(\sup_{t\in[0,T]}|\xb_m^R(t)-\qb^R(t)|\leq\frac{\xi}{R},\sigma^R_m < T\leq\sigma^R\Big)+\P(\sigma^R< T) \nt \\
&\qquad\qquad\qquad+\P\Big(\sup_{t\in[0,T]}|\xb_m^R(t)-\qb^R(t)|> \frac{\xi}{R}\Big)  \nt \\
&= I_1+I_2+I_3. \label{ineq:P(sigma<T):I_1+I_2+I_3}
\end{align}
Concerning $I_3$, the same argument as in \eqref{ineq:P(|x-q|>xi,sigma>T)} produces the bound
\begin{align}\label{ineq:P(sigma<T):I_3}
I_3=\P\Big(\sup_{0\leq t\leq T}|\xb^R_m(t)-\qb^R(t)|>\frac{\xi}{R}\Big)\le  \frac{m}{\xi^4}\cdot C(T,R).
\end{align}
Next, considering $I_2$, from \eqref{form:stopping-time:sigma^R}, observe that for all $\varepsilon$ small and $R$ large enough,
\begin{align*}
\big\{  \sigma^R<T\big\}&=\Big\{  \sup_{t\in[0,T]}|\qb(t)|+\close\sum_{1\le i<j\le N}\close|q_i(t)-q_j(t)|^{-1}\geq R\Big\}\\
&\subseteq \Big\{  \sup_{t\in[0,T]}|\qb(t)|\geq \frac{R}{N^2}\Big\}\bigcup_{1\le i<j\le N} \Big\{ -\varepsilon\log|q_i(t)-q_j(t)|\geq \varepsilon\log\Big(\frac{R}{N^2}\Big)\Big\}\\
&\subseteq \Big\{  \sup_{t\in[0,T]}|\qb(t)|^2-\varepsilon\close\sum_{1\le i<j\le N}\close\log|q_i(t)-q_j(t)| \geq \frac{R}{N^2}\Big\}\\
&\qquad\qquad\bigcup_{1\le i<j\le N} \Big\{ \sup_{t\in[0,T]}|\qb(t)|^2-\varepsilon\close\sum_{1\le i<j\le N}\close\log|q_i(t)-q_j(t)| \geq \varepsilon\log\Big(\frac{R}{N^2}\Big)\Big\},
\end{align*}
whence,
\begin{align*}
\big\{  \sigma^R<T\big\}
&\subseteq \Big\{ \sup_{t\in[0,T]}|\qb(t)|^2-\varepsilon\close\sum_{1\le i<j\le N}\close\log|q_i(t)-q_j(t)| \geq \varepsilon\log\Big(\frac{R}{N^2}\Big)\Big\}.
\end{align*}
We note that Proposition \ref{prop:small-mass:truncating} implies the estimate
\begin{align} \label{ineq:GLE:N-particle:m=0:supnorm:beta_1:log}
\E\Big[\sup_{t\in[0,T]}\Big(|\qb(t)|^2-\varepsilon\close\sum_{1\le i<j\le N}\close \log|q_i(t)-q_j(t)|\Big)\Big]\le C,
\end{align}
for some positive constant $C=C(T,\varepsilon)$. Together with Markov's inequality, we infer the bound for $R$ large enough
\begin{align} \label{ineq:P(sigma<T):I_2}
I_2=\P(\sigma^R<T)\le \frac{C(T)}{\varepsilon\log(R/N^2)}\le \frac{C(T)}{\varepsilon \log R}.
\end{align}
Turning to $I_1$ on the right-hand side of \eqref{ineq:P(sigma<T):I_1+I_2+I_3}, for $R$ large enough and $\xi\in(0,1)$, a chain of event implications is derived as follows:
\begin{align*}
&\Big\{\sup_{t\in[0,T]}|\xb_m^R(t)-\qb^R(t)|\leq\frac{\xi}{R},\sigma^R_m < T\leq\sigma^R\Big\}\\
&= \Big\{\sup_{t\in[0,T]}|\xb_m^R(t)-\qb(t)|\leq\frac{\xi}{R},\sup_{t\in[0,T]}\Big(|\xb_m^R(t)|+\close\sum_{1\le i<j\le N}\close|x_{i,m}^{R}(t)-x_{j,m}^R(t)|^{-1}\Big)\ge R,\sigma^R_m < T\leq\sigma^R\Big\}\\
&\subseteq \Big\{\sup_{t\in[0,T]}|\xb_m^R(t)-\qb(t)|\leq\frac{\xi}{R},\sup_{t\in[0,T]}|\xb_m^R(t)|\ge \frac{R}{N^2}\Big\}\\
&\qquad\qquad \bigcup_{1\le i<j\le N}\Big\{\sup_{t\in[0,T]}|\xb_m^R(t)-\qb(t)|\leq\frac{\xi}{R},\sup_{t\in[0,T]}|x_{i,m}^{R}(t)-x_{j,m}^R(t)|^{-1}\ge \frac{R}{N^2}\Big\}\\
&=\Big\{\sup_{t\in[0,T]}|\xb_m^R(t)-\qb(t)|\leq\frac{\xi}{R},\sup_{t\in[0,T]}|\xb_m^R(t)|\ge \frac{R}{N^2}\Big\}\\
&\qquad\qquad \bigcup_{1\le i<j\le N}\Big\{\sup_{t\in[0,T]}|\xb_m^R(t)-\qb(t)|\leq\frac{\xi}{R},\inf_{t\in[0,T]}|x_{i,m}^{R}(t)-x_{j,m}^R(t)|\le \frac{N^2}{R}\Big\}.
\end{align*}
Since $\xi$ and $N$ are fixed, for $R$ large enough, say, $\frac{R}{N^2}- \frac{\xi}{R}\ge \sqrt{R} $, we have
\begin{align*}
&\Big\{\sup_{t\in[0,T]}|\xb_m^R(t)-\qb(t)|\leq\frac{\xi}{R},\sup_{t\in[0,T]}|\xb_m^R(t)|\ge \frac{R}{N^2}\Big\} \\
&\subseteq  \Big\{\sup_{t\in[0,T]}|\qb(t)|\ge \frac{R}{N^2}- \frac{\xi}{R}\Big\}\subseteq  \Big\{\sup_{t\in[0,T]}|\qb(t)|\ge \sqrt{R}\Big\}\\
&\subseteq \Big\{ \sup_{t\in[0,T]}|\qb(t)|^2-\varepsilon\close\sum_{1\le i<j\le N}\close\log|q_i(t)-q_j(t)| \geq \sqrt{R}\Big\}.
\end{align*}
On the other hand, by triangle inequality,
\begin{align*}
\inf_{t\in[0,T]}|q_{i}(t)-q_{j}(t)| \le 2\sup_{t\in[0,T]}|\xb_m^{R}(t)-\qb(t)|+\inf_{t\in[0,T]}|x_{i,m}^{R}(t)-x_{j,m}^R(t)|,
\end{align*}
implying
\begin{align*}
&\Big\{\sup_{t\in[0,T]}|\xb_m^R(t)-\qb(t)|\leq\frac{\xi}{R},\inf_{t\in[0,T]}|x_{i,m}^{R}(t)-x_{j,m}^R(t)|\le \frac{N^2}{R}\Big\}\\
&\subseteq   \Big\{\inf_{t\in[0,T]}|q_{i}(t)-q_{j}(t)|\le \frac{2\xi+N^2}{R}\Big\} \subseteq   \Big\{\inf_{t\in[0,T]}|q_{i}(t)-q_{j}(t)|\le \frac{1}{\sqrt{R}}\Big\} \\
&= \Big\{-\varepsilon\sup_{t\in[0,T]}\log|q_{i}(t)-q_{j}(t)|\ge\frac{1}{2}\varepsilon\log R\Big\}.
\end{align*}
It follows that for $\varepsilon$ small and $R$ large enough
\begin{align*}
&\Big\{\sup_{t\in[0,T]}|\xb_m^R(t)-\qb(t)|\leq\frac{\xi}{R},\inf_{t\in[0,T]}|x_{i,m}^{R}(t)-x_{j,m}^R(t)|\le \frac{N^2}{R}\Big\}\\
&\subseteq \Big\{ \sup_{t\in[0,T]}|\qb(t)|^2-\varepsilon\close\sum_{1\le i<j\le N}\close\log|q_i(t)-q_j(t)| \geq \frac{1}{2}\varepsilon \log R\Big\}.
\end{align*}
As a consequence, the following holds
\begin{align*}
&\Big\{\sup_{t\in[0,T]}|\xb_m^R(t)-\qb^R(t)|\leq\frac{\xi}{R},\sigma^R_m < T\leq\sigma^R\Big\}\\
&\subseteq \Big\{ \sup_{t\in[0,T]}|\qb(t)|^2-\varepsilon\close\sum_{1\le i<j\le N}\close\log|q_i(t)-q_j(t)| \geq \frac{1}{2}\varepsilon \log R\Big\}.
\end{align*}
We employ \eqref{ineq:GLE:N-particle:m=0:supnorm:beta_1:log} and Markov's inequality to infer
\begin{align} \label{ineq:P(sigma<T):I_1}
I_1&=\P\Big(\sup_{t\in[0,T]}|\xb_m^R(t)-\qb^R(t)|\leq\frac{\xi}{R},\sigma^R_m < T\leq\sigma^R\Big)\le \frac{C(T)}{\varepsilon \log R}.
\end{align}
Turning back to \eqref{ineq:P(sigma<T):I_1+I_2+I_3}, we collect~\eqref{ineq:P(sigma<T):I_3}, \eqref{ineq:P(sigma<T):I_2} and \eqref{ineq:P(sigma<T):I_1} to arrive at the bound
\begin{align} \label{ineq:P(sigma<T)}
&\P\big(\sigma^R\mi\sigma^R_m<T\big)\le  \frac{m}{\xi^4}\cdot C(T,R)+ \frac{C(T)}{\varepsilon \log R}.
\end{align}
We emphasize that in the above estimate $C(T,R)$ and $C(R)$ are independent of $m$.

Now, putting everything together, from \eqref{ineq:P(|x-q|>xi)}, \eqref{ineq:P(|x-q|>xi,sigma>T)} and \eqref{ineq:P(sigma<T)}, we obtain the estimate
\begin{align*}
&\P\Big(\sup_{t\in[0,T]}|\xb_m(t)-\qb(t)|>\xi\Big) \le \frac{m}{\xi^4}\cdot C(T,R)+ \frac{C(T)}{\varepsilon \log R},
\end{align*}
for all $\varepsilon$ small and $R$ large enough. By sending $R$ to infinity and then shrinking $m$ further to zero, this produces the small mass limit\eqref{lim:small-mass:probability}, thereby completing the proof.

\end{proof}

\begin{remark}
For the underdamped Langevin dynamics, it is well known that the small mass limit $m\rightarrow 0$ and the high-friction limit $\gamma\rightarrow + \infty$ 
(under an appropriate time (or noise) rescaling) both lead to the same limiting system, which is the overdamped Langevin dynamics, see e.g. \cite[Section 2.2.4]{LRS2010} and also \cite{duong2017variational} for the Vlasov-Fokker-Planck system. Similarly, one can also derive the underdamped Langevin dynamics from the generalized Langevin dynamics (GLE) under the white-noise limit, by rescaling the friction coefficients $(\lambda_i)$ and the strength of the noises $(\alpha)_i$ appropriately, see e.g.  \cite{ottobre2011asymptotic,nguyen2018small} for a rigorous analysis for the non-interacting GLE and \cite{duong2019mean} for a formal derivation for the interacting-GLE with regular interactions. This white-noise limit is different from the small-mass limit studied in this paper. It would be interesting to study the white-noise limit for the generalized Langevin dynamics with irregular interactions, which we leave for future work.
\end{remark}
\section*{Acknowledgments}
The research of M. H. Duong was supported by EPSRC Grants EP/V038516/1 and  EP/W008041/1. M. H. Duong gratefully acknowledges Yulong Lu for useful discussions in an early stage of the topics of this work. The authors also would like to thank the anonymous reviewers for their valuable comments and suggestions.

\section*{conflict of interest}

The authors have no conflicts of interest to declare that are relevant to the content of this article.

\appendix

\section{Auxiliary estimates on singular potentials}
\label{sec:appendix}

In this section, we collect some useful estimates  on singular potentials, cf. Lemma \ref{lem:|x_i-x_i|} and Lemma \ref{lem:<|x_i-x_j|^s,|x_i-x_ell|^s>} below. In particular, Lemma \ref{lem:|x_i-x_i|} was employed to construct Lyapunov functions in Section \ref{sec:ergodicity}. Lemma \ref{lem:<|x_i-x_j|^s,|x_i-x_ell|^s>} appears in the proof of Lemma \ref{lem:GLE:N-particle:m=0:supnorm} which was invoked to prove the small mass result of Theorem \ref{thm:small-mass:N-particle}.

\begin{lemma} \label{lem:|x_i-x_i|}
For all $s\ge 0$ and $\xb=(x_1,\dots,x_N)\in\D$, the following holds:
\begin{align} \label{ineq:|x_i-x_i|}
\sum_{i=1}^N\bigg\la\sum_{j\neq i}\frac{x_i-x_j}{|x_i-x_j|^{s+1}} ,\sum_{\ell\neq i}\frac{x_i-x_\ell}{|x_i-x_\ell|}\bigg\ra \ge 2\close\sum_{1\leq i< j\leq N}\frac{1}{|x_i-x_j|^{s}}.
\end{align}
\end{lemma}
\begin{proof}
The proof is the same as that of \cite[Lemma 4.2]{lu2019geometric}. See also \cite[Inequality (5.1)]{bolley2018dynamics}.
\end{proof}

\begin{lemma} \label{lem:<|x_i-x_j|^s,|x_i-x_ell|^s>}
For all $\xb=(x_1,\dots,x_N)\in\D$, the followings hold:

\textup{(a)} For all $s\ge 0$,
\begin{align} \label{ineq:|x_i-x_i|:<|x_i-x_j|^s,|x_i-x_ell|^s>}
\sum_{i=1}^N\bigg\la\sum_{j\neq i}\frac{x_i-x_j}{|x_i-x_j|^{s+1}} ,\sum_{\ell\neq i}\frac{x_i-x_\ell}{|x_i-x_\ell|^{s+1}}\bigg\ra \ge  \frac{4}{N(N-1)^2}\close\sum_{1\leq i< j\leq N}\frac{1}{|x_i-x_j|^{2s}}.
\end{align}

\textup{(b)} Furthermore, for $s\in[0,1]$,
\begin{align} \label{ineq:|x_i-x_i|:<|x_i-x_j|^s,|x_i-x_ell|^s>:s<1}
\sum_{i=1}^N\bigg\la\sum_{j\neq i}\frac{x_i-x_j}{|x_i-x_j|^{s+1}} ,\sum_{\ell\neq i}\frac{x_i-x_\ell}{|x_i-x_\ell|^{s+1}}\bigg\ra \ge  2\close\sum_{1\leq i< j\leq N}\frac{1}{|x_i-x_j|^{2s}}.
\end{align}
\end{lemma}
\begin{proof}
(a) With regard to \eqref{ineq:|x_i-x_i|:<|x_i-x_j|^s,|x_i-x_ell|^s>}, we employ Cauchy-Schwarz inequality to infer
\begin{align*}
\bigg|\sum_{i=1}^N\bigg\la\sum_{j\neq i}\frac{x_i-x_j}{|x_i-x_j|^{s+1}} ,\sum_{\ell\neq i}\frac{x_i-x_\ell}{|x_i-x_\ell|}\bigg\ra \bigg|^2
&\le \sum_{i=1}^N\bigg|\sum_{j\neq i}\frac{x_i-x_j}{|x_i-x_j|^{s+1}}\bigg|^2\cdot \sum_{i=1}^N\bigg|\sum_{\ell\neq i}\frac{x_i-x_\ell}{|x_i-x_\ell|}\bigg|^2 \\
&\le \sum_{i=1}^N\bigg|\sum_{j\neq i}\frac{x_i-x_j}{|x_i-x_j|^{s+1}}\bigg|^2\cdot N(N-1)^2.
\end{align*}
In view of Lemma~\ref{lem:|x_i-x_i|}, the following holds
\begin{align*}
\bigg|\sum_{i=1}^N\bigg\la\sum_{j\neq i}\frac{x_i-x_j}{|x_i-x_j|^{s+1}} ,\sum_{\ell\neq i}\frac{x_i-x_\ell}{|x_i-x_\ell|}\bigg\ra \bigg|^2\ge  4\bigg(\sum_{1\leq i< j\leq N}\frac{1}{|x_i-x_j|^{s}}\bigg)^2.
\end{align*}
Therefore, we obtain
\begin{align*}
\sum_{i=1}^N\bigg|\sum_{j\neq i}\frac{x_i-x_j}{|x_i-x_j|^{s+1}}\bigg|^2 \ge \frac{4}{N(N-1)^2}\bigg(\sum_{1\leq i< j\leq N}\frac{1}{|x_i-x_j|^{s}}\bigg)^2.
\end{align*}
This produces~\eqref{ineq:|x_i-x_i|:<|x_i-x_j|^s,|x_i-x_ell|^s>}, thus establishing part (a).

(b) Turning to \eqref{ineq:|x_i-x_i|:<|x_i-x_j|^s,|x_i-x_ell|^s>:s<1}, we will follow the proof of \cite[Lemma 4.2]{lu2019geometric} adapted to our setting. A routine computation gives 
\begin{align*}
&\sum_{i=1}^N\bigg\la\sum_{j\neq i}\frac{x_i-x_j}{|x_i-x_j|^{s+1}} ,\sum_{\ell\neq i}\frac{x_i-x_\ell}{|x_i-x_\ell|^{s+1}}\bigg\ra   \\
&= 2\close \sum_{1\leq i< j\leq N}\frac{1}{|x_i-x_j|^{2s}}+2\close\sum_{1\le i<j<\ell\le N}\bigg[\frac{\la x_i-x_j,x_i-x_\ell\ra}{|x_i-x_j|^{s+1} |x_i-x_\ell|^{s+1}}\\
&\qquad\qquad+\frac{\la x_j-x_i,x_j-x_\ell\ra}{|x_j-x_i|^{s+1} |x_j-x_\ell|^{s+1}}+\frac{\la x_\ell-x_i,x_\ell-x_j\ra}{|x_\ell-x_i|^{s+1} |x_\ell-x_j|^{s+1}}\bigg].
\end{align*}
It suffices to prove that the second sum on the above right-hand side is non negative. To this end, we claim that for $i<j<\ell$,
\begin{align*}
\frac{\la x_i-x_j,x_i-x_\ell\ra}{|x_i-x_j|^{s+1} |x_i-x_\ell|^{s+1}}+\frac{\la x_j-x_i,x_j-x_\ell\ra}{|x_j-x_i|^{s+1} |x_j-x_\ell|^{s+1}}+\frac{\la x_\ell-x_i,x_\ell-x_j\ra}{|x_\ell-x_i|^{s+1} |x_\ell-x_j|^{s+1}}\ge 0.
\end{align*} 
Denote $\theta_i,\theta_j,\theta_\ell$ to be the angles formed by these points and whose vertices are respectively $x_i,x_j,x_\ell$. Observe that the above inequality is equivalent to
\begin{align} \label{ineq:cos(theta_i)|x_j-x_ell|^s}
\cos(\theta_i)|x_j-x_\ell|^{s}+\cos(\theta_j)|x_i-x_\ell|^{s}+\cos(\theta_\ell)|x_i-x_j|^{s}\ge 0.
\end{align}
Now, there are two cases to be considered depending on the triangle formed by $x_i,\,x_j,\,x_\ell$ in $\rbb^d$.

Case 1: the triangle is an acute or right triangle. In this case, it is clear that $\cos(\theta_i), \, \cos(\theta_j)$ and $\cos(\theta_\ell)$ are both non negative. This immediately produces \eqref{ineq:cos(theta_i)|x_j-x_ell|^s}.

Case 2: The triangle is an obtuse triangle. Without loss of generality, suppose that the angle $\theta_j\in (\pi/2,\pi]$, and thus $\theta_i+\theta_\ell\in[0,\pi/2)$. Observe that
\begin{align*}
\min\{\cos(\theta_i),\cos(\theta_\ell)\}\ge \cos(\theta_i+\theta_\ell)=-\cos(\theta_j)\ge 0.
\end{align*}
As a consequence, since $s\in[0,1]$,
\begin{align*}
&\cos(\theta_i)|x_j-x_\ell|^{s}+\cos(\theta_j)|x_i-x_\ell|^{s}+\cos(\theta_\ell)|x_i-x_j|^{s}\\
&\ge |\cos(\theta_j)|\Big(|x_j-x_\ell|^{s}-|x_i-x_\ell|^{s}+|x_i-x_j|^{s} \Big)\ge 0.
\end{align*}
This establishes \eqref{ineq:cos(theta_i)|x_j-x_ell|^s}, thereby finishing the proof of part (b).

\end{proof}

\bibliographystyle{abbrv}
{\footnotesize\bibliography{GLE-bib}}

\end{document}